\documentclass[a4paper,11pt]{article}
\usepackage{amsmath, amsthm, amsfonts, amssymb, mathtools, bbm}
\usepackage{graphicx, psfrag, color}
\usepackage{cite}
\usepackage{hyperref}
\usepackage{stmaryrd}
\usepackage[margin=0.8in]{geometry}
\usepackage{mathtools,xcolor}
\usepackage{sidecap}
\usepackage[font=small,labelfont=bf]{caption}

\newcommand{\ics}{\frac{1}{(2\pi\I)^2}}

\newcommand{\iid}{i.i.d.\,}
\newcommand{\hs}{{\rm hs-stat}}

\newcommand{\Or}{\mathcal{O}}
\newcommand{\Ai}{\mathrm{Ai}}

\newcommand{\Pb}{\mathbb{P}}

\newcommand{\Id}{\mathbbm{1}}
\newcommand{\e}{\varepsilon}
\newcommand{\I}{{\rm i}}

\newcommand{\R}{\mathbb{R}}
\newcommand{\N}{\mathbb{N}}
\newcommand{\Z}{\mathbb{Z}}
\renewcommand{\Re}{\mathrm{Re}\,}

\DeclareMathOperator{\sgn}{sgn}
\DeclareMathOperator*{\pf}{\mathrm{pf}}
\newcommand{\bra}[1]{\left\langle #1 \right|}
\newcommand{\ket}[1]{\left| #1 \right\rangle}
\newcommand{\braket}[2]{\left\langle #1 \left| #2 \right\rangle \right.}
\newcommand{\ketbra}[2]{\left| #1 \right\rangle \left\langle #2 \right|}

\newcommand{\zcd}{\raisebox{-0.5\height}{\,\includegraphics[scale=0.03]{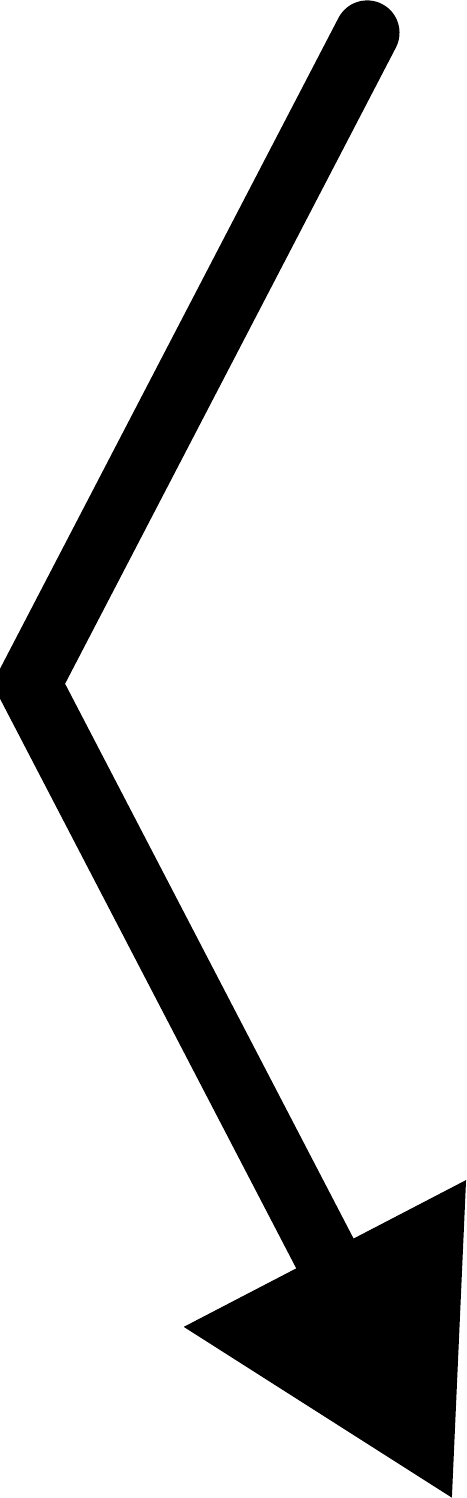}}}
\newcommand{\wcu}{\raisebox{-0.5\height}{\includegraphics[scale=0.03]{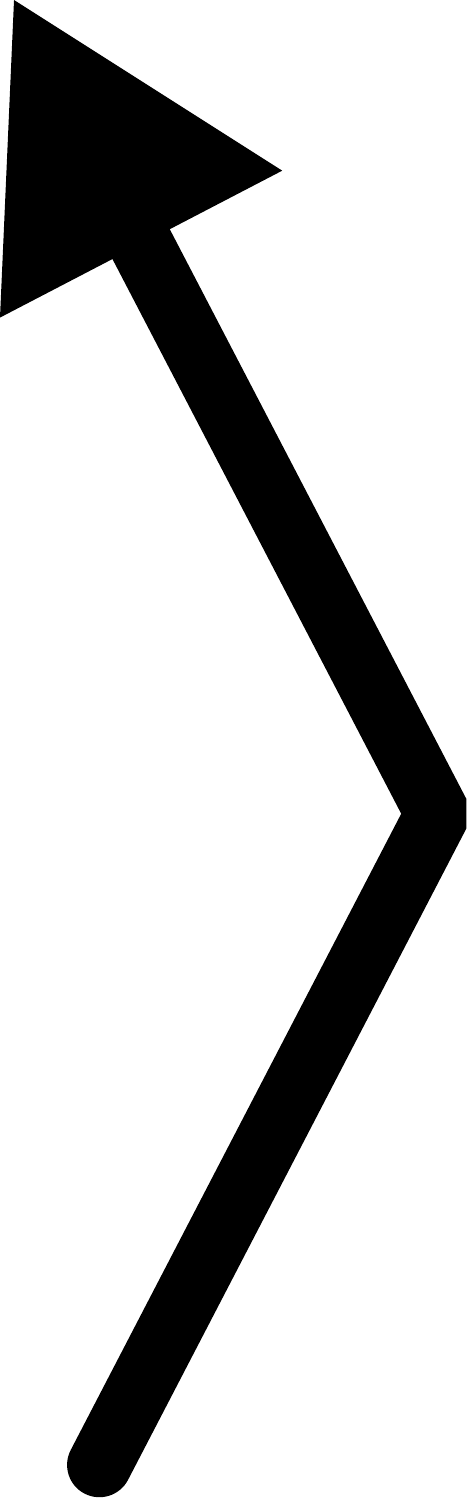}\,}}

\DeclareMathAlphabet{\mathpzc}{OT1}{pzc}{m}{it}

\newtheorem{prop}{Proposition}[section]
\newtheorem{thm}[prop]{Theorem}
\newtheorem{lem}[prop]{Lemma}
\newtheorem{defin}[prop]{Definition}
\newtheorem{cor}[prop]{Corollary}

\newtheorem{cla}[prop]{Claim}

\newtheorem{remark}[prop]{Remark}

\newenvironment{rem}{\begin{remark}\normalfont}{\end{remark}}

\numberwithin{equation}{section}
\allowdisplaybreaks

\title{The half-space Airy stat process}
\author{D. Betea\thanks{Department of Mathematics, KU Leuven, Celestijnenlaan 200b -- box 2400,
B-3001 Leuven, Belgium. E-mail: {\tt dan.betea@gmail.com}}
\and P. L. Ferrari\thanks{Institute for Applied Mathematics, Bonn University, Endenicher Allee 60, 53115 Bonn, Germany. E-mail: {\tt ferrari@uni-bonn.de}}
\and A. Occelli\thanks{Department of Mathematics, Instituto Superior T\'ecnico, Av.~Rovisco Pais 1, 1049-001 Lisbon, Portugal. E-mail: {\tt alessandra.occelli@tecnico.ulisboa.pt}}}
\date{}

\begin{document}
\sloppy
\setcounter{equation}{2}
\maketitle

\begin{abstract}
 We study the multipoint distribution of stationary half-space last passage percolation with exponentially weighted times. We derive both finite-size and asymptotic results for this distribution. In the latter case we observe a new one-parameter process we call \emph{half-space Airy stat}. It is a one-parameter generalization of the Airy stat process of Baik--Ferrari--P\'ech\'e, which is recovered far away from the diagonal. All these results extend the one-point results previously proven by the authors.
\end{abstract}

\tableofcontents

\newpage

\section{Introduction}

\paragraph{Background.} The one-dimensional Kardar--Parisi--Zhang (KPZ) universality class has received a lot of attention in recent years, see e.g.~the surveys and lecture notes~\cite{FS10,Cor11,QS15,BG12,Qua11,Fer10b,Tak16,Zyg18}. A model in this class describes the evolution of a height function $h(x, t)$ (at position $x$ and time $t$) subject to an irreversible stochastic and local microscopic evolution. Macroscopically, the height function evolves according to a certain PDE and this gives a non-random limit shape. Among the models in the KPZ class we list: the KPZ equation itself~\cite{KPZ86}; directed random polymers (the free energy playing the role of the height function); their zero-temperature limits falling in the category called last passage percolation; interacting particle systems like the asymmetric simple exclusion process; the Eden model; and others~\cite{KS92,Mea98}. The last several decades have seen some of these models analyzed for a variety of classes of initial/boundary conditions. It turns out the height function has $\Or(t^{1/3})$ fluctuations ($t$ large) and $\Or(t^{2/3})$ correlation length scales, as predicted in~\cite{FNS77,BKS85}\footnote{Such results hold around smooth limit shapes; hydrodynamical shocks can behave differently and we refer the reader to~\cite{FN13,FN16,FF94b,FGN17,Nej17,FN19,BF20} for some works on shocks.}.

Large time limiting processes of KPZ models usually depend on subclasses of initial conditions. For full-space models ($x\in\R$ for continuous or $x\in\Z$ for discrete models), one encounters the Airy$_2$ process of Pr\"ahofer and Spohn in the case of curved limit shape points~\cite{PS02,Jo03b,BF07}. Its one-point distribution is the GUE Tracy--Widom distribution~\cite{TW94} of random matrix theory. It was discovered in KPZ models first by Baik--Deift--Johansson~\cite{BDJ99} (longest increasing subsequences of random permutations and Hammersley last passage percolation); the same distribution was then shown to hold for a variety of other models in the KPZ class~\cite{Jo00b,SS10,ACQ10,SS10b,FV13,BCF12,Bar14,KPS19}. Beyond models with determinantal structure, the \emph{extended} limit process has been proven to be Airy$_2$ only recently for a stochastic six-vertex model by Dimitrov~\cite{Dim20}. For flat limit shapes and non-random initial conditions, one obtains the Airy$_1$ process, discovered by Sasamoto~\cite{Sas05}, in the large time limit~\cite{Sas05,BFPS06,BFP06}. It has the GOE Tracy--Widom distribution as its one-point distribution~\cite{TW96}. See also~\cite{BR99,BR99b,PS00} for related work.

Stationary initial conditions for full-space also lead to flat limit shapes, but the randomness of the initial condition is relevant. The limit process was obtained by Baik--Ferrari--P\'ech\'e~\cite{BFP09} and called the Airy$_{\rm stat}$ process. It has the Baik--Rains distribution~\cite{BR00} (see~\cite{FS05a} for an alternative formula) as its one-point distribution. See also~\cite{FS05a,BCFV14,IS12,FSW15,Agg16,SI17,SI17b,IMS19,AB19} for related KPZ work and models at stationarity. One obtains stationary models as limits of specific two-sided random initial conditions\footnote{The random initial condition on the two sides is recovered for most of the studied models using boundary sources, by the use of some Burke-type property~\cite{Bur56,DMO05}, as shown for the exclusion process in~\cite{PS02b}. The only exceptions are the works on the stochastic six-vertex model and its partially asymmetric limit in the work by Aggarwal and Borodin~\cite{Agg16,AB19}.}.

The half-space stationary LPP model was introduced and studied by the authors~\cite{BFO20}. A little later Barraquand--Krajenbrink--Le Doussal~\cite{BKLD20} studied the large time half-space KPZ equation started at stationarity. At the boundary of the system (near $0$) and for asymptotically large time, they discovered a special case of the one-point asymptotics of~\cite{BFO20}. Their representation is quite different from ours. See also~\cite{KLD19, NKDT20} for related work on (non-stationary) half-space KPZ. Results for random but not necessarily stationary initial conditions are also known, see~\cite{CLW16,FO17,CFS16,QR16,FV20}.

\paragraph{Main contribution.} In this work we consider a stationary model of last passage percolation in half-space with exponential weights. By half-space we mean the height function $h(x,t)$ is only defined for $x\in\N$ or $x\in\R_+$. This model was introduced in~\cite{BFO20} where we obtained finite-size and asymptotic formulas for the one-point distribution of the last passage time. In the scaling limit we obtained a two-parameter family of probability distributions, where one parameter gives the strength of the weights at the boundary of the system and the second gives the distance from it. It is a half-space analogue and one-parameter generalization of the Baik--Rains distribution~\cite{BR00}. Moreover this distribution converges to the one of Baik--Rains far away from the system boundary.

In this paper we extend the results of~\cite{BFO20} to the multipoint setting. We first obtain the $m$-point joint distribution of last passage times in the half-space stationary model (Theorem~\ref{thm:main_finite}). We further obtain its asymptotics as the size of the system goes to infinity. The resulting process, defined via its finite dimensional distributions, we call the \emph{half-space Airy stat} process, Airy$_\hs$ (Theorem~\ref{thm:main_asymptotics}). It converges to the Airy$_{\rm stat}$ process when moving far away from the boundary of the system (Theorem~\ref{thm:LimitAiryStat}).

The presence of a source at the origin (in the particle system formulation of LPP) is a physical difference going from full- to half-space. This makes the half-space problem richer both mathematically and physically. As already seen from the one-point distribution, the influence of the boundary persists in the limiting process~\cite{BFO20}. We thus obtain a one-parameter dependent process, the parameter being related with the strength of the diagonal in LPP language. By contrast, the full-space Airy$_{\rm stat}$ process is a parameter-free process.

\paragraph{Motivation.} We expect the Airy$_\hs$ process to be universal within the half-space KPZ class to the same extent that the Airy$_{\rm stat}$ is in full-space. In particular, the half-space stationary case has been only recently considered with the one-point distribution analyzed in~\cite{BFO20} and, for a particular case and using a different algebraic approach in~\cite{BKLD20}.

Time-time correlations give a second reason for studying the half-space stationary LPP limit process. In the case of full-space, there has been intense recent activity in the area~\cite{Jo15,Jo18,JR19,BL16,BL17,Liu19,ND17,ND18,NDT17,FS16,FO18}. In particular, the last two authors~\cite{FO18} have shown that the first order correction of the time-time covariance for times macroscopically close to each other is governed by the variance of the Baik--Rains distribution, confirming a prediction of Takeuchi~\cite{Tak13}. The reason is that the system locally converges to equilibrium and the limit process is locally like the stationary one. Thus the Airy$_{\rm stat}$ process plays an important role. For a comparable study in half-space, we would need knowledge of the respective half-space process; it is this process, Airy$_\hs$, we introduce here.

A third reason for our study is that there are considerably fewer results in half-space KPZ models than there are in full-space and the situation is richer in half-space. For a start one needs to prescribe the dynamics at the origin $x=0$ (in the particle system language). A strong influence of the (growth) mechanism at the origin leads to Gaussian fluctuations. A small influence does not effect the asymptotics. Between these two situations there is a critical value of the parameter governing the influence of the origin where it starts becoming relevant. Under a critical scaling one then obtains a family of distributions interpolating between the two extremes. This can be seen in some (non-stationary) half-space LPP models, as well as in some growth models with non-random initial conditions. For example one has a transition between Gaussian fluctuations (at super-criticality and on a different scale) to GOE Tracy--Widom fluctuations (at the critical value) to GSE Tracy--Widom fluctuations (at sub-criticality)~\cite{BR99b,SI03,BBC18,KLD19}. In some models this transition also persists at the extended process level~\cite{SI03,BBCS17,BBNV18,BBCS17b}. Usually these models have a Pfaffian structure, see e.g.~\cite{Ra00, BR01b, RB04, FR07, Gho17, BBCW17, BBNV18, BZ17, BZ19} for further references on this.

\paragraph{Outline.} This paper is organized as follows: we continue this section with some useful notation we use throughout. In Section~\ref{sec:results} we present the stationary half-space LPP model and the main results: in Section~\ref{sec:lpp} we define the model and discuss its connections to TASEP; in Section~\ref{sec:results_finite} we state the main finite-size result on the multipoint joint distribution of stationary half-space LPP (Theorem~\ref{thm:main_finite}), and we prove it in Section~\ref{sec:finite_proof}; in Section~\ref{sec:results_asymptotics} we give the asymptotic result under critical scaling (Theorem~\ref{thm:main_asymptotics}), and we prove it in Section~\ref{sec:proof_asymptotics}; in Section~\ref{sec:resultAiryStat} we show what, moving far away from the origin, the limit process converges to the Airy$_{\rm stat}$ process (Theorem~\ref{thm:LimitAiryStat}), and we prove it in Section~\ref{sec:proof_Airy_stat}. In the last two sections we make efforts to keep everything concise and thus we rely as much as possible on previously done asymptotic analysis in~\cite{BFO20} and~\cite{BFP09}. In Appendix~\ref{sec:pfaff} we discuss basics of Pfaffians. In Appendix~\ref{sec:geom_wts} we state Theorem~\ref{thm:geom_corr_2}, a result on multipoint LPP times with geometric random variable weights which in some sense is the true starting point of all our analysis. In Appendix~\ref{sec:geom_exp_limit} we recover our exponential LPP model of interest from geometric random variables. Finally, in Appendix~\ref{AppWellDef} we prove some odds and ends ensuring that the half-space Airy stat Airy$_\hs$ process is well-defined (and in fact we do so for the Airy stat Airy$_{\rm stat}$ process as well).

\paragraph{Notations.} We use throughout the same notational conventions as in our earlier work~\cite{BFO20}. Notably for complex contours we denote by $\Gamma_{I}$ any simple possibly disconnected counter-clockwise contour around the points in the set $I$. We thus allow for disjoint unions of simple counter-clockwise contours each encircling one point of $I$. We also use the following notation for the usual Airy contours we'll need in the asymptotics: ${}_I\,\raisebox{-0.5\height}{\includegraphics[scale=0.035]{z_cont_down_thick}}\,{}_J$ is a down-oriented contour coming in a straight line from $\exp(\pi i /3)\infty$ to a point on the real line to the right of (all points of) $I$ and to the left of $J$, and continuing in a straight line to $\exp(-\pi i /3)\infty$; and with ${}_{I}\,\raisebox{-0.5\height}{\includegraphics[scale=0.035]{w_cont_up_thick}}\, {}_{J}$ an up-oriented contour from $\exp(-2\pi i /3)\infty$ to $\exp(2\pi i/3)\infty$. Finally, we write $\I \R$ for the up-oriented imaginary axis.

If $A$ is an integral operator with kernel $A(x, y)$ and $f$ a function, we use usual multiplication notation $Af$ to denote $A$ acting on $f$, that is
\begin{equation}
 (Af)(x) = \int_{X} A(x, y) f(y) dy
\end{equation}
with integration over an appropriate space $X$ (usually a semi-infinite interval, the real line, or their discrete analogues).

We use the $\braket{\rm bra}{\, \rm ket}$ notation throughout. If $f, g$ are functions, we denote the scalar product on $L^2((s_1,\infty))$ by
\begin{equation}
 \braket{f}{P_{s_1} g} = \int_{s_1}^\infty f(x) g(x) dx,
\end{equation}
where $P_{s_1}$ is the projector onto $(s_1, \infty)$, while by $\ketbra{f}{g}$ we denote the outer product kernel
\begin{equation}
 \ketbra{f}{g} (x, y) = f(x) g(y).
\end{equation}
We find it useful to extend this notation by putting row/column vectors (and matrices) inside bras and kets. As an example, for vectors and matrices of dimension 2, the quantity
\begin{equation}
\braket{f_1 \quad f_2} { \begin{pmatrix} a_{11} & a_{12} \\ a_{21} & a_{22} \end{pmatrix} \begin{pmatrix} g_1 \\ g_2 \end{pmatrix} } = \braket{f_1}{a_{11} g_1 + a_{12} g_2} + \braket{f_2}{a_{21} g_1 + a_{22} g_2}
\end{equation}
is a sum of four terms (the $a_{ij}:=a_{ij}(x,y)$'s are integral operators) and
\begin{equation}
\ketbra{\begin{array}{c} f_1 \\ f_2 \end{array} } {g_1 \quad g_2} (x, y) = \begin{pmatrix} f_1(x) g_1(y) & f_1(x) g_2(y) \\ f_2(x) g_1(y) & f_2(x) g_2(y) \end{pmatrix}
\end{equation}
is a $2 \times 2$ matrix kernel. We warn the reader that this notation will usually involve $2m \times 1$ vectors of functions and $2m \times 2m$ matrix kernels for some $m \geq 1$.

Letters like $G, K, W, V$ (with possible ornaments and in different fonts) will usually denote $2m \times 2m$ matrix kernels; we also denote by $G_s$ (or $K_s, V_s, \dots$) the $2m \times 2m$ matrix kernel $G_s = P_s G P_s$ where $P_s$ is a certain projector to be defined in the sequel ($P_s = P_s^* = P_s^2$). Note $P_s$ commutes with $G_s$ and thus $P_s$ commutes with its resolvent.

We introduce throughout a variety of kernels and functions that depend on many parameters and live in the finite-size or asymptotic world. To distinguish between them:
\begin{itemize}
 \item in Section~\ref{sec:results_finite} we use a sans-serif $\mathsf{font}$ to denote all our objects that enter the finite-size result;
 \item in Section~\ref{sec:results_asymptotics} we use calligraphic $\mathpzc{fonts}$ and $\mathcal{FONTS}$ to denote all our objects that enter the half-space Airy stat process $\mathcal{A}_{\hs}$;
 \item in Section~\ref{sec:resultAiryStat} we use the notation of Baik--Ferrari--P\'ech\'e for the Airy stat process $\mathcal{A}_{\rm stat}$;
 \item finally and \emph{in Section~\ref{sec:finite_proof} alone} we break from these conventions for technical reasons and we use an up-right sans-serif $\mathsf{font}$ to denote limits of objects as two parameters become one $\beta \to -\alpha$ \emph{if and only if the pre-limit objects depend explicitly on $\beta$}. For example,
 \begin{equation}
 {\mathsf{K}} = \lim_{\beta \to -\alpha} {K},
 \end{equation}
 where $K$ is assumed to depend on both parameters $\alpha$ and $\beta$.
\end{itemize}

\paragraph{Acknowledgements.} We are grateful to G. Barraquand, A. Krajenbrink, and P. Le Doussal for their suggestions for improving and their feedback regarding a preprint of this paper. 

This work was started when all three authors were at the University of Bonn; as such it is supported by the German Research Foundation through the Collaborative Research Center 1060 ``The Mathematics of Emergent Effects'', project B04, and by the Deutsche Forschungs-gemeinschaft (DFG, German Research Foundation) under Germany's Excellence Strategy - GZ 2047/1, Projekt ID 390685813. The article was finished when: D.B. was at KU Leuven (Leuven, Belgium), supported by FWO Flanders project EOS 30889451; and A.O. was at Instituto Superior T\'ecnico (Lisbon, Portugal), supported by the HyLEF ERC starting grant 2016.

\section{Model and main results} \label{sec:results}

\subsection{Stationary last passage percolation}\label{sec:lpp}
This section is expository and introduces all terminology we use then in the presentation of our main results. We discuss last passage percolation, its relation to TASEP, the half-space model, and its stationary version in the sense of~\cite{BCS06}.

\paragraph{Last passage percolation.} We first introduce generic last passage percolation (LPP). We start with independent random variables $\{\omega_{i,j},\,i,j\in\Z\}$. By an \emph{up-right path} $\pi$ on $\Z^2$ from points $A$ to $E$ we mean a sequence of points $(A=\pi(0),\pi(1),\ldots,\pi(n)=E)$ in $\Z^2$ such that \mbox{$\pi(k+1)-\pi(k)\in \{(0,1),(1,0)\}$}. We call $n = \ell(\pi)$ \emph{the length} of $\pi$. Given a set $S_A$ of points and $E$ an endpoint, we define the \emph{last passage time} $L_{S_A\to E}$ by
\begin{equation}\label{eq:3.2}
L_{S_A\to E}=\max_{\begin{subarray}{c}\pi:A\to E\\A\in S_A\end{subarray}} \sum_{(i,j)\in \pi} \omega_{i,j}.
\end{equation}
The maximizing path for the last passage time $L_{S_A\to E}$ is a.s.~unique if the random variables $\omega_{i,j}$ are continuous.

For our purposes we only consider exponentially distributed random variables. LPP with exponential random variables is connected to the well-known and studied \emph{totally asymmetric simple exclusion process} (TASEP) in continuous time. We describe this next.

\paragraph{TASEP and LPP.} TASEP is an interacting particle system on the integers having state space $\Omega=\{0,1\}^\Z$. If $\eta = (\eta_j)_{j \in \Z} \in \Omega$ is a configuration, $\eta_j$ is the \emph{occupation variable} at site $j$. It is $1$ if site $j$ is occupied by a particle and $0$ otherwise. The Markov generator $\cal L$ for TASEP is~\cite{Li99}
\begin{equation}\label{eq:1.1}
{\cal L}f(\eta)=\sum_{j\in\Z}\eta_j(1-\eta_{j+1})\big(f(\eta^{j,j+1})-f(\eta)\big),
\end{equation}
where $f$ is any function depending only on finitely many sites, and $\eta^{j,j+1}$ is the configuration obtained from $\eta$ by interchanging occupation at sites $j$ and $j+1$. We observe that the ordering of particles is preserved for TASEP: if we initially order particles from right to left as
\begin{equation}
 \ldots < x_2(0) < x_1(0) < 0 \leq x_0(0)< x_{-1}(0)< \ldots\, ,
\end{equation}
then we also have $x_{n+1}(t)<x_n(t)$, $n\in\Z$ for all times $t\geq 0$.

Let us now explain the link between TASEP and LPP. Consider $\omega_{i,j}$ to be the waiting time of particle $j$ to jump from site $i-j-1$ to site $i-j$. Then the $\omega_{i,j}$'s are exponential random variables. Moreover, setting $S_A=\{(u,k)\in\Z^2: u=k+x_k(0), k\in\Z\}$ we have the relations
\begin{equation}\label{eq:2.4}
\Pb\left(L_{{S_A}\to (m,n)}\leq t\right)=\Pb\left(x_n(t)\geq m-n\right)=\Pb\left(h_t(m-n)\geq m+n\right),
\end{equation}
where we denoted by $h_t$ the standard \emph{height function} for TASEP at time $t$.

We further explain the terminology \emph{full-space} and \emph{half-space} for LPP. It comes from TASEP, more precisely from the fact that the height function and particles live on $\Z$ for full-space and $\Z_+$ for half-space. For half-space we have $x = m-n \geq 0$ in~\eqref{eq:2.4}; it means that LPP random variables $\omega_{i,j}$ are restricted to $\{(m,n)| m\geq n\}$. The reader can equivalently imagine the other random variables are set to $0$.

\paragraph{Invariant measures for TASEP.} Liggett studied invariant measures for TASEP in full-space~\cite{Lig75,Lig77}. He first considered a finite system to achieve his result. From this finite system we can obtain the half-space model as a simple limiting case. Thus for half-space TASEP (defined on $\N$) where particles can enter from a reservoir at the origin with rate $\lambda\in[0,1]$, Liggett showed that the stationary measure with particle density $\rho=\lambda$ on $\N$ is a product measure. It is for this reason that, in the half-space LPP analogue, we consider diagonal weights which are exponential random variables of parameter $\rho$. In our case and below we set $\rho=\tfrac12+\alpha$. The random initial condition in $\N$ can be replaced, by Burke's theorem~\cite{Bur56}, with a first row of weights which are exponential random variables of parameter $1-\rho$.

The stationary measures for half-space TASEP are not unique. There are other examples which \emph{are not product measures}. See Theorem~1.8 of~\cite{Lig75}. A matrix-product ansatz representation is given in~\cite[Theorem 3.2]{Gro04}. The mapping from TASEP to LPP implies, in such cases, that the $\omega_{i,j}$'s are not independent random variables anymore. Our techniques from this note cannot handle such cases.

\paragraph{Stationary LPP.} We now discuss the object of our study, stationary half-space last passage percolation with exponential weights. It was introduced in~\cite{BFO20} and we follow the exposition therein. We focus on half-space LPP model. By half-space we mean the set ${\cal D}=\{(i,j)\in\Z^2 | 1\leq j \leq i\}$. On it we place non-negative random variables $\{\omega_{i,j}\}_{(i,j)\in {\cal D}}$. The half-space LPP time to the point $(n,m)$ (for $m\leq n$) denoted by $L_{n,m}$ is given by
\begin{equation}
L_{n,m}=\max_{\pi:(1,1)\to (n,m)} \sum_{(i,j)\in\pi} \omega_{i,j}
\end{equation}
with the maximum taken over all up-right paths in ${\cal D}$ from $(1, 1)$ to $(m,n)$.

Our interest is the stationary version of this model. To define it, first write $\mathrm{Exp}(a)$ for an exponential random variable with parameter $a>0$. The stationary setting is then given by
\begin{equation} \label{eq:stat_wts}
 \omega_{i, j} = \begin{cases}
 \mathrm{Exp}\left( \frac{1}{2} + \alpha \right), & i=j>1,\\
 \mathrm{Exp}\left( \frac{1}{2} - \alpha \right), & j=1, i>1, \\
 0, & \textrm{if}\ i=j=1, \\
 \mathrm{Exp}(1), &\textrm{otherwise}
 \end{cases}
\end{equation}
with $\alpha \in (-1/2, 1/2)$ is a fixed parameter. An illustration is given in Figure~\ref{fig:stat_lpp}.

\begin{figure}[!t]
 \centering
 \includegraphics[height=5.5cm]{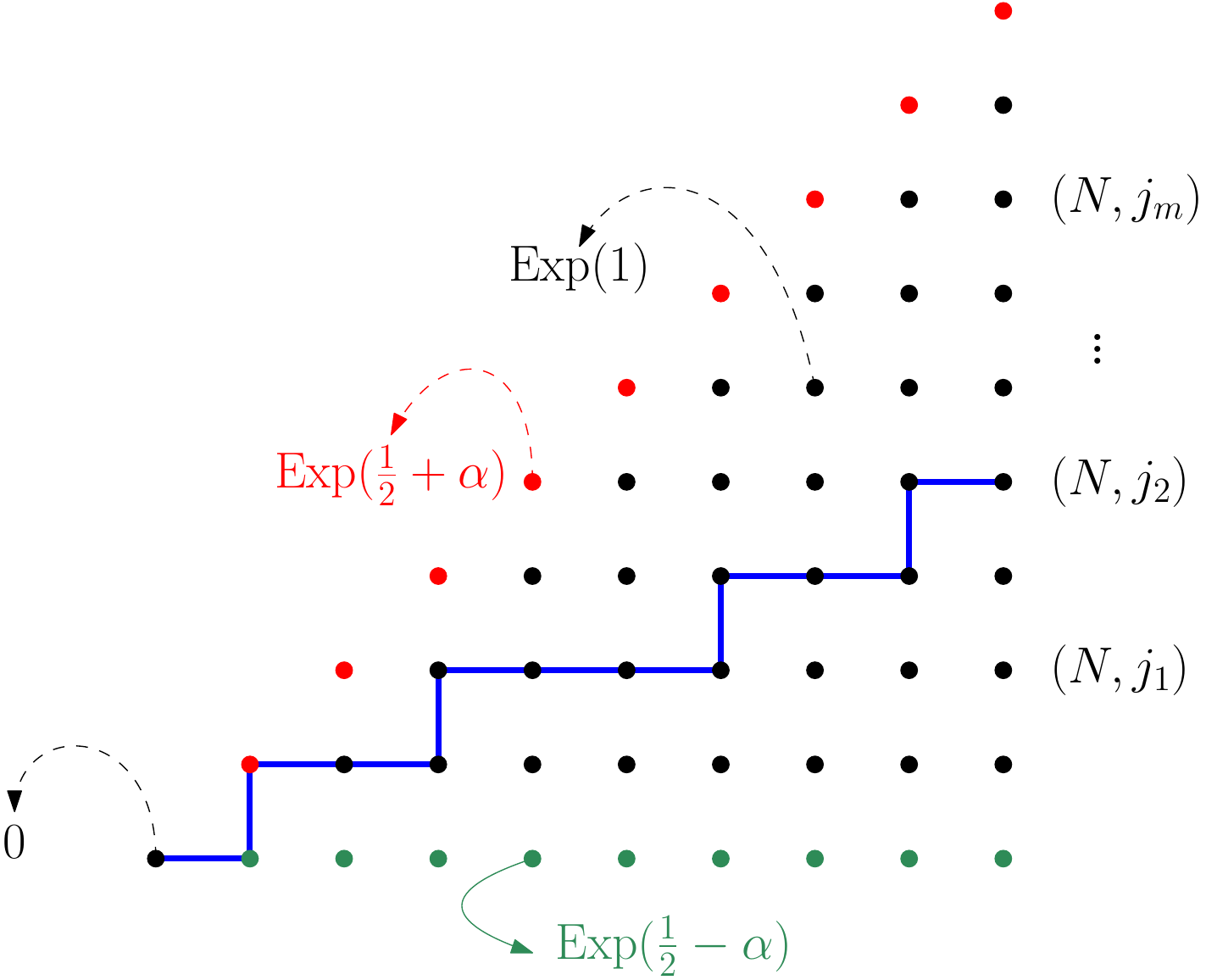}
 \caption{A possible LPP path (polymer) from $(1, 1)$ to $(N, j_2) = (10, 5)$ in the stationary case. The dots are independent random variables: deterministically 0 at the origin, $\mathrm{Exp}(\tfrac12+\alpha)$ (respectively $\mathrm{Exp}(\tfrac12 - \alpha)$) on the rest of the diagonal (respectively the bottom line), and $\mathrm{Exp}(1)$ everywhere else in the bulk.}
 \label{fig:stat_lpp}
\end{figure}

We will be interested in joint $m$-point distributions of LPP times of the ending points $(N, j_1), \dots, (N, j_m)$. Here $m \geq 1$ and $N$ are integers, and
\begin{equation}
 1 \leq j_1 < \dots < j_m \leq N
\end{equation}
are $m$ points we think of as ordinates of our LPP endpoints (the abscissa is always $N$, the ``large parameter'' eventually). The one-point distribution was considered in~\cite{BFO20}.

This model is stationary in the sense of~\cite{BCS06}, i.e.~it has stationary increments as stated in the following result. The proof was given in~\cite{BFO20}. It is a simple extension of the original proof of~\cite{BCS06}.
\newpage
\begin{lem}(Half-space version of~\cite[Lemma 4.1 and 4.2]{BCS06})\label{lem:stationarity}
 Fix any $i, j\geq1$ with $i<j$. The following three random variables are jointly independent and distributed as follows:
 \begin{itemize}
 \item the increment along the horizontal direction $H_{i+1, j+1} = L_{i+1,j+1}-L_{i,j+1}$ is an $\textup{Exp} \left(\tfrac12-\alpha \right)$ random variable;
 \item the increment along the vertical direction $V_{i+1, j+1} = L_{i+1, j+1}-L_{i+1, j}$ is an $\textup{Exp} \left( \tfrac12+\alpha \right)$ random variable;
 \item finally, the minimum of the horizontal and vertical increments at a vertex $(i,j)$, defined by $X_{i, j} = \min(H_{i+1, j}, V_{i, j+1}) = \min( L_{i+1, j} - L_{i, j}, L_{i, j+1} - L_{i, j} )$ is an $\textup{Exp(1)}$ random variable.
 \end{itemize}
 Moreover, fix $\Pi$ any down-right path in half-space from the diagonal to the horizontal axis. Then increments along $\Pi$ are jointly independent, the horizontal ones being $\textup{Exp} \left(\tfrac12-\alpha \right)$ and the vertical ones $\textup{Exp} \left(\tfrac12 + \alpha \right)$ random variables. Moreover, they are independent from the \iid random variables $X_{i, j} = \min( L_{i+1, j} - L_{i, j}, L_{i, j+1} - L_{i, j} )$ for $(i, j)$ any point strictly below and to the left of $\Pi$.
 \end{lem}

We make the following important remark, which motivates the next section and all of Section~\ref{sec:finite_proof}.

\begin{rem}
We do not have good (tractable) formulas to study the joint statistics of $L_{N,j_k}$ (for $1 \leq k \leq m$) in the stationary exponential case. However we can recover such statistics using a two-step procedure first carried out in half-space by the authors in~\cite{BFO20} for the one-point distribution and before in~\cite{BFP09} for the full-space multipoint distribution. We first consider a related \emph{integrable} LPP model having two parameters $\alpha$ and $\beta$. For this latter the joint distribution is a Fredholm Pfaffian with a $2m \times 2m$ matrix kernel. We recover our original stationary model by then performing a standard shift argument and a lengthy analytic continuation procedure that will allow us to take the limit $\beta \to -\alpha$ (which leads to stationarity). This last step is far from trivial and will occupy all of Section~\ref{sec:finite_proof}. It is different from both the one-point and multipoint full-space stationary cases of~\cite{FS05a} and~\cite{BFP09} respectively; it is similar to but a multi-point extension of the case considered in~\cite{BFO20} by the authors. Whenever computations are similar to those of~\cite{BFO20} we indicate it.
\end{rem}

\subsection{Finite-time multipoint distribution for stationary LPP} \label{sec:results_finite}

Throughout this section we fix $m \geq 1$, $m$ positive \emph{ordered} integers $1 \leq j_1 < j_2 < \cdots < j_m \leq N$ and real numbers $s_\ell \geq 0,\, 1 \leq \ell \leq m$.

We first start by fixing some auxiliary functions we will need below. Let
\begin{equation}
 \Phi(x, z) = e^{-xz} \phi(z), \quad \phi(z) = \left[ \frac{\frac12+z}{\frac12-z} \right]^{N-1}
\end{equation}
and define
\begin{equation} \label{eq:f_def_sec2}
 \begin{split}
 \mathsf{f}_{+,\alpha}^{j} (x) &= \Phi(x, \alpha)\left( \tfrac12-\alpha \right)^{N-j}, \\
 \mathsf{e}^{\alpha,\, j} (x) &= - \oint\limits_{\Gamma_{1/2, \alpha}} \frac{dz}{2 \pi \I} \frac{\Phi(x, z)}{\Phi(x, \alpha)} \left[ \frac{\tfrac12 + \alpha} {\tfrac12 + z} \right]^{N-j} \frac{1}{(z - \alpha)^2}
 \end{split}
\end{equation}
and furthermore
\begin{equation}
\begin{aligned} \label{eq:g_def_sec2}
 & \mathsf{g}_1^{j}(x) &=\oint\limits_{\Gamma_{1/2}} \frac{dz}{2\pi\I} \Phi (x,z)\left( \tfrac12-z \right)^{N-j}\frac{z+\alpha}{2z}, \quad \mathsf{g}_2^{j}(x) &=\!\!\! \oint\limits_{\Gamma_{1/2,\alpha}} \!\!\!\frac{dz}{2\pi\I} \frac{\Phi(x,z)}{\left( \tfrac12+z \right)^{N-j}}\frac{1}{z-\alpha}, \\
 & \mathsf{g}_3^{j}(x) &= \oint\limits_{\Gamma_{1/2}}\frac{dz}{2\pi\I}\Phi(x,z)\left(\tfrac12-z\right)^{N-j}\frac{1}{z+\alpha},\quad
 \mathsf{g}_4^{j}(x) &= \!\!\! \oint\limits_{\Gamma_{1/2,\pm\alpha}} \!\!\! \frac{dz}{2\pi\I}\frac{\Phi(x,z)}{\left(\frac12-z\right)^{N-j_k}}\frac{2z}{(z-\alpha)(z+\alpha)^2}.
 \end{aligned}
\end{equation}

We further define the following \emph{extended} anti-symmetric kernel (the reason for the indices will become clear soon):
\begin{equation}
 \begin{split}
 \breve{\mathsf{K}}^{j j'}_{11}(x,y) =& - \oint\limits_{\Gamma_{1/2}} \frac{dz}{2\pi\I} \oint\limits_{\Gamma_{-1/2}} \frac{dw}{2\pi\I}\frac{\Phi(x,z)}{\Phi(y,w)} (\tfrac12-z)^{N-j} (\tfrac12+w)^{N-j'} \frac{(z-\alpha)(w+\alpha)(z+w)}{4zw(z-w)}, \\
 \breve{\mathsf{K}}^{j j'}_{12} (x,y) =& - \!\!\! \oint\limits_{\Gamma_{1/2}} \frac{dz}{2\pi\I} \!\!\! \oint\limits_{\Gamma_{-1/2,\alpha}} \!\!\! \frac{dw}{2\pi\I}\frac{\Phi(x,z)}{\Phi(y,w)} \frac{(\tfrac12-z)^{N-j}} {(\tfrac12-w)^{N-j'}} \frac{z-\alpha}{w-\alpha}\frac{z+w}{2z(z-w)} + \mathsf{V}^{j j'} (x, y) \\
 =& -\breve{\mathsf{K}}^{j' j}_{21} (y,x),\\
 \breve{\mathsf{K}}^{j j'}_{22}(x,y) =&\ \mathsf{E}^{j j'}(x,y)+\oint \frac{dz}{2\pi\I}\oint\frac{dw}{2\pi\I} \frac{\Phi(x,z)}{\Phi(y,w)} \frac{1}{(\tfrac12+z)^{N-j} (\tfrac12-w)^{N-j'}} \frac{1}{z-w}\left(\frac{1}{z+\alpha}+\frac{1}{w-\alpha}\right),
 \end{split}
 \end{equation}
 where the integration contours for $\breve{\mathsf{K}}^{j j'}_{22}$ are $\Gamma_{1/2,-\alpha}\times\Gamma_{-1/2}$ for the term with $1/(z+\alpha)$ and $\Gamma_{1/2}\times\Gamma_{-1/2,\alpha}$ for the term with $1/(w-\alpha)$, where
 \begin{equation}
 \begin{split}
 \mathsf{V}^{j j'} (x, y) &= - \Id_{[j > j']} \int_{\I \R} \frac{dz}{2 \pi \I} \frac{e^{-(x-y)z}} {(\tfrac12-z)^{j-j'}} = \Id_{[j > j']} \Id_{[x \geq y]} \oint\limits_{\Gamma_{1/2}} \frac{dz}{2 \pi \I} \frac{e^{-(x-y)z}} {(\tfrac12-z)^{j - j'}} \\
 &=- \Id_{[j > j']} \Id_{[x \geq y]} \frac{(x-y)^{j-j'-1} e^{-\frac{x-y}{2}}} {(j-j'-1)!},
 \end{split}
 \end{equation}
 and we have denoted $\mathsf{E}^{j j'} (x, y) = \mathsf{E}^{j j'}_0 (x, y) + \mathsf{E}^{j j'}_1 (x, y)$. The latter are defined by
 \begin{equation}
 \begin{split}
 \mathsf{E}^{j j'}_0 (x, y) &= - \frac{e^{(x-y) \alpha}} {(\tfrac12 - \alpha)^{N-j} (\tfrac12 + \alpha)^{N-j'}} \quad \text{if } x \geq y,\\
 \quad \mathsf{E}^{j j'}_1 (x, y) &= -\oint\limits_{\Gamma_{1/2}} \frac{dz}{2 \pi \I} \frac{2 z e^{-(x-y)z}} {\left[\tfrac12+z\right]^{N-j} \left[\tfrac12-z\right]^{N-j'}} \frac{1}{z^2-\alpha^2} \quad \text{if } x \geq y,
 \end{split}
 \end{equation}
and extended for $x < y$ by the anti-symmetry property of $\mathsf{E}$, namely
\begin{equation}
 \mathsf{E}^{j j'} (x, y) = - \mathsf{E}^{j' j} (y, x).
\end{equation}
Let us also define two more operators of the same sort by
\begin{equation}
 \begin{split}
 \widetilde{\mathsf{K}}_{12}^{j j'}(x,y) &=-\oint\limits_{\Gamma_{1/2}} \frac{dz}{2\pi\I} \oint\limits_{\Gamma_{-1/2}}\frac{dw}{2\pi\I}\frac{\Phi(x,z)}{\Phi(y,w)} \frac{(\tfrac12-z)^{N-j}} {(\tfrac12-w)^{N-j'}} \frac{z-\alpha}{w-\alpha} \frac{z+w}{2z(z-w)},\\
\widetilde{\mathsf{K}}^{j j'}_{22}(x,y) &=\oint\limits_{\Gamma_{1/2,- \alpha}} \frac{dz}{2\pi\I}\oint\limits_{\Gamma_{-1/2}}\frac{dw}{2\pi\I} \frac{\Phi(x,z)}{\Phi(y,w)} \frac{1}{(\tfrac12+z)^{N-j} (\tfrac12-w)^{N-j'}} \frac{1}{(z+\alpha)(w-\alpha)} \frac{z+w}{z-w}.
 \end{split}
\end{equation}

Define the $2m \times 1$ column vector $\mathsf{U}$ (below $-\alpha < \eta <1/2$) by
\begin{equation}
 \begin{split}
 \mathsf{U}_{2k}(y) &= - \mathsf{f}^{j_1}_{+, -\alpha} (s_1) \int_{\I \R +\eta} \frac{dz}{2 \pi \I} \frac{e^{-(y-s_1)z}}{(\tfrac12-z)^{j_k-j_1} (z+\alpha)} \\
 &= \Id_{[y \geq s_1]} \mathsf{f}^{j_1}_{+, -\alpha} (s_1) \oint\limits_{\Gamma_{1/2}} \frac{dz}{2 \pi \I} \frac{e^{-(y-s_1)z}}{(\tfrac12-z)^{j_k-j_1} (z+\alpha)} - \Id_{[y < s_1]} \mathsf{f}^{j_k}_{+, -\alpha} (y), \quad 1 < k \leq m, \\
 \mathsf{U}_{\ell}(y) &= 0, \quad \text{otherwise}
 \end{split}
\end{equation}
and the row and column vectors (still with $2m$ components)
\begin{equation}
 \begin{split}
 \mathsf{Y} &= \left( - \mathsf{g}_1^{j_1}, \mathsf{g}_2^{j_1}, \cdots -\mathsf{g}_1^{j_m}, \mathsf{g}_2^{j_m} \right), \\
 \mathsf{Q} &= \left(\mathsf{h}^{\alpha,\, j_1}_1, \mathsf{h}^{\alpha,\, j_1}_2, \dots, \mathsf{h}^{\alpha,\, j_m}_1, \mathsf{h}^{\alpha,\, j_m}_2 \right)^t,
 \end{split}
\end{equation}
where
\begin{equation}
 \begin{split}
 \mathsf{h}^{\alpha,\, j}_1(y) &= -\int_{s_1}^\infty \widetilde{\mathsf{K}}^{j j_1}_{22}(y,v) \mathsf{f}_{+, -\alpha}^{j_1} (v) dv - \int_{s_1}^\infty \mathsf{E}^{j j_1}_1 (y,v) \mathsf{f}_{+, -\alpha}^{j_1} (v) dv + \mathsf{g}_4^{j}(y) + \mathsf{j}^{\alpha,\, j} (s_1, y),\\
 \mathsf{h}^{\alpha,\, j}_2(y)&= \int_{s_1}^\infty \widetilde{\mathsf{K}}^{j j_1}_{12}(y,v) \mathsf{f}_{+, -\alpha}^{j_1}(v) dv + \mathsf{g}_3^{j}(y), \\
 \mathsf{j}^{\alpha,\, j}(s, y) &= \Id_{[y > s]} \frac{e^{\alpha s} \phi(-\alpha)} { \big( \frac 12-\alpha \big)^{N-j_1}} \left[\big(\tfrac 12+\alpha\big)^{j-j_1} \frac{\sinh \alpha (y-s)}{\alpha} + \big(\tfrac 12-\alpha\big)^{j-j_1}e^{\alpha (y-s)}(y-s)\right].
 \end{split}
\end{equation}

Consider the $2m \times 2m$ diagonal projector operator
\begin{equation}
 P_s = {\rm diag} \{ \Id_{[x > s_1]}, \Id_{[x > s_1]}, \dots, \Id_{[x > s_m]}, \Id_{[x > s_m]} \}.
\end{equation}

Let $\breve{\mathsf{K}} (x, y)$ be the $2m \times 2m$ matrix kernel having $2 \times 2$ block at position $(k, \ell)$ given by
\begin{equation}
 \begin{pmatrix}
 \breve{\mathsf{K}}^{j_k j_\ell}_{11} (x, y) & \breve{\mathsf{K}}^{j_k j_\ell}_{12} (x, y) \\ \breve{\mathsf{K}}^{j_k j_\ell}_{21} (x, y) & \breve{\mathsf{K}}^{j_k j_\ell}_{22} (x, y)
 \end{pmatrix},
 \quad 1 \leq k, \ell \leq m
\end{equation}
 and let $J$ denote the $2m \times 2m$ matrix with $2 \times 2$ block $\begin{pmatrix} 0 & 1 \\ -1 & 0 \end{pmatrix}$ on the diagonal and $0$'s elsewhere\footnote{Otherwise said, $J = I_m \otimes \begin{pmatrix} 0 & 1 \\ -1 & 0 \end{pmatrix}$ with $I_m$ the $m \times m$ identity matrix; note also that $J^{-1} = -J$.}. Write
\begin{equation}
 \breve{\mathsf{K}}_s = P_s \breve{\mathsf{K}} P_s.
\end{equation}

The following is our main finite-size joint distributions of the stationary half-space LPP model.

\begin{thm} \label{thm:main_finite}
 Fix $m\in\N$ and $\alpha \in (-1/2, 1/2)$ a real parameter. Let $1 \leq j_1 < \dots <j_m \leq N$ be different endpoints (times), $s_k \in\R_+$ ($1 \leq k \leq m$) and consider the stationary last passage times $L_{N, j_k}$. We have:
 \begin{equation}\label{eq:main_finite}
 \begin{split}
 \Pb \left( \bigcap_{k=1}^m \left\{ L_{N,j_k} \leq s_k \right\} \right) = \sum_{k=1}^m \partial_{s_k} \left\{ \pf (J - \breve{\mathsf{K}}_s) \cdot \left[ \mathsf{e}^{\alpha,\, j_1} (s_1) - \braket{P_s \mathsf{Y}} {(\Id-J^{-1} \breve{\mathsf{K}}_s )^{-1} P_s (\mathsf{Q}-\mathsf{U}) } \right] \right\}.
 \end{split}
 \end{equation}
\end{thm}

\begin{rem}
 The case $m=1$ recovers the main finite result of~\cite{BFO20}, for which case the vector $\mathsf{U}$ disappears.
\end{rem}

\subsection{Asymptotic multipoint distribution for stationary LPP} \label{sec:results_asymptotics}

In this section we present our main asymptotic result. We first discuss critical scaling exponents, then define the necessary ingredients for giving the result. Finally we state Theorem~\ref{thm:main_asymptotics} and a few of its consequences.

\paragraph{Scaling limit.} The same heuristics described in~\cite{BFO20} (beginning of Section 2.3) applies here as well for determining how we scale the various parameters. We will consider critical scaling here, namely $\alpha$ of order $N^{-1/3}$ close to $0$ and all the $j$'s of the form $N - \Or(N^{2/3})$. More precisely, for
\begin{equation}
\alpha = \delta 2^{-4/3} N^{-1/3},\quad N-j=u 2^{5/3} N^{2/3}
\end{equation}
with $\delta \in \R, u>0$ fixed, the macroscopic approximation of LPP times is given by (see~\cite[Section~2.3]{BFO20})
\begin{equation}
L_{N, j}\simeq 4 N -2 u 2^{5/3} N^{2/3} +\delta(2u+\delta) 2^{4/3} N^{1/3}.
\end{equation}

\begin{rem} \label{rem:simple_scaling}
 We will not include the $\Or(N^{1/3})$ contribution of $\delta(2u+\delta) 2^{4/3} N^{1/3}$ in the limit result we give below. The reason is that many formulas are more compact without it. Thus we look at the scaling
 \begin{equation}\label{eq:ScalingS}
 s=4 N - 2 u 2^{5/3}N^{2/3} + S\, 2^{4/3}N^{1/3}.
 \end{equation}
 This term does however need to be accounted for when one takes various limits. For instance in Section~\ref{sec:resultAiryStat} it will be reintroduced in the Airy$_{\rm stat}$ limit $\delta \to -\infty$, i.e.~in such a limit we'll have to substitute $S$ by $S+\delta(2u+\delta)$.
\end{rem}

\paragraph{Definition of the main ingredients.} Throughout this section we fix $m \geq 1$, $m$ \emph{ordered} non-negative real numbers integers $u_1 > u_2 > \cdots > u_m\geq 0$ and $m$ real numbers $S_k$, $1 \leq k \leq m$. We use generic $u, v$ to denote one of the $m$ $u$'s and generic $S$ to denote one of the $m$ $S$'s whenever needed.

In order to state the main result we have to define its various components, functions and kernels we'll need in its statement. Define the functions
\begin{equation}\label{eq2.26}
\begin{aligned}
\mathpzc{f}^{-\delta,\, u}(X)&=e^{-\frac{\delta^3}{3} - \delta^2 u + \delta X}, \\
\mathpzc{e}^{\delta,\, u} (S)&= -\int\limits_{\zcd\, {}_\delta} \frac{d \zeta}{2\pi\I} \frac{e^{\frac{\zeta^3}{3} + \zeta^2 u - \zeta S} }{e^{\frac{\delta^3}{3} + \delta^2 u - \delta S}} \frac{1}{(\zeta - \delta)^2},
\end{aligned}
\end{equation}
as well as
\begin{alignat}{2}
\mathpzc{g}_1^{\delta,\, u} (X) &= \int\limits_{{}_0\zcd} \frac{d \zeta}{2\pi\I} e^{\frac{\zeta^3}{3} - \zeta^2 u - \zeta X} \frac{\zeta + \delta}{2 \zeta},
\quad & \mathpzc{g}_2^{\delta,\, u} (X) &= \int\limits_{\zcd\, {}_\delta} \frac{d \zeta}{2\pi\I} e^{\frac{\zeta^3}{3} + \zeta^2 u - \zeta X} \frac{1}{\zeta - \delta}, \nonumber \\
\mathpzc{g}_3^{\delta,\, u} (X) &= \int\limits_{{}_{-\delta}\zcd} \frac{d \zeta}{2\pi\I} e^{\frac{\zeta^3}{3} - \zeta^2 u - \zeta X} \frac{1}{\zeta+\delta},
\quad & \mathpzc{g}_4^{\delta,\, u} (X) &= \int\limits_{\zcd\, {}_{\pm\delta}} \frac{d \zeta}{2\pi\I} e^{\frac{\zeta^3}{3} + \zeta^2 u - \zeta X} \frac{2\zeta}{(\zeta-\delta)(\zeta + \delta)^2}.
\end{alignat}

We define the following \emph{anti-symmetric} extended Airy-like kernel:
\begin{equation}\label{eq:227}
 \begin{aligned}
 \breve{\mathcal{A}}^{u v}_{11} (X, Y) &= - \int\limits_{{}_{0}\zcd } \frac{d \zeta}{2\pi\I} \int\limits_{\wcu\, {}_{0,\zeta}} \frac{d \omega}{2\pi\I}\frac{ e^{\frac{\zeta^3}{3} - \zeta^2 u - \zeta X} } { e^{\frac{\omega^3}{3} + \omega^2 v - \omega Y} } (\zeta - \delta)(\omega + \delta) \frac{\zeta + \omega}{4 \zeta \omega (\zeta - \omega)},\\
 \breve{\mathcal{A}}^{u v}_{12} (X, Y) &= -\int\limits_{{}_{0}{\zcd}} \frac{d \zeta}{2\pi\I} \int\limits_{{}_{\delta}\wcu\, {}_\zeta} \frac{d \omega}{2\pi\I} \frac{ e^{\frac{\zeta^3}{3} - \zeta^2 u - \zeta X} }{ e^{\frac{\omega^3}{3} - \omega^2 v - \omega Y} } \frac{\zeta-\delta}{\omega-\delta} \frac{\zeta+\omega}{2 \zeta (\zeta-\omega)} + \mathcal{V}^{u v} (X, Y)\\
 &= -\breve{\mathcal{A}}^{v u}_{21} (Y, X),\\
 \breve{\mathcal{A}}^{u v}_{22} (X, Y) &= \mathcal{E}^{u v}(X,Y) + \int \frac{d \zeta}{2\pi\I} \int \frac{d \omega}{2\pi\I} \frac{ e^{\frac{\zeta^3}{3} + \zeta^2 u - \zeta X} }{ e^{\frac{\omega^3}{3} - \omega^2 v - \omega Y} } \frac{1}{\zeta - \omega} \left(\frac{1}{ \zeta + \delta}+\frac{1}{\omega-\delta}\right),
 \end{aligned}
 \end{equation}
 where in $\breve{\mathcal{A}}^{u v}_{22}$ the integration contours for $(\zeta,\omega)$ are $\zcd\, {}_{-\delta}\times \wcu\, {}_{\zeta}$ for the term $1/(\zeta+\delta)$, and $\zcd\, {}\times{}_{\delta}\wcu\, {}_\zeta$ for the term $1/(\omega-\delta)$. We have denoted
\begin{equation}
 \mathcal{V}^{u v} (X, Y) = -\Id_{[u < v]} \int\limits_{\I \R} \frac{d \zeta}{2\pi\I} e^{- \zeta^2 (u - v) - \zeta (X-Y)},
\end{equation}
and $\mathcal{E}^{u v} (X, Y) = \mathcal{E}_0^{u v} (X, Y) + \mathcal{E}_1^{u v} (X, Y)$ with
\begin{equation}
\mathcal{E}_0^{u v} (X, Y) = - e^{\delta (X - Y) + \delta^2 (u + v)}, \quad \mathcal{E}_1^{u v} (X, Y) = - \int\limits_{{}_{\pm\delta}\zcd} \frac{d \zeta}{2\pi\I} e^{-\zeta (X - Y) + \zeta^2 (u + v)} \frac{2\zeta}{\zeta^2-\delta^2}, \quad \text{if } X \geq Y.
\end{equation}
The definition for $X<Y$ comes from the anti-symmetry property of $\mathcal{E}$, namely
\begin{equation}
 \mathcal{E}^{u v} (X, Y) = -\mathcal{E}^{v u} (Y, X).
\end{equation}

\noindent Let us also set
\begin{equation}
\begin{aligned}
\widetilde{\mathcal{A}}^{uv}_{12} (X, Y) &= -\int\limits_{{}_{0}{\zcd }} \frac{d \zeta}{2\pi\I} \int\limits_{\wcu\, {}_{\delta,\zeta}} \frac{d \omega}{2\pi\I} \frac{ e^{\frac{\zeta^3}{3} - \zeta^2 u - \zeta X} }{ e^{\frac{\omega^3}{3} - \omega^2 v -\omega Y} } \frac{\zeta-\delta}{\omega-\delta} \frac{\zeta+\omega}{2 \zeta (\zeta-\omega)},\\
\widetilde{\mathcal{A}}^{uv}_{22} (X, Y) &= \int\limits_{\zcd\, {}_{-\delta}} \frac{d \zeta}{2\pi\I} \int\limits_{\wcu\, {}_{\delta,\zeta}} \frac{d \omega}{2\pi\I} \frac{ e^{\frac{\zeta^3}{3} + \zeta^2 u - \zeta X} } { e^{\frac{\omega^3}{3} - \omega^2 v - \omega Y} } \frac{1} {(\zeta + \delta)(\omega - \delta)} \frac{\zeta + \omega}{\zeta - \omega}.
\end{aligned}
\end{equation}

\noindent Define the $2m \times 1$ column vector (below $-\delta < \eta$):
\begin{equation}
 \begin{split}
 \mathcal{U}_{2 k} (Y) &= -\mathpzc{f}^{-\delta,\, u_1} (S_1) \int_{\I \R+\eta} \frac{d \zeta}{2 \pi \I} \frac{e^{-\zeta^2 (u_k - u_1) - \zeta(Y-S_1)}} {\zeta+\delta}, \quad 1 < k \leq m,\\
 \mathcal{U}_{\ell}(Y) &= 0, \quad \text{otherwise}.
 \end{split}
\end{equation}
Further define the row and respectively column vectors:
\begin{equation}\label{eq2.24}
 \begin{split}
 \mathcal{Y} &= (-\mathpzc{g}^{\delta,\, u_1}_1, \mathpzc{g}^{\delta,\, u_1}_2, \dots, -\mathpzc{g}^{\delta,\, u_m}_1, \mathpzc{g}^{\delta,\, u_m}_2 ), \\
 \mathcal{Q} &= (\mathpzc{h}^{\delta,\, u_1}_1, \mathpzc{h}^{\delta,\, u_1}_2, \dots, \mathpzc{h}^{\delta,\, u_m}_1, \mathpzc{h}^{\delta,\, u_m}_2 )^t,
 \end{split}
\end{equation}
where
\begin{equation}
\begin{aligned}
\mathpzc{h}^{\delta,\, u}_1(Y) &= -\int_{S_1}^\infty dV \widetilde{\mathcal{A}}^{u u_1}_{22}(Y, V) \mathpzc{f}^{-\delta,\, u_1}(V) - \int_{S_1}^\infty dV \mathcal{E}^{u u_1}_1 (Y,V) \mathpzc{f}^{-\delta,\, u_1}(V) + \mathpzc{g}_4^{\delta,\, u}(Y) + \mathpzc{j}^{\delta,\, u} (S_1, Y),\\
\mathpzc{h}^{\delta,\, u}_2(Y)&= \int_{S_1}^\infty dV \widetilde{\mathcal{A}}^{u u_1}_{12}(Y,V) \mathpzc{f}^{-\delta,\, u_1}(V) + \mathpzc{g}_3^{\delta,\, u}(Y), \\
\mathpzc{j}^{\delta,\, u}(S, Y)&= \Id_{[Y > S]} \mathpzc{f}^{-\delta,\, -u} (S) \left[ \frac{\sinh \delta (Y-S)}{\delta} + e^{\delta (Y-S)} (Y-S) \right].
\end{aligned}
\end{equation}

Consider the $2m \times 2m$ projector operator
\begin{equation}
 P_S = {\rm diag} \{ \Id_{[X > S_1]}, \Id_{[X > S_1]}, \dots, \Id_{[X > S_m]}, \Id_{[X > S_m]} \}.
\end{equation}
Finally let $\breve{\mathcal{A}} (X, Y)$ be the $2m \times 2m$ matrix kernel having $2 \times 2$ block at position $(k, \ell)$ given by
\begin{equation}
 \begin{pmatrix}
 \breve{\mathcal{A}}^{u_k u_\ell}_{11} (X, Y) & \breve{\mathcal{A}}^{u_k u_\ell}_{12} (X, Y) \\ \breve{\mathcal{A}}^{u_k u_\ell}_{21} (X, Y) & \breve{\mathcal{A}}^{u_k u_\ell}_{22} (X, Y)
 \end{pmatrix},
 \quad 1 \leq k, \ell \leq m,
\end{equation}
and let $J$ denote the $2m \times 2m$ matrix with $2 \times 2$ block $\begin{pmatrix} 0 & 1 \\ -1 & 0 \end{pmatrix}$ on the diagonal and $0$'s elsewhere. We also write
\begin{equation}
 \breve{\mathcal{A}}_S = P_S \breve{\mathcal{A}} P_S.
\end{equation}

The following is our main asymptotic result.

\begin{thm} \label{thm:main_asymptotics}
 Let $m \geq 1$ be an integer and $\delta \in \R$ be a parameter. Fix $m$ ordered positive real numbers $u_1 > u_2 > \dots > u_m \geq 0$ (thought of as times of a stochastic process) and $m$ real numbers $S_k$, $1 \leq k \leq m$. Consider the stationary last passage times $L_{N, j_k}$ ($1 \leq j_1 < \dots < j_m \leq N$) in the following $N \to \infty$ limit:
 \begin{equation}
 N - j_k = u_k 2^{5/3} N^{2/3}, \quad \alpha = \delta 2^{-4/3} N^{-1/3}.
 \end{equation}
 We have that
 \begin{equation} \label{eq:main_asymptotics}
 \begin{split}
 \lim_{N \to \infty} &\Pb \left( \bigcap_{k=1}^m \left\{ \frac{L_{N,j_k} - 4N + 4 u_k (2N)^{2/3}}{2^{4/3}N^{1/3}} \leq S_k \right\} \right) \\
  &=\sum_{k=1}^m \partial_{S_k} \left\{ \pf (J - \breve{\mathcal{A}}_S) \cdot \left[ \mathpzc{e}^{\delta,\, u_1} (S_1) - \braket{P_S \mathcal{Y}} {(\Id-J^{-1} \breve{\mathcal{A}}_S )^{-1} P_S (\mathcal{Q}-\mathcal{U}) } \right] \right\}.
 \end{split}
 \end{equation}
\end{thm}

\begin{rem}
 The case $m=1$ recovers the main asymptotic result of~\cite{BFO20}; the vector $\mathcal{U}$ disappears in that case.
\end{rem}

Let us give a name to the process with joint distribution given by the right-hand side of~\eqref{eq:main_asymptotics}. In Appendix~\ref{AppWellDef} we show this process is indeed well-defined, thus validating the definition.

\begin{defin} \label{def:Airy_half_stat_def}
 We define the \emph{half-space Airy stationary process}, denoted by $\mathcal{A}_\hs^\delta$, via its finite dimensional distributions, by
 \begin{equation} \label{eq:Airy_half_stat_def}
 \begin{split}
 &\Pb \left( \bigcap_{k=1}^m \left\{ \mathcal{A}_\hs^\delta(u_k) \leq S_k \right\} \right) \\
  & \quad \quad \quad =\sum_{k=1}^m \partial_{S_k} \left\{ \pf (J - \breve{\mathcal{A}}_S) \cdot \left[ \mathpzc{e}^{\delta,\, u_1} (S_1) - \braket{P_S \mathcal{Y}} {(\Id-J^{-1} \breve{\mathcal{A}}_S )^{-1} P_S (\mathcal{Q}-\mathcal{U}) } \right] \right\},
 \end{split}
 \end{equation}
 where $\delta \in \R$ is a fixed parameter, $m \geq 1$ is an integer, $u_1 > u_2 >\cdots > u_m \geq 0$ and $S_1, \dots, S_m \in \R$.
\end{defin}

\subsection{Limit to the Airy$_{\rm stat}$ process}\label{sec:resultAiryStat}

In order to define the Airy$_{\rm stat}$ process, we need to introduce a few objects following the conventions of Baik--Ferrari--P\'ech\'e~\cite{BFP09}. In~\cite{BFP09} the functions and kernels were given in terms of integrals of exponentials and Airy functions. We will show the equality between these formulas and the ones in~\cite{BFP09} at the end of Section~\ref{sec:proof_Airy_stat}. Define
\begin{equation}\label{eqDefinFctAiryStat}
\begin{aligned}
{\cal R}=&-e^{-\frac23 \tau_1^3-\tau_1 s_1} \int\limits_{\zcd {}_{-\tau_1}} \frac{dz}{2\pi\I}\frac{e^{\frac{z^3}{3}-z(s_1+\tau_1^2)}}{(z+\tau_1)^2},\\
\Psi^k(x)=& \int\limits_{\zcd\, {}_{-\tau_k}} \frac{dz}{2\pi\I} e^{\frac{z^3}{3} - z(x+\tau_k^2)} \frac{1}{z+\tau_k},\\
\Phi^k(y)=& \int\limits_{{}_{\tau_k}\zcd} \frac{dz}{2\pi\I} e^{\frac{z^3}{3}-z(y+\tau_k^2)} \frac{1}{z-\tau_k}
+\Id_{[\tau_k>\tau_1]}e^{-\frac23 \tau_k^3-\tau_k y} \int\limits_{\I\R+\eta}\frac{dz}{2\pi\I} \frac{e^{(\tau_k-\tau_1) z^2-z(y-s_1)}}{z}\\
&+e^{-\frac23\tau_1^3-\tau_1 s_1}\int\limits_{{\zcd}} \frac{dz}{2\pi\I} \int\limits_{\wcu{}_{z-\tau_k+\tau_1,\tau_1}} \frac{dw}{2\pi\I} \frac{ e^{\frac{z^3}{3}-z(y+\tau_k^2)}}{ e^{\frac{w^3}{3} - w (s_1+\tau_1^2)}} \frac{1}{(z-w-\tau_k+\tau_1)(w-\tau_1)},
\end{aligned}
\end{equation}
where $\eta>0$. Furthermore, define the extended Airy kernel with entries shifted by $\tau_i^2$ by
\begin{equation}
\widehat K_{\rm Ai}^{i,j}(x,y)=-\Id_{[\tau_i<\tau_j]} \frac{e^{-\frac23\tau_j^3-\tau_j x}}{e^{-\frac23\tau_i^3-\tau_i y}} \frac{e^{-(x-y)^2/(4(\tau_j-\tau_i))}}{\sqrt{4\pi(\tau_j-\tau_i)}} -\int\limits_{{\zcd}} \frac{dz}{2\pi\I} \int\limits_{\wcu{}_{z-\tau_j+\tau_i}} \frac{dw}{2\pi\I} \frac{ e^{\frac{z^3}{3}-z(x+\tau_j^2)}}{ e^{\frac{w^3}{3} - w (y+\tau_i^2)}} \frac{1}{(z-w-\tau_j+\tau_i)}.
\end{equation}

Now we define the Airy$_{\rm stat}$ process by giving its finite-dimensional distributions.
\begin{defin}\label{defAiryStat}
Fix any $m\in \N$, real numbers $\tau_1 < \tau_2 < \ldots < \tau_m$ and $s_1,\ldots,s_m\in\R$. Then the stochastic process ${\rm Airy}_{\rm stat}$, denoted ${\cal A}_{\rm stat}$, is defined by its joint distributions given by
\begin{equation}\label{eqAiryStat}
\Pb\left(\bigcap_{k=1}^m \{{\cal A}_{\rm stat}(\tau_k)\leq s_k\}\right) = \sum_{k=1}^m \partial_{s_k} \left[g_m(\tau,s) \det(\Id-P_s \widehat K_{\rm Ai} P_s)_{L^2(\{1,\ldots,m\}\times\R)} \right]
\end{equation}
with
\begin{equation}
 \begin{split}
 g_m(\tau,s) & = {\cal R}-\sum_{i,j=1}^m\int_{s_i}^\infty dx \int_{s_j}^\infty dy \Psi^i(x) [(\Id-P_s \widehat K_{\rm Ai} P_s)^{-1}]^{i,j}(x,y) \Phi^j(y) \\
 & = {\cal R} - \braket{\Psi} {(\Id-P_s \widehat K_{\rm Ai} P_s)^{-1} \Phi}.
 \end{split}
\end{equation}
\end{defin}

Our last result is the convergence of the Airy$_\hs^\delta$ process to the Airy$_{\rm stat}$ process as $\delta\to -\infty$. In this limit we consider positions around $-\delta$, thus moving away from the origin.
\begin{thm}\label{thm:LimitAiryStat}
Fix $m\geq 1$ an integer. Let $S_k=s_k+\delta(2u_k+\delta)$ and $u_k=-\tau_k-\delta$ for fixed real numbers $\tau_1 < \dots < \tau_m$ (times) and $s_i$ ($1 \leq i \leq m$). Then we have
\begin{equation}
\lim_{\delta\to -\infty} \Pb \left( \bigcap_{k=1}^m \left\{ \mathcal{A}_\hs^\delta(u_k) \leq S_k \right\} \right) = \Pb\left(\bigcap_{k=1}^m \{{\cal A}_{\rm stat}(\tau_k)\leq s_k\}\right).
\end{equation}
\end{thm}

\section{Finite-size analysis: proof of Theorem~\ref{thm:main_finite}} \label{sec:finite_proof}

\subsection{The integrable model} \label{sec:int}

Our starting point in this section is the modified last passage model with weights
\begin{equation} \label{eq:int_wts}
 \tilde\omega_{i, j} = \begin{cases}
 \mathrm{Exp}\left( \frac{1}{2} + \alpha \right), & i=j>1,\\
 \mathrm{Exp}\left( \frac{1}{2} + \beta \right), & j=1, i>1, \\
 \mathrm{Exp}\left( \alpha + \beta \right), & i=j=1, \\
 \mathrm{Exp}(1), &\text{otherwise}.
 \end{cases}
\end{equation}
Here $\alpha, \beta \in (-1/2, 1/2)$ are parameters satisfying $\alpha + \beta > 0$, though in the early stages of the analysis we'll only consider the case $\beta > 0$. See Figure~\ref{fig:exp_lpp} for an illustration of the geometry and weights.

\begin{figure}[t!]
 \centering
 \includegraphics[height=6cm]{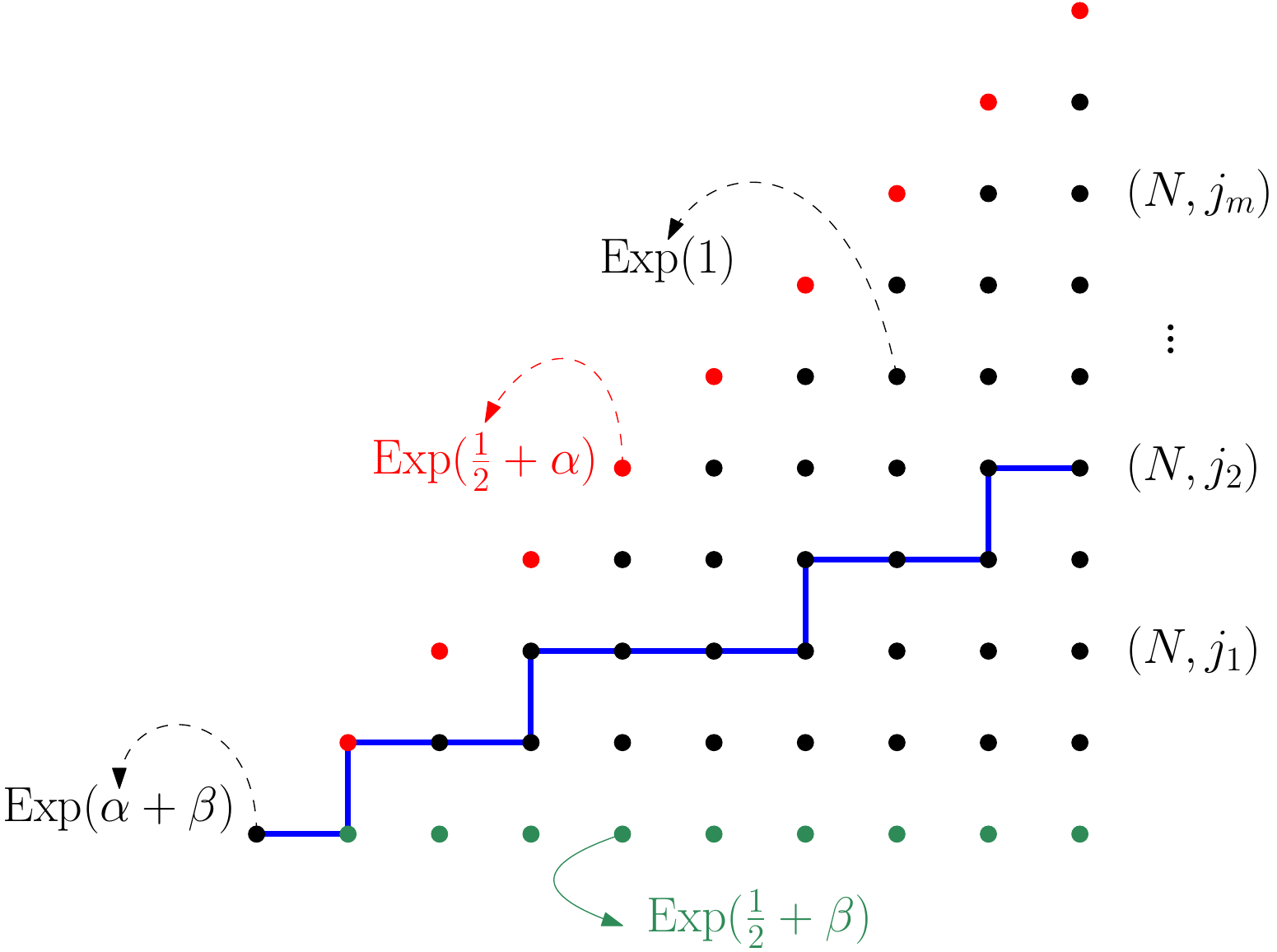}
 \caption{A possible LPP path (polymer) from $(1, 1)$ to $(N, j_2)$ for $(N, j_2) = (10, 5)$ in the integrable case. The dots are independent exponential random variables: $\mathrm{Exp}(\alpha + \beta)$ at the origin, $\mathrm{Exp}(\tfrac12+\alpha)$ (respectively $\mathrm{Exp}(\tfrac12 + \beta)$) on rest of the diagonal (respectively the bottom line), and $\mathrm{Exp}(1)$ everywhere else in the bulk.}
 \label{fig:exp_lpp}
\end{figure}

Let $L^{\rm pf}_{N, j_\ell}$ be the last passage time from $(1, 1)$ to $(N, j_\ell)$ for $1 \leq \ell \leq m$ in this model. We order the $j$'s as $1 \leq j_1 < \dots < j_m \leq N$.

If $\beta>0$ the joint distribution of the $L^{\rm pf}$'s is given by a Fredholm Pfaffian; this explains the superscript ``$\rm pf$''. We prove this in Appendix~\ref{sec:geom_exp_limit} as an exponential limit of a widely studied model with geometric random variables given in Appendix~\ref{sec:geom_wts}. See Appendix~\ref{sec:pfaff} for more on Fredholm Pfaffians; see also~\cite{BBCS17} for a proof of a similar result (the case $\beta = 1/2$) which can be adapted for our purposes with some effort.

\begin{thm} \label{thm:exp_corr}
Let $\beta\in (0,1/2)$ and $\alpha \in (-1/2,1/2)$. Take $1 \leq j_1 < j_2 < \dots < j_m \leq N$ and $s_\ell \in \R_+$ for $1 \leq \ell \leq m$. Let $X = \bigcup_{\ell=1}^m \{ \ell \} \times (s_\ell, \infty)$. Then
 \begin{equation}
 \Pb \left( \bigcap_{\ell=1}^m \{ L^{\rm pf}_{N, j_\ell} \leq s_\ell \} \right) = \pf (J - K)_{L^2(X)},
 \end{equation}
where $K(j, x; j', x')$ is the following $2 \times 2$ matrix kernel:
 \begin{equation} \label{eq:kernel}
 \begin{split}
 K_{11}(j, x; j', x') =& - \! \oint \! \frac{dz}{2\pi\I} \! \oint \! \frac{dw}{2\pi\I}\frac{\Phi(x,z)}{\Phi(x',w)} (\tfrac12-z)^{N-j} (\tfrac12+w)^{N-j'} \\
 & \qquad \qquad \qquad \qquad \times \frac{(z+\beta)(w-\beta)}{(z-\beta)(w+\beta)}\frac{(z+\alpha)(w-\alpha)(z+w)}{4zw(z-w)}, \\
 K_{12}(j, x; j', x') =& - \! \oint \! \frac{dz}{2\pi\I} \! \oint \! \frac{dw}{2\pi\I}\frac{\Phi(x,z)}{\Phi(x',w)} \frac{(\tfrac12-z)^{N-j}} {(\tfrac12-w)^{N-j'}} \frac{z+\alpha}{w+\alpha}\frac{z+\beta}{z-\beta}\frac{w-\beta}{w+\beta}\frac{z+w}{2z(z-w)} \\
 & + V (j, x; j', x')\\
 =& -K_{21} (j',x';j,x),\\
 K_{22}(j, x; j', x') =& \ \tilde \varepsilon(j, x; j', x') + \! \oint \! \frac{dz}{2\pi\I} \! \oint \! \frac{dw}{2\pi\I} \frac{\Phi(x,z)}{\Phi(x',w)} \frac{1} {(\tfrac12+z)^{N-j} (\tfrac12-w)^{N-j'}} \\
 & \qquad \qquad \qquad \qquad \qquad \qquad \quad \times \frac{1}{(z-\alpha)(w+\alpha)} \frac{z+\beta}{z-\beta}\frac{w-\beta}{w+\beta}\frac{z+w}{z-w}.
 \end{split}
 \end{equation}
The contours of integration for the double integrals are unions of the following ones:
\begin{itemize}
\item for $11$ entries: $(z,w) \in \Gamma_{1/2}\times\Gamma_{-1/2, -\beta}$ and $(z,w) \in \Gamma_{\beta}\times\Gamma_{-1/2}$;
\item for $12$ entries: $(z,w) \in \Gamma_{1/2}\times\Gamma_{-1/2,-\alpha,-\beta}$ and $(z,w) \in \Gamma_{\beta} \times \Gamma_{-1/2,-\alpha}$;
\item for $22$ entries: $(z,w) \in \Gamma_{1/2,\alpha,\beta} \times \Gamma_{-1/2}$ and $(z,w) \in \Gamma_{1/2,\beta} \times \Gamma_{-\alpha}$ and $(z,w) \in \Gamma_{1/2,\alpha}\times\Gamma_{-\beta}$.
\end{itemize}
We are using the following notation:
\begin{equation} \label{eq:phi}
 \Phi(x, z) = e^{-xz} \phi (z) \quad \textrm{with } \phi(z) = \left[ \frac{ \tfrac12 + z } { \tfrac12-z }\right]^{N-1},
 \end{equation}
 \begin{equation}\label{eq:V}
 \begin{split}
 V (j, x; j', x') &= - \Id_{[j > j']} \int_{\I \R} \frac{dz}{2 \pi \I} \frac{e^{-(x-x')z}} {(\tfrac12-z)^{j-j'}} = - \Id_{[j > j']} \Id_{[x \geq x']} \frac{(x-x')^{j-j'-1} e^{-\frac{x-x'}{2}}} {(j-j'-1)!},
 \end{split}
 \end{equation}
 \begin{equation}\label{eq:epsilon}
 \begin{split}
 \tilde{\varepsilon} (j, x; j', x') &=-\oint\limits_{\Gamma_{1/2}} \frac{dz}{2 \pi \I} \frac{2 z e^{-(x-x')z}} {\left[\tfrac12+z\right]^{N-j} \left[\tfrac12-z\right]^{N-j'}} \frac{1}{z^2-\alpha^2} - \frac{e^{-(x-x')\alpha}} {\left[\tfrac12 + \alpha\right]^{N-j} \left[\tfrac12 - \alpha\right]^{N-j'}}, \quad \text{if } x \geq x'\\
 &=\tilde{\varepsilon}_1 (j, x; j', x') + \tilde{\varepsilon}_2 (j, x; j', x')
 \end{split}
\end{equation}
with $\tilde{\varepsilon}$ anti-symmetric (which covers the case $x < x'$)
\begin{equation}
 \tilde{\varepsilon} (j, x; j', x') = -\tilde{\varepsilon} (j', x'; j, x).
\end{equation}
\end{thm}

\begin{rem}
 The two equivalent formulas for $V$ follow as limits $q \to 1$ from the two equivalent formulas for $V^{\rm geo}$ from Appendix~\ref{sec:geom_wts} (notably equation~\eqref{eq:V_E_geo}); alternatively, one can close the contour $\I \R$ at $\pm \infty$ (depending on $\sgn(x-x')=\pm 1$) and pick up the residue at $1/2$ (which also gives the indicator in $x,x'$); yet a third way is to see that the first formula for $V$ is a Fourier transform (put $z = \I \tau$) of the function $(\tfrac12-\I \tau)^{-(j-j')}$, and computing this explicitly yields the second formula.
\end{rem}

\begin{rem}
 The contents of Remark~\ref{rem:ext_pfaffian} applies provided we switch from counting to Lebesgue measure. Thus on one hand we can write the $m$-point distribution as
 \begin{equation}
 \Pb \left( \bigcap_{\ell=1}^m \{ L^{\rm pf}_{N, j_\ell} \leq s_\ell \} \right) = \pf (J - P_s K P_s)_{L^2( \{1, 2, \dots, m\} \times \R)},
 \end{equation}
 where $P_s(k, x) = \Id_{[x > s_k]}$. The expansion of $\pf (J + \lambda P_s K P_s)_{L^2( \{1, 2, \dots, m\} \times \R)}$ ($\lambda = -1$ for us) is then
 \begin{equation}
 \pf (J + \lambda P_s K P_s)_{L^2( \{1, 2, \dots, m\} \times \R)} = \sum_{n=0}^{\infty} \frac{\lambda^n}{n!} \sum_{i_1, \dots, i_n = 1}^m \int_{I_{i_1} \times \cdots \times I_{i_n}} \pf [K^{(n)} (j_{i_a}, x_a; j_{i_b}, x_b)]_{1 \leq a, b \leq n} \prod_{a=1}^n d x_a,
 \end{equation}
 where for brevity $I_k = (s_k, \infty)$ and $[K^{(n)} (j_{i_a}, x_a; j_{i_b}, x_b)]_{1 \leq a, b \leq n}$ is the skew-symmetric $2n \times 2n$ matrix with $2 \times 2$ block at $(a, b)$ ($1 \leq a, b \leq n$) given by the matrix kernel $K (j_{i_a}, x_a; j_{i_b}, x_b)$.

On the other hand we can write the same distribution as
 \begin{equation} \label{eq:exp_2m_ker}
 \begin{split}
 \Pb \left( \bigcap_{\ell=1}^{m} \{ L_{N, j_\ell} \leq s_\ell \} \right) &= \pf (J^{(m)}-K^{(m)})_{L^2((s_1, \infty)) \oplus \cdots \oplus L^2((s_m, \infty))} \\
 &= \pf (J^{(m)} - P^{(m)}_s K^{(m)} P^{(m)}_s)_{L^2(\R) \oplus \cdots \oplus L^2(\R)},
 \end{split}
 \end{equation}
where $J^{(m)}$ is the $2m \times 2m$ anti-symmetric matrix having just the blocks $J = \left( \begin{smallmatrix} 0 & 1 \\ -1 & 0 \end{smallmatrix} \right)$ on the diagonal, $K^{(m)}(x, x')$ is the $2m \times 2m$ matrix kernel whose $(a, b)$ $2 \times 2$ block/component ($1 \leq a, b \leq m$) is the $2 \times 2$ matrix kernel $K(j_a, x; j_b, x')$ from the above theorem, and $P^{(m)}_s$ is the $2m \times 2m$ diagonal matrix $\mathrm{diag} (\chi_1, \chi_1, \dots, \chi_m, \chi_m)$ with $\chi_\ell$ the characteristic function of $(s_\ell, \infty)$.

Doing the latter enables us to do useful computations on the $2m \times 2m$ matrix kernel $K^{(m)}$. Moreover, we will drop the superscript ${}^{(m)}$ and just use $K$ for the $2m \times 2m$ matrix kernel hereinafter.
\end{rem}

\begin{rem}
 We will use the following trivial identity many times throughout, and so we make note of it here:
 \begin{equation}\label{eq:phi_symm}
 \phi(-z)=\phi(z)^{-1},\quad \Phi(x, -z)=\Phi(x, z)^{-1}.
 \end{equation}
\end{rem}

\subsection{From integrable to stationary} \label{sec:int_to_stat}

\subsubsection{Shift argument} \label{sec:shift}

For recovering the desired stationary distribution from the Pfaffian one, we follow a strategy easy to explain: we first remove $\tilde \omega_{1,1}$, the undesired random variable at the origin, and then we take the $\beta \to -\alpha$ limit. The first step is achieved by a standard shift argument, used already in the full-space stationary LPP problem~\cite{BR00, FS05a, BFP09, SI04}\footnote{Baik--Rains~\cite{BR00} treat the Poisson case instead of the exponential one but the shift argument is similar.} and also in the half-space case by the authors~\cite[Lemma 3.3]{BFO20}.

We recall that $L^{\rm pf}_{N,j_\ell}$ denotes the LPP time for the random variables $\tilde{\omega}_{i,j}$ of~\eqref{eq:int_wts}. Denote by $\widetilde L_{N,j_\ell} = L^{\rm pf}_{N,j_\ell} - \tilde\omega_{1,1}$ and recall that $L_{N,j_\ell}$ is the $\beta\to-\alpha$ limit of $\widetilde L_{N,j_\ell}$. The shift argument is captured by the following lemma.

\begin{lem} \label{lem:shift}
Let $\alpha,\beta\in(-1/2,1/2)$ with $\alpha+\beta>0$ and let $1 \leq j_1 < \cdots < j_m \leq N$. Define
\begin{equation}
 \begin{split}
 \widetilde \Pb(s_1,\dots,s_m) &= \Pb\left(\bigcap_{\ell=1}^m \{\widetilde L_{N,j_\ell}\leq s_\ell\}\right), \\
 \Pb^{\rm pf}(s_1,\dots,s_m) &= \Pb\left(\bigcap_{\ell=1}^m \{L^{\rm pf}_{N,j_\ell}\leq s_\ell\}\right).
 \end{split}
\end{equation}
Then
\begin{equation}\label{eq:shift3}
 \widetilde\Pb(s_1,\dots,s_m) = \left(1+\frac{1}{\alpha+\beta}\sum_{\ell=1}^m\partial_{s_\ell}\right)\Pb^{\rm pf}(s_1,\dots,s_m).
\end{equation}
\end{lem}
\begin{proof}
It is a generalization of the proof of~\cite[Lemma 3.3]{BFO20} for the one-point distribution; see~\cite[Proposition 2.1]{BFP09} for the full argument.
\end{proof}

\subsubsection{Kernel decomposition} \label{sec:kernel_decomp}

Throughout this section we fix $1 \leq j_1 < \cdots < j_m \leq N$ a sequence of ordinates (times) and $m$ real numbers $s_\ell \geq 0$, $1 \leq \ell \leq m$.

We use $K = K(x, y)$ to stand for the $2m \times 2m$ matrix kernel from Theorem~\ref{thm:exp_corr}. We denote its $2 \times 2$ block at $(k, \ell)$ ($1 \leq k, \ell \leq m$) by
\begin{equation}
 K^{j_k j_\ell} (x, y) = K(j_k, x; j_\ell, y)
\end{equation}
to save space in most of the formulas below (the right-hand side above uses the notation of Theorem~\ref{thm:exp_corr}).

Note that $K$ satisfies $K(x, y) = -K^t (y, x)$, that is, we have the following symmetries:
\begin{equation} \label{eq:ker_symm}
 K^{j_k j_\ell}_{12} (x, y) = - K^{j_\ell j_k}_{21} (y, x), \qquad K^{j_k j_\ell}_{aa} (x, y) = - K^{j_\ell j_k}_{aa} (y, x), \quad a=1,2.
\end{equation}
Thus, the $(k,\ell)$ block satisfies
\begin{equation}\label{eq:k_ex}
 \begin{pmatrix}
 K^{j_k j_\ell}_{11}(x,y) & K^{j_k j_\ell}_{12}(x,y) \\
 K^{j_k j_\ell}_{21}(x,y) & K^{j_k j_\ell}_{22}(x,y)
 \end{pmatrix}
 =
\begin{pmatrix}
 K^{j_k j_\ell}_{11}(x,y) & K^{j_k j_\ell}_{12}(x,y) \\
 -K^{j_\ell j_k}_{12}(y,x) & K^{j_k j_\ell}_{22}(x,y)
 \end{pmatrix} .
\end{equation}

We denote by $V=V(x, y)$ the following block $2m \times 2m$ matrix: at position $(k, \ell)$, $1 \leq k, \ell \leq m$, it has the following two-by-two block:
\begin{equation}
\begin{pmatrix} 0 & 0 \\ -V^{j_\ell j_k} (y, x) & 0 \end{pmatrix}\textrm{ if }k < \ell,\quad
\begin{pmatrix} 0 & V^{j_k j_\ell} (x, y) \\ 0 & 0 \end{pmatrix} \textrm{ if }k > \ell,\quad \textrm{and }
\begin{pmatrix} 0 & 0 \\ 0 & 0 \end{pmatrix} \textrm{ if }k = \ell.
\end{equation}
As above, we use the notation
 \begin{equation}
 V^{j_k j_\ell} (x, y) := V(j_k, x; j_\ell, y),
 \end{equation}
which is non-zero only for $k > \ell$ (as this corresponds to $j_k>j_\ell$).
Finally we use $P_s = P_s(x)$ to stand for the following $2m \times 2m$ matrix of projectors:
\begin{equation}
 P_s(x) = \mathrm{diag} (\Id_{[x > s_1]}, \Id_{[x > s_1]}, \dots, \Id_{[x > s_m]}, \Id_{[x > s_m]}).
\end{equation}
Finally, define the following auxiliary functions:
\begin{equation} \label{eq:fg_def}
\begin{aligned}
f_{+,\beta}^{j} (x) &=\Phi(x, \beta)\left( \tfrac12-\beta \right)^{N-j}, & g_1^{j}(x)&=\oint\limits_{\Gamma_{1/2}} \frac{dz}{2\pi\I} \Phi (x,z)\left( \tfrac12-z \right)^{N-j}\frac{z+\alpha}{2z},\\
f_{-,\beta}^{j} (x) &=\frac{\Phi(x, \beta)}{\left( \tfrac12+\beta \right)^{N-j}},& g_2^{j}(x)&=\!\!\! \oint\limits_{\Gamma_{1/2,\alpha}} \!\!\!\frac{dz}{2\pi\I} \frac{\Phi(x,z)}{\left( \tfrac12+z \right)^{N-j}}\frac{1}{z-\alpha}.
\end{aligned}
\end{equation}

\begin{prop} \label{prop:decomposition} Let $\alpha \in (-1/2,1/2)$, $\beta \in (0, 1/2)$. Then the kernel $K$ splits as
 \begin{equation}
 K = \breve{K} + (\alpha+\beta) R,
 \end{equation}
 where for $1\leq k,\ell\leq m$, the $(k, \ell)$ $2 \times 2$ block of $\breve{K}$ is given by
 \begin{equation} \label{eq:h_bar_def}
 \begin{split}
 \breve{K}^{j_k j_\ell}_{11}(x, y) =& - \!\!\! \oint\limits_{\Gamma_{1/2}} \!\!\! \frac{dz}{2\pi\I} \!\!\! \oint\limits_{\Gamma_{-1/2}} \!\!\! \frac{dw}{2\pi\I}\frac{\Phi(x,z)}{\Phi(y,w)}(\tfrac12-z)^{N-j_k} (\tfrac12+w)^{N-j_\ell} \\
 & \qquad \qquad \qquad \qquad \times \frac{(z+\beta)(w-\beta)}{(z-\beta)(w+\beta)}\frac{(z+\alpha)(w-\alpha)(z+w)}{4zw(z-w)}, \\
 \breve{K}^{j_k j_\ell}_{12}(x, y) =& - \!\!\! \oint\limits_{\Gamma_{1/2}} \!\!\! \frac{dz}{2\pi\I} \!\!\! \oint\limits_{\Gamma_{-1/2,-\alpha,-\beta}} \!\!\!\!\!\!\!\!\!\! \frac{dw}{2\pi\I}\frac{\Phi(x,z)}{\Phi(y,w)} \frac{(\tfrac12-z)^{N-j_k}} {(\tfrac12-w)^{N-j_\ell}} \frac{z+\alpha}{w+\alpha}\frac{z+\beta}{z-\beta}\frac{w-\beta}{w+\beta}\frac{z+w}{2z(z-w)} + V^{j_k j_\ell} (x, y)\\
 =& - \breve{K}_{21}^{j_\ell j_k}(y, x), \\
 \breve{K}^{j_k j_\ell}_{22}(x, y) =&\ \tilde\e^{j_k j_\ell}(x,y)+\oint \frac{dz}{2\pi\I}\oint\frac{dw}{2\pi\I} \frac{\Phi(x,z)}{\Phi(y,w)} \frac{1} {(\tfrac12+z)^{N-j_k} (\tfrac12-w)^{N-j_\ell}}\\
 & \qquad \qquad \qquad \qquad \qquad \qquad \quad \times \frac{1}{(z-\alpha)(w+\alpha)} \frac{z+\beta}{z-\beta}\frac{w-\beta}{w+\beta}\frac{z+w}{z-w},
 \end{split}
 \end{equation}
where the integration contours for $\breve{K}^{j_k j_\ell}_{22}$ are, for $(z,w)$, the union of $\Gamma_{1/2,\alpha,\beta}\times\Gamma_{-1/2}$, $\Gamma_{1/2,\beta}\times \Gamma_{-\alpha}$, and $\Gamma_{1/2,\alpha}\times\Gamma_{-\beta}$.

The operator ($2 \times 2$ matrix kernel) $R^{j_k j_\ell}$ is of rank two and given by
 \begin{equation}
 R^{j_k j_\ell}=\begin{pmatrix}
 \ketbra{g_1^{j_k}}{f_{+,\beta}^{j_\ell}} - \ketbra{f_{+,\beta}^{j_k}}{g_1^{j_\ell}} &\quad \ketbra{f_{+,\beta}^{j_k}} {g_2^{j_\ell}} \\[10pt]
 - \ketbra{ g^{j_k}_2}{f_{+,\beta}^{j_\ell}} & \quad 0
 \end{pmatrix} = \ketbra{ \begin{matrix} f^{j_k}_{+, \beta} \\ 0 \end{matrix}}{-g_1^{j_\ell} \quad g_2^{j_\ell}} + \ketbra{\begin{matrix} g_1^{j_k} \\ -g_2^{j_k} \end{matrix}} { f_{+, \beta}^{j_\ell} \quad 0}.
 \end{equation}
\end{prop}

\begin{proof}
 The decomposition follows from standard residue computations. Namely, we include in $R^{j_k j_\ell}$ the contribution from (a) the residues at $(z, w) \in \{(\beta,-1/2), (1/2,-\beta)\}$ for $K_{11}^{j_k j_\ell}$, and (b) the residues at $(z, w) \in \{ (\beta,-1/2), (\beta,-\alpha)\}$ for $K_{12}^{j_k j_\ell}$. The reader can easily verify that these residue computations give $R$ the above factorization. This finishes the proof.
\end{proof}

\begin{rem}
We can decompose the kernel $\breve{K}^{j_k j_\ell}_{22}(x, y) - \tilde\e^{j_k j_\ell}(x,y)$ using the identities
\begin{equation}
\frac{(z+\beta)(w-\beta)}{(z-\beta)(w+\beta)(z-w)}=\frac{1}{z-w}+\frac{2\beta}{(w+\beta)(z-\beta)},\quad \frac{z+w}{(z-\alpha)(w+\alpha)}=\frac{1}{z-\alpha}+\frac{1}{w+\alpha}
\end{equation}
as
\begin{equation}\label{eq:term_1}
\oint \frac{dz}{2\pi\I}\oint\frac{dw}{2\pi\I} \frac{\Phi(x,z)}{\Phi(y,w)} \frac{1}{(\tfrac12+z)^{N-j_k}(\tfrac12-w)^{N-j_\ell}}\frac{1}{z-w}\left(\frac{1}{z-\alpha}+\frac{1}{w+\alpha}\right)
\end{equation}
plus
\begin{equation}\label{eq:term_2}
\oint \frac{dz}{2\pi\I}\oint\frac{dw}{2\pi\I} \frac{\Phi(x,z)}{\Phi(y,w)} \frac{1}{(\tfrac12+z)^{N-j_k}(\tfrac12-w)^{N-j_\ell}} \frac{2\beta}{(w+\beta)(z-\beta)}\left(\frac{1}{z-\alpha}+\frac{1}{w+\alpha}\right).
\end{equation}
The integration contour for the term in~\eqref{eq:term_1} and~\eqref{eq:term_2} with $1/(z-\alpha)$ is $\Gamma_{1/2,\alpha}\times\Gamma_{-1/2}$, while the one for the term with $1/(w+\alpha)$ is $\Gamma_{1/2}\times\Gamma_{-1/2,-\alpha}$.
\end{rem}

Now we have decomposed the kernel as
\begin{equation}
 K = \overline K + V + (\alpha+\beta)R
\end{equation}
where by definition
\begin{equation}
 \overline K = \breve K - V.
\end{equation}

Further define
\begin{equation}
 (G, \breve G, \overline{G}) = (J^{-1} K, J^{-1} \breve K, J^{-1} \overline{K}), \qquad T=J^{-1}R, \qquad W = J^{-1} V,
\end{equation}
where $J$ is the $2m \times 2m$ matrix with $2 \times 2$ diagonal blocks given by $\begin{pmatrix} 0 & 1 \\ -1 & 0 \end{pmatrix}$ and zeros elsewhere. The effect of multiplying a $2m \times 2m$ matrix by $J^{-1}$ on the left is as follows: if $\begin{pmatrix} a & b \\ c & d \end{pmatrix}$ is the block of the original matrix at $(k, \ell)$ ($1 \leq k, \ell \leq m$), it is sent to $\begin{pmatrix} -c & -d \\ a & b \end{pmatrix}$.

Thus the $(k,\ell)$ block of $G$ reads
\begin{equation} \label{eq:g_ex}
 \begin{pmatrix}
 G^{j_k j_\ell}_{11}(x,y) & G^{j_k j_\ell}_{12}(x,y) \\
 G^{j_k j_\ell}_{21}(x,y) & G^{j_k j_\ell}_{22}(x,y)
 \end{pmatrix}
=
 \begin{pmatrix}
 K^{j_\ell j_k}_{12}(y,x) & -K^{j_k j_\ell}_{22}(x,y) \\
 K^{j_k j_\ell}_{11}(x,y) & K^{j_k j_\ell}_{12}(x,y)
 \end{pmatrix}.
\end{equation}

Notice the entries satisfy (a consequence of equation~\eqref{eq:ker_symm}) the following symmetries:
 \begin{equation}\label{eq3.31}
 G^{j_k j_\ell}_{11} (x, y) = G^{j_\ell j_k}_{22} (y, x), \quad G_{12}^{j_k j_\ell}(x,y)=-G_{12}^{j_\ell j_k}(y,x),\quad G_{21}^{j_k j_\ell}(x,y)=-G_{21}^{j_\ell j_k}(y,x).
 \end{equation}

We can write $T = J^{-1}R$ as
 \begin{equation}
 T = \ketbra{X_1}{Y_1} + \ketbra{X_2}{Y_2}
\end{equation}
with
\begin{equation} \label{eq:X_def}
X_1 = \left(g_2^{j_1}, g_1^{j_1}, \cdots, g_2^{j_m}, g_1^{j_m}\right)^t, \qquad X_2 = \left( 0,f_{+,\beta}^{j_1},\cdots, 0, f_{+,\beta}^{j_m} \right)^t
\end{equation}
and
\begin{equation} \label{eq:Y_def}
 Y_1 = \left( f_{+,\beta}^{j_1}, 0, \cdots, f_{+,\beta}^{j_m} , 0 \right), \qquad Y_2 = \left( -g_1^{j_1} , g_2^{j_1} , \cdots ,-g_1^{j_m} , g_2^{j_m} \right).
\end{equation}

Furthermore we can write
\begin{equation}\label{eq:det_decomp}
\begin{split}
 \det(\Id - G) &= \det(\Id - \breve{G}) \cdot \det (\Id - (\alpha+\beta) (\Id-\breve{G})^{-1} T)\\
 &=\det(\Id - \breve{G}) \cdot\det\left(\Id-(\alpha+\beta)
 \begin{pmatrix}
  \braket{Y_1}{Z_1} & \braket{Y_2}{Z_1}\\
  \braket{Y_1}{Z_2} & \braket{Y_2}{Z_2}
 \end{pmatrix}
 \right) \\
 &= \det(\Id-\breve{G}) (1-(\alpha+\beta)\braket{Y_2}{Z_2})^2
 \end{split}
\end{equation}
with $Z_i=(\Id-\breve G)^{-1} X_i$. Here we have used the fact that
\begin{equation} \label{eq:xz_eq}
 \braket{Y_1}{Z_2} = \braket{Y_2}{Z_1} = 0, \qquad \braket{Y_1}{Z_1} = \braket{Y_2}{Z_2}.
\end{equation}
For a proof of the computations and of the equalities from~\eqref{eq:xz_eq}, the arguments from respectively Section 3.2.3 and Appendix B of~\cite{BFO20} apply \emph{mutatis mutandis}. Since all of this is valid on $L^2((s_1, \infty)) \oplus \dots \oplus L^2((s_m, \infty))$, upon reintroducing the projectors and writing $X_s := P_s X P_s$ ($X$ a matrix kernel), we are interested in the following quantity:
\begin{equation}
 \begin{split}
 \det(\Id - G_s ) &= \det(\Id - \breve{G}_s ) \cdot (1-(\alpha+\beta) \braket{P_s Y_2}{P_s Z_2})^2 \\
 &= \det(\Id - \breve{G}_s ) \cdot \left( 1-(\alpha+\beta) \braket{P_s Y_2}{(\Id-\breve{G}_s)^{-1} P_s X_2} \right)^2.
 \end{split}
\end{equation}
Equivalently in terms of Fredholm Pfaffians and after dividing by $(\alpha + \beta)$, we need to study
\begin{equation}
 \frac{1}{\alpha+\beta} \pf(J - K_s ) = \pf(J - \breve{K}_s ) \cdot \left( \frac{1}{\alpha+\beta} - \braket{P_s Y_2}{(\Id-\breve{G}_s)^{-1} P_s X_2} \right).
\end{equation}

We wish to show that it is analytic for any $\alpha, \beta \in (-\tfrac12, \tfrac12)$ and determine its limit as $\beta \to - \alpha$. To that end, we further decompose the above expression into more manageable terms. Let us first define the following $2m \times 1$ vector:
\begin{equation}\label{eq:X_tilde_def}
 \widetilde{X}_2 = \left( 0 ,f_{+,\beta}^{j_1}, 0 ,\cdots, 0, 0 \right)^t.
\end{equation}

We have the following lemma.

\begin{lem}\label{lem:magic}
 It holds that
 \begin{equation}\label{eq:magic}
 X_2 = \widetilde{X}_2 - W \widetilde{X}_2.
 \end{equation}
\end{lem}

\begin{proof}
 Let us compute the column vector $W \widetilde{X}_2$: it has zeros in the odd components and in the 2nd component as well, while in component $2 \ell$, $1 < \ell \leq m$ it has the following entry:
 \begin{equation}
 \begin{split}
 (W \widetilde{X}_2)_{2 \ell}(x) &= \int_\R V^{j_\ell j_1}(x, y) f_{+,\beta}^{j_1}(y) dy \\
 &=-\phi(z) (\tfrac12-\beta)^{N-j_1} \int_{-\infty}^\infty dy e^{-y \beta} \int_{\I \R} \frac{dz}{2 \pi \I} \frac{e^{-(x-y)z}}{(\tfrac12-z)^{j_\ell-j_1}}.
 \end{split}
 \end{equation}
By choosing $\eta$ an appropriately small positive number and shifting the $\I \R$ to the left or right of $\beta$ (without crossing $1/2$), the integral above can be decomposed as:
\begin{equation}
 \begin{split}
 &\int_{-\infty}^\infty dy e^{-y \beta} \int_{\I \R} \frac{dz}{2 \pi \I} \frac{e^{-(x-y)z}}{(\tfrac12-z)^{j_\ell-j_1}} \\
 & \quad = \int_{-\infty}^0 dy e^{-y \beta} \int_{\beta + \eta + \I \R} \frac{dz}{2 \pi \I} \frac{e^{-(x-y)z}}{(\tfrac12-z)^{j_\ell-j_1}} + \int_{0}^\infty dy e^{-y \beta} \int_{\beta - \eta + \I \R} \frac{dz}{2 \pi \I} \frac{e^{-(x-y)z}}{(\tfrac12-z)^{j_\ell-j_1}}.
 \end{split}
\end{equation}
We can now interchange order of integration in both integrals by Fubini, and integrate in $y$ explicitly since we have $\Re(z) > \beta$ in the first integral (ensuring convergence) and $\Re(z) < \beta$ in the second. The result we obtain is (recall $\I \R$ is bottom-to-top oriented):
\begin{equation}
 \begin{split}
 &\int_{\beta + \eta + \I \R} \frac{dz}{2 \pi \I} \frac{e^{-xz}}{(\tfrac12-z)^{j_\ell-j_1} (z-\beta)} + \int_{\beta - \eta + \I \R} \frac{dz}{2 \pi \I} \frac{e^{-xz}}{(\tfrac12-z)^{j_\ell-j_1} (\beta-z)} \\
 & \quad = \oint\limits_{\Gamma_\beta} \frac{dz}{2 \pi \I} \frac{e^{-xz}}{(\tfrac12-z)^{j_\ell-j_1} (z-\beta)} =\frac{e^{-x \beta}} {(\tfrac12-\beta)^{j_\ell-j_1}}.
 \end{split}
\end{equation}
In the second equation we changed the sign of $(\beta-\zeta)$ by reversing the orientation of $\I \R$ (so that it goes top-down), in the third we closed the contour around $\beta$ (at $\pm \I \infty$), and in the fourth we evaluated the integral via residue calculus. Putting all together and recalling the minus sign in front of everything, we get:
\begin{equation}
 (W \widetilde{X}_2)_{2 \ell}(x) = -f_{+,\beta}^{j_\ell}(x), \quad 1 < \ell \leq m.
\end{equation}
This gives the result upon recalling the definition of $X_2$ and $\widetilde{X}_2$.
\end{proof}

In view of what we have just proven, we have the following further decomposition of our inner product.

\begin{lem} \label{lem:ip_split}
 We have:
 \begin{equation} \label{eq:ip_split}
 \begin{split}
 \frac{1}{\alpha+\beta} - \braket{P_s Y_2}{(\Id-\breve{G}_s)^{-1} P_s X_2} = & \frac{1}{\alpha+\beta} - \braket{P_s Y_2}{P_s \widetilde{X}_2} - \braket{P_s Y_2} {(\Id-\breve{G}_s)^{-1} \overline{G}_s \widetilde{X}_2} \\
 & + \braket{P_s Y_2}{(\Id-\breve{G}_s)^{-1} P_s W (\Id - P_s) \widetilde{X}_2}.
 \end{split}
 \end{equation}
\end{lem}

\begin{proof}
 From Lemma~\ref{lem:magic} we have that $X_2 = \widetilde{X}_2 - W \widetilde{X}_2$ and splitting the quantity based on this yields
 \begin{equation}
 \text{l.h.s. of~\eqref{eq:ip_split}} =\frac{1}{\alpha+\beta} - \braket{P_s Y_2}{(\Id-\breve{G}_s)^{-1} P_s \widetilde{X}_2} + \braket{P_s Y_2}{(\Id-\breve{G}_s)^{-1} P_s W \widetilde{X}_2}.
 \end{equation}
 On the right-hand side above we use $(\Id-\breve{G}_s)^{-1} = \Id + (\Id-\breve{G}_s)^{-1} \breve{G}_s$ for the second term and $W = W (\Id - P_s + P_s)$ for the third term, and recombine the expansions recalling $\breve{G}_s = \overline{G}_s + W_s$ and that subscript $s$ means projection by $P_s$. The result follows.
\end{proof}

Putting everything together, we have
\begin{equation}
 \begin{split}
 \frac{1}{\alpha+\beta} \pf(J - K_s ) = \pf(J - \breve{K}_s ) \cdot & \left( \frac{1}{\alpha+\beta} - \braket{P_s Y_2}{P_s \widetilde{X}_2} - \braket{P_s Y_2} {(\Id-\breve{G}_s)^{-1} \overline{G}_s \widetilde{X}_2} \right. \\
 & \left. + \braket{P_s Y_2}{(\Id-\breve{G}_s)^{-1} P_s W (\Id - P_s) \widetilde{X}_2} \right).
 \end{split}
\end{equation}
We will show that the following four terms are analytic for all $\alpha,\beta\in (-1/2,1/2)$ and then take their limit as $\beta\to -\alpha$:
\begin{equation}\label{eq3.49}
 \begin{split}
 \text{Fredholm Pfaffian: } \quad &\pf(J - \breve{K}_s ), \\
 \text{term A: }\quad &\frac{1}{\alpha+\beta} - \braket{P_s Y_2}{P_s \widetilde{X}_2}, \\
 \text{term B: }\quad &\braket{P_s Y_2} {(\Id-\breve{G}_s)^{-1} \overline{G}_s \widetilde{X}_2}, \\
 \text{term C: }\quad &\braket{P_s Y_2}{(\Id-\breve{G}_s)^{-1} P_s W (\Id - P_s) \widetilde{X}_2}.\\
 \end{split}
\end{equation}

\subsection{Analytic continuation} \label{sec:analytic_continuation}

In this section we will show that the different terms in~\eqref{eq3.49} are analytic in $\alpha,\beta$ in a bounded subset of $(-1/2,1/2)$. More precisely, fix arbitrarily small $\epsilon>0$. We will show analyticity for $\alpha,\beta\in [-1/2+\epsilon,1/2-\epsilon]$.

The functions and kernels appearing in~\eqref{eq3.49} are not necessarily $L^2$ in the natural space, but they are once we will apply a proper conjugation, given as follows. Let us first choose $m$ positive numbers
\begin{equation} \label{eq:conj_constants}
 \frac{1}{2} - \epsilon < \mu_m < \mu_{m-1} < \cdots < \mu_2 < \mu_1 < \frac{1}{2} - \frac{\epsilon}{2}.
\end{equation}
For a $2m\times 2m$ kernel $K$, we define its conjugate as follows:
\begin{equation}
 K_{\rm conj}(x,y) = M(x) K(x, y) M(y),
\end{equation}
where
\begin{equation} \label{eq:M_def}
M(x)={\rm diag}(e^{\mu_1 x}, e^{-\mu_1 x}, \dots, e^{\mu_m x}, e^{-\mu_m x}).
\end{equation}
For the first term in~\eqref{eq3.49} we have
\begin{equation}
\pf(J-\breve{K}_s) = \pf(J-\breve{K}_{s,{\rm conj}}).
\end{equation}
For the other terms in~\eqref{eq3.49}, we need to see how the conjugation acts on functions. First of all notice that
\begin{equation}
J^{-1} M^{-1} = M J^{-1}
\end{equation}
and recall that $\breve{G}=J^{-1}\breve{K}$. Thus for the scalar product we have
\begin{equation}
\begin{split}
\braket{a}{(\Id-\breve{G})^{-1} b} &= \braket{a}{(\Id-J^{-1}M^{-1}\breve{K}_{\rm conj} M^{-1})^{-1} b} \\
&= \braket{a}{ (\Id- M J^{-1} \breve{K}_{\rm conj}M^{-1})^{-1} b} \\
&=\braket{a M}{ (\Id- \breve{G}_{\rm conj})^{-1} M^{-1} b} \\
&=\braket{a_{\rm conj}}{ (\Id- \breve{G}_{\rm conj})^{-1} b_{\rm conj}},
\end{split}
\end{equation}
where $\breve{G}_{\rm conj}= J^{-1} \breve{K}_{\rm conj}= M^{-1} \breve{G} M$ and
\begin{equation}
a_{\rm conj}(x)= a(x) M(x),\quad b_{\rm conj}(y) = M^{-1}(y) b(y)=M(-y) b(y).
\end{equation}
With the above given conjugation, the operators and functions will then be in $L^2$. For term B we need to be a bit more careful since we will first take away some terms which give zero scalar product before doing the analytic continuation to $\beta \leq 0$.

\subsubsection{Analytic continuation and limit of the Fredholm Pfaffian}

We start by showing that the Fredholm Pfaffian $\pf (J-\breve{K_s})$ is analytic in $\alpha, \beta \in (-1/2, 1/2)$ with well-defined limit as $\beta \to -\alpha$. We further omit writing $K_s = P_s K P_s$ for our kernels throughout this section for simplicity but implicitly we assume the projector $P_s$ everywhere.

\begin{lem} \label{lem:Kbreve}
 The $2m \times 2m$ matrix kernel $\breve{K}$ is analytic for $\alpha, \beta \in (-1/2, 1/2)$. The limit kernel $\breve{\mathsf{K}}=\lim_{\beta\to-\alpha} \breve{K}$ has the $2 \times 2$ block at $(k, \ell)$ given by:
 \begin{equation}\label{eq:336}
 \begin{split}
 \breve{\mathsf{K}}^{j_k j_\ell}_{11}(x,y) =& - \oint\limits_{\Gamma_{1/2}} \frac{dz}{2\pi\I} \oint\limits_{\Gamma_{-1/2}} \frac{dw}{2\pi\I}\frac{\Phi(x,z)}{\Phi(y,w)} (\tfrac12-z)^{N-j_k} (\tfrac12+w)^{N-j_\ell} \frac{(z-\alpha)(w+\alpha)(z+w)}{4zw(z-w)}, \\
 \breve{\mathsf{K}}^{j_k j_\ell}_{12} (x,y) =& - \!\!\! \oint\limits_{\Gamma_{1/2}} \frac{dz}{2\pi\I} \!\!\! \oint\limits_{\Gamma_{-1/2,\alpha}} \!\!\! \frac{dw}{2\pi\I}\frac{\Phi(x,z)}{\Phi(y,w)} \frac{(\tfrac12-z)^{N-j_k}} {(\tfrac12-w)^{N-j_\ell}} \frac{z-\alpha}{w-\alpha}\frac{z+w}{2z(z-w)} + V^{j_k j_\ell} (x, y) \\
 =& -\breve{\mathsf{K}}^{j_\ell j_k}_{21} (y,x),\\
 \breve{\mathsf{K}}^{j_k j_\ell}_{22}(x,y) =&\ \e^{j_k j_\ell}(x,y)+\oint \frac{dz}{2\pi\I}\oint\frac{dw}{2\pi\I} \frac{\Phi(x,z)}{\Phi(y,w)} \frac{1}{(\tfrac12+z)^{N-j_k} (\tfrac12-w)^{N-j_\ell}} \frac{1}{z-w}\left(\frac{1}{z+\alpha}+\frac{1}{w-\alpha}\right),
 \end{split}
 \end{equation}
 where the integration contours for $\breve{\mathsf{K}}^{j_k j_\ell}_{22}$ are $\Gamma_{1/2,-\alpha}\times\Gamma_{-1/2}$ for the term with $1/(z+\alpha)$ and $\Gamma_{1/2}\times\Gamma_{-1/2,\alpha}$ for the term with $1/(w-\alpha)$. Furthermore, $\e$ is given by
 \begin{equation}
 \begin{split}
 \e^{j_k j_\ell} (x, y) &= \tilde{\e}^{j_k j_\ell} (x, y) + \frac{e^{-(x-y) \alpha}} {(\tfrac12 + \alpha)^{N-j_k} (\tfrac12 - \alpha)^{N-j_\ell}} - \frac{e^{(x-y) \alpha}} {(\tfrac12 - \alpha)^{N-j_k} (\tfrac12 + \alpha)^{N-j_\ell}} \\
 &= -\oint\limits_{\Gamma_{1/2}} \frac{dz}{2 \pi \I} \frac{2 z e^{-(x-y)z}} {\left[\tfrac12+z\right]^{N-j_k} \left[\tfrac12-z\right]^{N-j_\ell}} \frac{1}{z^2-\alpha^2} - \frac{e^{(x-y) \alpha}} {(\tfrac12 - \alpha)^{N-j_k} (\tfrac12 + \alpha)^{N-j_\ell}}, \quad \text{if } x \geq y,
 \end{split}
 \end{equation}
where $\e$ is anti-symmetric $\e^{j_k j_\ell} (x, y) = - \e^{j_\ell j_k} (y, x)$ (covering the case $x < y$ above).
 \end{lem}

\begin{proof}
 The proof is almost the same as that of Lemma 3.7 of~\cite{BFO20}, so we only highlight the differences. Firstly, one should compare $\breve{K}$ in this manuscript with $\overline{K}$ in~\cite[Lemma 3.7]{BFO20}; in the latter we are considering a $2 \times 2$ matrix kernel, while here we have a $2m \times 2m$ matrix kernel but nevertheless the entries are very similar. Secondly, our $\breve{K}_{12}$ has the extra $V$ kernel, but this latter is independent of $(\alpha, \beta)$ and so it is obviously analytic in them. Thirdly, in~\cite{BFO20} we had a single integer parameter $n$ in all $z$ and $w$ integrands, and that is replaced here by $N-j_k$ for the $z$ integrands and $N-j_\ell$ for the $w$ integrands: again this does not affect analyticity. Finally $\tilde \e$ is also different from that of~\cite{BFO20}, since it now depends on $j_k$ and $j_\ell$, but this does not affect analyticity: for the limit $\beta \to -\alpha$, $\e$ comes from $\tilde{\e}$ to which we have added the two explicit terms above, which are the residues of the integrand in $\breve{K}^{j_k j_\ell}_{22}$ at $(z, w) \in \{ (\beta, -\alpha), (\alpha, -\beta) \}$ in the limit $\beta \to -\alpha$. Note that $\tilde{\e}$ only depends on $\alpha$ so there is no limit $\beta \to -\alpha$ to be taken for it.
\end{proof}

We now show that also the Fredholm Pfaffian itself is analytic with a well-defined $\beta \to -\alpha$ limit.

\begin{prop}\label{prop:cvgOfFredPf}
 The Fredholm Pfaffian $\pf (J- \breve{K})$ is analytic for $\alpha, \beta \in (-1/2, 1/2)$. It has the following well-defined $\beta \to -\alpha$ limit:
 \begin{equation}
 \lim_{\beta\to -\alpha} \pf (J- \breve{K}) = \pf (J - \breve{\mathsf{K}}).
 \end{equation}
\end{prop}

\begin{proof}
We start by choosing a small positive $\epsilon$ and fixing $\alpha, \beta \in [-1/2+\epsilon, 1/2-\epsilon]$. Let us look at the $2 \times 2$ $(k, \ell)$ block $\breve{K}^{j_k j_\ell}(x, y)$ of $\breve{K}$ and how it behaves as $x, y \to \infty$. We have:
\begin{equation}
 \begin{split}
 \breve{K}^{j_k j_\ell}(x, y) &=
 \left(
 \begin{array}{cc}
 \breve{K}^{j_k j_\ell}_{11} (x, y) & \breve{K}^{j_k j_\ell}_{12} (x, y) \\
 \breve{K}^{j_k j_\ell}_{21} (x, y) & \breve{K}^{j_k j_\ell}_{22} (x, y)
 \end{array}
 \right) \\
 & \leq
 C \left(
 \begin{array}{cc}
  e^{-(\frac{1}{2}-\frac{\epsilon}{2})x} e^{-(\frac{1}{2}-\frac{\epsilon}{2})y} & e^{-(\frac{1}{2}-\frac{\epsilon}{2})x} e^{(\frac{1}{2}-\epsilon)y} + \Id_{[j_k > j_\ell]} \Id_{[x \geq y]} e^{-(\frac{1}{2}-\frac{\epsilon}{2})|x-y|} \\
 \cdots & e^{(\frac{1}{2} - \epsilon)x} e^{(\frac{1}{2} - \epsilon)y }
 \end{array}
 \right),
 \end{split}
\end{equation}
where $C$ is a constant independent of $x, y$ and we put dots in the $21$ entry as the bounds are like in the 12 term due to the relation $\breve{K}^{j_k j_\ell}_{21} (x, y) = -\breve{K}^{j_\ell j_k}_{12} (y, x)$.

We achieve these bounds as follows: for the $11$ entry, we take the integrals around $\pm 1/2$ to have contours $|z-1/2| = \epsilon/2$ and $|w + 1/2| = \epsilon/2$ respectively and then we use the estimate
 \begin{equation}
 \left| e^{-x z} \right| = e^{-x \Re(z)} \leq e^{-(1/2-\epsilon/2)x}.
 \end{equation}

 For the $12$ entry we consider the double contour integrals first. We proceed similarly for the $z$ integral but we notice the $w$ integral is dominated by the poles at $-\alpha, -\beta$ which give an asymptotic contribution in $y$ as $e^{\max(-\alpha, -\beta)y} = e^{-\min(\alpha, \beta) y}$ and the stated bound in the first summand follows from our choice $\alpha, \beta \in [-1/2+\epsilon, 1/2-\epsilon]$. Finally, the second term in the bound comes from the $V$ function which we can write as
\begin{equation}
 V^{j_k j_\ell}(x, y) = \Id_{[j_k > j_\ell]} \Id_{[x \geq y]} \oint\limits_{\Gamma_{1/2}} \frac{dz}{2 \pi \I} \frac{e^{-(x-y)z}} {(\tfrac12-z)^{j_k - j_\ell}}
\end{equation}
and then to obtain the stated bound we use again the contour $|z-1/2| = \epsilon/2$.

Finally, for the $22$ entry the asymptotic contribution comes from the poles at $\alpha$ and $\beta$ (for the $z$ integrals) and $-\alpha, -\beta$ respectively (for $w$).

Since $\pf(J-\breve{K}_s) = \pf(J-\breve{K}_{s,{\rm conj}})$, we use the conjugated kernel instead of the original one. The $(k,\ell)$ block of $\breve{K}_{\rm conj}$ is given by
\begin{equation}
 \breve{K}^{j_k j_\ell}_{\rm conj} (x, y) =
 \left(
 \begin{array}{cc}
 \breve{K}^{j_k j_\ell}_{11} (x, y) e^{\mu_k x} e^{\mu_\ell y} & \breve{K}^{j_k j_\ell}_{12} (x, y) e^{\mu_k x} e^{-\mu_\ell y} \\
 \breve{K}^{j_k j_\ell}_{21} (x, y) e^{-\mu_k x} e^{\mu_\ell y} & \breve{K}^{j_k j_\ell}_{22} (x, y) e^{-\mu_k x} e^{-\mu_\ell y}
 \end{array}
 \right).
\end{equation}
The right-hand side above is then bounded, entry-wise, by (note we again recover the $21$ entry from anti-symmetry):
\begin{equation}
 C \left(
 \begin{array}{cc}
 e^{-(\frac{1}{2}-\frac{\epsilon}{2})x} e^{-(\frac{1}{2}-\frac{\epsilon}{2})y } e^{\mu_k x} e^{\mu_\ell y} & e^{-(\frac{1}{2}-\frac{\epsilon}{2})x} e^{(\frac{1}{2}-\epsilon)y} e^{\mu_k x} e^{-\mu_\ell y} + \Id_{[j_k > j_\ell]} \Id_{[x \geq y]} e^{-(\frac{1}{2}-\frac{\epsilon}{2})|x-y|} e^{\mu_k x} e^{-\mu_\ell y} \\
 \cdots & e^{(\frac{1}{2} - \epsilon)x} e^{(\frac{1}{2} - \epsilon)y } e^{-\mu_k x} e^{-\mu_\ell y}
 \end{array}
 \right).
\end{equation}
Let us denote $\overline{\mu}_{k\ell} = \frac{\mu_k + \mu_\ell}{2}$. After some simple algebra, the above bound equals
\begin{equation}\label{eq3.65}
 \begin{split}
 &C \left(
 \begin{array}{cc}
 e^{[\mu_k-(\frac{1}{2}-\frac{\epsilon}{2})]x} e^{[\mu_\ell-(\frac{1}{2}-\frac{\epsilon}{2})]y } & e^{[\mu_k-(\frac{1}{2}-\frac{\epsilon}{2})]x} e^{[(\frac{1}{2}-\epsilon)-\mu_\ell] y} \\
 \cdots & e^{[(\frac{1}{2} - \epsilon) - \mu_k]x} e^{[(\frac{1}{2} - \epsilon) - \mu_\ell]y }
 \end{array}
 \right) \\
 +\ & C \left(
 \begin{array}{cc}
  0 & \Id_{[j_k > j_\ell]} \Id_{[x \geq y]} e^{(\mu_k - \overline{\mu}_{k\ell}) x} e^{(\overline{\mu}_{k\ell} - \mu_\ell) y} e^{[\overline{\mu}_{k\ell} - (\frac{1}{2}-\frac{\epsilon}{2})](x-y)} \\
 \cdots & 0
 \end{array}
 \right).
 \end{split}
\end{equation}
We see that every single exponent multiplying $x$ or $y$ above is negative due to the inequalities in~\eqref{eq:conj_constants}. Furthermore, the term $e^{[\overline{\mu}_{k\ell} - (\frac{1}{2}-\frac{\epsilon}{2})](x-y)}$ is bounded above by 1 (it only appears for $x \geq y$). Hence all entries decrease exponentially in $x$ and $y$ (recall the $21$ entries are recovered from anti-symmetry) and so $\breve{K}_{\rm conj}$ has entry-wise exponential decay. This decay allows us to apply the usual Hadamard bound for Pfaffians/determinants and conclude that the series for $\pf(J - \breve{K})$ is absolutely convergent. We can then pass the $\beta \to -\alpha$ limit inside the series to conclude by dominated convergence that
\begin{equation}
 \lim_{\beta \to -\alpha} \pf(J - \breve{K}) = \pf(J - \breve{\mathsf{K}}),
\end{equation}
finishing the proof.
\end{proof}

\subsubsection{Analyticity of term A}

The next term (term A) is the easiest to analyze and we have the following result.

\begin{lem} \label{lem:term_A}
 The term $\frac{1}{\alpha+\beta} - \braket{P_s Y_2}{P_s \widetilde{X}_2}$ is analytic for $\alpha, \beta \in (-1/2, 1/2)$, with limit
 \begin{equation}
 \begin{split}
 \lim_{\beta \to -\alpha} \frac{1}{\alpha+\beta} - \braket{P_s Y_2}{P_s \widetilde{X}_2} = - \oint\limits_{\Gamma_{1/2, \alpha}} \frac{dz}{2 \pi \I} \frac{\Phi(s_1, z)}{\Phi(s_1, \alpha)} \left[ \frac{\tfrac12 + \alpha} {\tfrac12 + z} \right]^{N-j_1} \frac{1}{(z - \alpha)^2}.
 \end{split}
 \end{equation}
 \end{lem}

 \begin{proof}
 The proof is identical to the computation of Lemma~3.9 of~\cite{BFO20} with $n$ replaced by $N-j_1$ and $s$ replaced by $s_1$.
 \end{proof}

\subsubsection{Analyticity of term B}

In this section we consider the term $\braket{P_sY_2}{ (\Id - \breve{G}_s)^{-1} \overline{G}_s \widetilde X_2 }$. Let us first look at $\overline{G}_s \widetilde X_2$. We have
\begin{equation}
 (\overline{G}_s \widetilde{X}_2)(x) = \left(\begin{matrix}-(\overline K^{j_1j_1}_{22} f_{+\beta}^{j_1})(x)\\ (\overline K^{j_1j_1}_{12} f_{+\beta}^{j_1})(x)\\ \vdots\\ -(\overline K^{j_mj_1}_{22} f_{+\beta}^{j_1})(x)\\ (\overline K^{j_mj_1}_{12} f_{+\beta}^{j_1})(x)\end{matrix}\right)=\left(\begin{matrix}-\int_{s_1}^\infty\overline K^{j_1j_1}_{22}(x,y) f_{+\beta}^{j_1}(y)dy\\ \int_{s_1}^\infty\overline K^{j_1j_1}_{12}(x,y) f_{+\beta}^{j_1}(y)dy\\ \vdots\\ -\int_{s_m}^\infty\overline K^{j_mj_1}_{22}(x,y) f_{+\beta}^{j_1}(y)dy\\ \int_{s_m}^\infty\overline K^{j_mj_1}_{12}(x,y) f_{+\beta}^{j_1}(y)dy\end{matrix}\right).
\end{equation}

To prove analyticity of this term we follow the strategy from Section 3.3.3 of~\cite{BFO20}. The technical issues are as follows. The contribution from the pole $w=-\beta$ of $\overline K^{j_\ell j_1}_{22}$ is of the form $\ketbra{a}{f^{j_1}_{-,\beta}}$, which when multiplied by $\ket{f^{j_1}_{+,\beta}}$ is well-defined only for $\beta>0$. Moreover, the pole at $w=-\alpha$ produces a similar term $\ketbra{\tilde a}{f^{j_1}_{-,\alpha}}$ and its scalar product with $\ket{f^{j_1}_{+,\beta}}$ contributes a factor $1/(\beta-\alpha)$ which is not analytic when $\beta=\alpha$. The same occurs for $\overline K^{j_\ell j_1}_{12}$ when considering the pole at $w=-\alpha$.

We first have the following decomposition.

\begin{prop} \label{prop:Seconddecomposition} Let $\alpha \in (-1/2,1/2)$, $\beta>0$. Then the kernel $\overline{K}$ splits as
	\begin{equation}
	\overline{K}^{j_kj_\ell}=\widetilde{K}^{j_kj_\ell} + \begin{pmatrix} 0 & 0 \\ 0 & \tilde \varepsilon^{j_kj_\ell} \end{pmatrix} +\widetilde O^{j_kj_\ell}+\widetilde P^{j_kj_\ell},
	\end{equation}
	where
	\begin{equation}
	\begin{split}
	\widetilde{K}_{11}^{j_kj_\ell} &=\overline{K}_{11}^{j_kj_\ell}, \quad \widetilde{K}_{21}^{j_kj_\ell}=\overline{K}_{21}^{j_kj_\ell},\\
	\widetilde{K}_{12}^{j_kj_\ell}(x,y) &=-\oint\limits_{\Gamma_{1/2}} \frac{dz}{2\pi\I} \oint\limits_{\Gamma_{-1/2}}\frac{dw}{2\pi\I}\frac{\Phi(x,z)}{\Phi(y,w)} \frac{(\tfrac12-z)^{N-j_k}} {(\tfrac12-w)^{N-j_\ell}} \frac{z+\alpha}{w+\alpha}\frac{z+\beta}{z-\beta}\frac{w-\beta}{w+\beta}\frac{z+w}{2z(z-w)} ,\\
	\widetilde{K}_{22}(x,y) &=\oint\limits_{\Gamma_{1/2,\alpha,\beta}} \frac{dz}{2\pi\I}\oint\limits_{\Gamma_{-1/2}}\frac{dw}{2\pi\I} \frac{\Phi(x,z)}{\Phi(y,w)} \frac{1}{(\tfrac12+z)^{N-j_k} (\tfrac12-w)^{N-j_\ell}} \frac{1}{(z-\alpha)(w+\alpha)} \frac{z+\beta}{z-\beta}\frac{w-\beta}{w+\beta}\frac{z+w}{z-w}
	\end{split}
	\end{equation}
	and
	\begin{equation}
	\begin{split}
	\widetilde O^{j_kj_\ell} &=\ketbra{\begin{array}{c} -g^{j_k}_1 \\ g^{j_k}_2 \end{array}}{0 \quad \tfrac{2\beta}{\beta-\alpha}f_{-,\beta}^{j_\ell}-\tfrac{\alpha+\beta}{\beta-\alpha}f_{-,\alpha}^{j_\ell}},\\
	\widetilde P^{j_kj_\ell} &= \ketbra{\begin{array}{c} (\alpha+\beta)g_3^{j_k} \\ -(\alpha+\beta) g_4^{j_k}-f_{-,-\alpha}^{j_\ell} \end{array}}{0 \quad f_{-,\alpha}^{j_\ell}}.
	\end{split}
	\end{equation}
\end{prop}

\begin{proof}
	 We have
	 \begin{equation}
	 \widetilde O_{12}^{j_kj_\ell} + \widetilde P_{12}^{j_kj_\ell}=\oint\limits_{\Gamma_{1/2}}\oint\limits_{\Gamma_{-\alpha,-\beta}}\cdots.
	 \end{equation}
	 The residue computations at $w=-\alpha$ and $w=-\beta$ lead to
\begin{equation}	
\begin{split}
\widetilde{O}^{j_kj_\ell}_{12}+\widetilde{P}^{j_kj_\ell}_{12}=&-\frac{2\beta}{\beta-\alpha}\ketbra{g^{j_k}_1}{f^{j_\ell}_{-,\beta}}+\frac{\alpha+\beta}{\beta-\alpha}\ketbra{g^{j_k}_5}{f^{j_\ell}_{-,\alpha}}\\
=&-\frac{2\beta}{\beta-\alpha}\ketbra{g^{j_k}_1}{f^{j_\ell}_{-,\beta}}+\frac{\alpha+\beta}{\beta-\alpha}\ketbra{g^{j_k}_1}{f^{j_\ell}_{-,\alpha}}+(\alpha+\beta)\ketbra{g^{j_k}_3}{f^{j_\ell}_{-,\alpha}},
\end{split}
\end{equation}
where
\begin{equation} \label{eq:g_3_def}
 \begin{split}
 g_3^{j}(x)&=\oint\limits_{\Gamma_{1/2}}\frac{dz}{2\pi\I}\Phi(x,z)\left(\tfrac12-z\right)^{N-j}\frac{1}{z-\beta}, \\
 g_5^{j}(x)&=\oint\limits_{\Gamma_{1/2}}\frac{dz}{2\pi\I}\Phi(x,z)\left(\tfrac12-z\right)^{N-j}\frac{z+\beta}{z-\beta}\frac{z-\alpha}{2z}.
 \end{split}
\end{equation}
The last equality follows from the relation $\frac{(z-\alpha)(z+\beta)}{z-\beta}-(z+\alpha)=\frac{\beta-\alpha}{z-\beta}$.
Similarly we have
\begin{equation}
\widetilde O_{22} + \widetilde P_{22}=\oint\limits_{\Gamma_{1/2,\alpha}}\oint\limits_{\Gamma_{-\beta}}\cdots + \oint\limits_{\Gamma_{1/2,\beta}}\oint\limits_{\Gamma_{-\alpha}}\cdots.
\end{equation}
Computing the residues at $w=-\alpha$ and $w=-\beta$ we get
\begin{equation}	
\begin{split}
\widetilde{O}^{j_kj_\ell}_{22}+\widetilde{P}^{j_kj_\ell}_{22}=&\frac{2\beta}{\beta-\alpha}\ketbra{g_2^{j_k}}{f^{j_\ell}_{-,\beta}}-\frac{2\beta}{\beta-\alpha}\ketbra{f^{j_k}_{-,\beta}}{f^{j_\ell}_{-,\alpha}}-\frac{\alpha+\beta}{\beta-\alpha}\ketbra{g^{j_k}_6}{f^{j_\ell}_{-,\alpha}}\\
=&\frac{2\beta}{\beta-\alpha}\ketbra{g_2^{j_k}}{f^{j_\ell}_{-,\beta}}-\frac{\alpha+\beta}{\beta-\alpha}\ketbra{g^{j_k}_{2}}{f^{j_\ell}_{-,\alpha}}-(\alpha+\beta)\ketbra{g^{j_k}_4}{f^{j_\ell}_{-,\alpha}}-\ketbra{f^{j_k}_{-,-\alpha}}{f^{j_\ell}_{-,\alpha}},
\end{split}
\end{equation}
where
\begin{equation} \label{eq:g_4_def}
\begin{split}
g_4^{j}(x)&=\oint\limits_{\Gamma_{1/2,\pm\alpha,\beta}}\frac{dz}{2\pi\I}\frac{\Phi(x,z)}{\left(\frac12-z\right)^{N-j}}\frac{2z}{(z-\alpha)(z+\alpha)(z-\beta)},\\
g_6^{j}(x)&=\oint\limits_{\Gamma_{1/2}}\frac{dz}{2\pi\I}\frac{\Phi(x,z)}{\left(\frac12-z\right)^{N-j}}\frac{z+\beta}{(z-\beta)(z+\alpha)}.
\end{split}
\end{equation}
The last equality follows from the relation
\begin{equation}
g_6^{j_k}(x)+\frac{2\beta}{\alpha+\beta}f_{-,\beta}^{j_k}(x)=g_2^{j_k}(x)+(\beta-\alpha)g_4^{j_k}(x)+\frac{\beta-\alpha}{\alpha+\beta}f_{-,-\alpha}^{j_k}(x).
\end{equation}
A proof of the above identity can be found in Lemma 3.12 of~\cite{BFO20}.
\end{proof}

Now we can decompose the kernel
\begin{equation}
	\overline{G}=\widehat G+O
\end{equation}
with $O=J^{-1}\widetilde O$, where
\begin{equation} \label{eq:g_hat_decomp}
\widehat G^{j_kj_\ell}=\begin{pmatrix}
-\widetilde K^{j_kj_\ell}_{21} & -\widetilde K^{j_kj_\ell}_{22} - \tilde \e^{j_kj_\ell} \\ \widetilde K^{j_kj_\ell}_{11} & \widetilde K^{j_kj_\ell}_{12}
\end{pmatrix}+ \begin{pmatrix}
0 & \ketbra{(\alpha+\beta) g^{j_k}_4 + f_{-,-\alpha}^{j_k}} { f_{-,\alpha}^{j_\ell}} \\
0 & (\alpha+\beta) \ketbra{g^{j_k}_3} {f_{-,\alpha}^{j_\ell}}
\end{pmatrix}.
\end{equation}

The key property is that $O$ is of the form
\begin{equation}
O=\ket{ g_2^{j_1},g_1^{j_1},\cdots,g_2^{j_m},g_1^{j_m}}^t \bra{a_1, b_1,\ldots, a_m,b_m}.
\end{equation}

\begin{lem}\label{lemma3.14}It holds that
	\begin{equation}
	\label{eq:scalar_prod_O}
	\braket{P_sY_2}{(\Id-\breve{G}_s)^{-1}\overline{G}_s\widetilde X_2}=\braket{P_sY_2}{(\Id-\breve{G}_s)^{-1}\widehat{G}_s\widetilde X_2}.
	\end{equation}
\end{lem}
\begin{proof}
The proof is a generalization of Lemma 3.10 of~\cite{BFO20}. First recall that $\bra{Y_2}=\bra{-g_1^{j_1},g_2^{j_1},\ldots,-g_1^{j_m},g_2^{j_m}}$. Thus we have
\begin{equation}
\braket{P_sY_2}{(\Id-\breve{G}_s)^{-1}(\overline{G}_s-\widehat{G}_s)\widetilde X_2}=\braket{P_sY_2}{(\Id-\breve{G}_s)^{-1} O \widetilde X_2},
\end{equation}
which is proportional to
\begin{equation}\label{eq3.84}
\bra{-g_1^{j_1},g_2^{j_1},\ldots,-g_1^{j_m},g_2^{j_m}}
P_s(\Id-\breve{G}_s)^{-1}P_s\ket{ g_2^{j_1},g_1^{j_1},\cdots,g_2^{j_m},g_1^{j_m}}^t.
\end{equation}
We use the short-cut notation $H=P_s(\Id-\breve{G}_s)^{-1}P_s$ in the rest of this proof. The terms coming from the $2 \times 2$ $(k, \ell)$-block $H^{j_k j_\ell}$ are:
\begin{equation}
-\braket{g_1^{j_k}}{H^{j_k j_\ell}_{11} g_2^{j_\ell}}
-\braket{g_1^{j_k}}{H^{j_k j_\ell}_{12} g_1^{j_\ell}}
+\braket{g_2^{j_k}}{H^{j_k j_\ell}_{21} g_2^{j_\ell}}
+\braket{g_2^{j_k}}{H^{j_k j_\ell}_{22} g_1^{j_\ell}}.
\end{equation}
Similarly to~\eqref{eq3.31}, $H$ has some (anti)-symmetry properties. Indeed, $H^{j_k j_\ell}_{22}(x,y)=H^{j_\ell j_k}_{11}(y,x)$ and we see that the first (resp.\ fourth) term for $(k,\ell)$ cancels with the fourth (resp.\ first) term for $(\ell,k)$.

Next, we also have $H^{j_k j_\ell}_{12}(x,y)=-H^{j_\ell j_k}_{12}(y,x)$ and $H^{j_k j_\ell}_{21}(x,y)=-H^{j_\ell j_k}_{21}(y,x)$. This implies that the second (resp.\ third) term for $(k,\ell)$ cancels with the second (resp.\ third) term for $(\ell,k)$.
\end{proof}

We now check the analyticity of $\braket{P_sY_2}{(\Id - \breve{G}_s)^{-1} \widehat{G}_s \widetilde X_2}$ for \mbox{$\alpha,\beta\in (-1/2,1/2)$}. Let us first denote
\begin{equation}
 Q_2 = \widehat{G}_s \widetilde X_2
\end{equation}
and recall the following fact:
\begin{equation}
 \braket{P_sY_2}{ (\Id - \breve{G}_s)^{-1} \widehat{G}_s \widetilde X_2 } = \braket{P_sY_{2, \rm conj}}{ (\Id - \breve{G}_{s, \rm conj})^{-1} Q_{2, \rm conj} },
\end{equation}
where
\begin{equation}
 Y_{2,\rm conj}=Y_2 M, \quad Q_{2, \rm conj} =M^{-1} Q_2
\end{equation}
and $M$ is defined in~\eqref{eq:M_def}.

\begin{lem} \label{lem:Y_2_bounds}
 The vector $Y_2$ is analytic for $\alpha, \beta \in (-1/2, 1/2)$; it is independent of $\beta$ so $\lim_{\beta \to -\alpha} Y_2 = Y_2$. The following bounds hold for $\alpha, \beta \in [-1/2+\epsilon, 1/2-\epsilon]$:
 \begin{equation} \label{eq:Y_2_bounds}
 |(Y_{2, \rm conj})_{2k-1} (x)| \leq C e^{[\mu_k - (\frac{1}{2}-\frac{\epsilon}{2})] x}, \quad |(Y_{2, \rm conj})_{2k} (x)| \leq C e^{[(\frac{1}{2} - \epsilon) - \mu_k] x}, \quad 1 \leq k \leq m.
 \end{equation}
\end{lem}

\begin{proof}
 $Y_2$ is independent of $\beta$, and by taking the integration contour in each of its entries to be $|z-1/2|=\epsilon/2$, it is also analytic in $\alpha$ since the contour is bounded away from $\alpha$. Moreover the contour above gives the bounds for the odd entries of $Y_2$, while the bounds for the even entries come from the residue at $\alpha$.
\end{proof}

\begin{lem}\label{lem:Q_2_limits}
	The operator $Q_2 = \widehat G_s \widetilde X_2$ is analytic in $\alpha, \beta \in (-1/2, 1/2)$. For any $\alpha,\beta\in [-1/2+\epsilon,1/2-\epsilon]$ and any $1 \leq \ell \leq m$, we have the following bounds:
\begin{equation}\label{eq:Q_2_bounds}
\begin{split}
| (Q_{2, \rm conj})_{2 \ell-1}(y) | \leq C e^{[(\frac{1}{2}-\epsilon)-\mu_\ell] y}, \quad | (Q_{2, \rm conj})_{2 \ell}(y) | \leq C e^{[\mu_\ell - (\frac{1}{2} - \frac{\epsilon}{2})] y}.
\end{split}
\end{equation}
Moreover we have $\lim_{\beta \to -\alpha} \widehat{G} \widetilde X_2 = \lim_{\beta \to -\alpha} Q_2 = \mathsf{Q}_2$ with
	\begin{equation} \label{eq:Q_2_mathsf}
 \mathsf{Q}_2 = \left(\mathsf{h}^{\alpha,\, j_1}_1,\mathsf h^{\alpha,\, j_1}_2,\cdots,\mathsf h^{\alpha,\, j_m}_1, \mathsf h^{\alpha,\, j_m}_2\right)^t,
	\end{equation}
	where, for $1\leq \ell\leq m$,
\begin{equation} \label{eq:h_def}
	\begin{split}
	\mathsf h^{\alpha,\, j_\ell}_1 &= -\widetilde{\mathsf{K}}^{j_\ell j_1}_{22} P_{s_1}f_{+,-\alpha}^{j_1} - \tilde\varepsilon_1^{j_\ell j_1} P_{s_1}f_{+,-\alpha}^{j_1} + \mathsf{g}_4^{j_\ell} + \mathsf j^{\alpha,\, j_\ell} (s_1, \cdot), \\
	\mathsf h^{\alpha,\, j_\ell}_2 &= \widetilde{\mathsf{K}}^{j_\ell j_1}_{12} P_{s_1} f_{+,-\alpha}^{j_1} + \mathsf{g}_3^{j_\ell}
	\end{split}
	\end{equation}
with
\begin{equation} \label{eq:j_def}
	\mathsf{j}^{\alpha,\, j_\ell}(s, y) = \Id_{[y > s]} \frac{e^{\alpha s} \phi(-\alpha)}{\big(\frac 12-\alpha\big)^{N-j_1}} \left[\big(\tfrac 12+\alpha\big)^{j_\ell-j_1} \frac{\sinh \alpha(y-s)}{\alpha}+\big(\tfrac 12-\alpha\big)^{j_\ell-j_1} e^{\alpha (y-s)} (y-s)\right]
 \end{equation}
 and with our notational conventions that
 \begin{equation}
 (\mathsf{g}_3^{j_\ell}, \mathsf{g}_4^{j_\ell}, \widetilde{\mathsf{K}}_{\dots}^{\dots}) = \lim_{\beta \to -\alpha} (g_3^{j_\ell}, g_4^{j_\ell}, \widetilde{K}_{\dots}^{\dots}).
 \end{equation}
\end{lem}

\begin{proof}
	We start with
	\begin{equation}
	Q_2 = \widehat{G}_s \widetilde X_2 = \begin{pmatrix} a_1 P_{s_1}f_{+,\beta}^{j_1} \\ b_1 P_{s_1}f_{+,\beta}^{j_1} \\ \vdots\\ a_m P_{s_1} f_{+,\beta}^{j_1} \\ b_m P_{s_1}f_{+,\beta}^{j_1} \end{pmatrix},
	\end{equation}
	where the kernels $a_k, b_k$ are read from the decomposition~\eqref{eq:g_hat_decomp}, namely
	\begin{equation}\label{eq:termB_dec}
	\begin{split}
	&a_k = -\widetilde K_{22}^{j_k j_1} - \tilde\varepsilon^{j_k j_1}+ (\alpha+\beta) \ketbra{g_4^{j_k}}{f_{-,\alpha}^{j_1}} + \ketbra{f_{-,-\alpha}^{j_k}}{f_{-,\alpha}^{j_1}},\\&b_k = \widetilde K_{12}^{j_kj_1} + (\alpha+\beta) \ketbra{g_3^{j_k}} {f_{-,\alpha}^{j_1}}.
	\end{split}
 \end{equation}

 For analyticity in the parameters we argue as above. We take contours of the form $|z-1/2|=\epsilon/2$ for any $z$ contour integral around $1/2$ (and similarly for $w$ and $-1/2$). They are bounded away from $\alpha$ and $\beta$ and thus those parts are analytic. The rest of the terms are analytic by explicit inspection.

 The bounds on the even entries of $ Q_{2, \rm conj} (y) $ come from taking the $z$-integration contour to be $|z-1/2|=\epsilon/2$ in both $ \widetilde{K}_{12}^{j_k j_1} (y, u)$ and $ g_3^{j_k} (y)$ (note we integrate over $u$). The bounds on the odd entries of the same vector come from residues at $\alpha$ and $\beta$ of the integrands and functions involved. Finally, the $\mu$'s come from conjugation by $M$.

 We now turn to the limit $\beta \to -\alpha$. First we have:
	\begin{equation}\label{eq:limit_termB}
	\begin{split}
	\lim_{\beta\to -\alpha}\widetilde{K}^{j_k j_1}_{12}P_{s_1} f_{+,\beta}^{j_1} & = \widetilde{\mathsf{K}}_{12}^{j_kj_1} P_{s_1}f_{+,-\alpha}^{j_1},\\
	\lim_{\beta\to -\alpha}-(\widetilde{K}_{22}^{j_kj_1}P_{s_1}f_{+,\beta}^{j_1} + \tilde\varepsilon^{j_kj_1}_1 P_{s_1}f_{+,\beta}^{j_1}) & = -(\widetilde{\mathsf{K}}_{22}^{j_kj_1} P_{s_1}f_{+,-\alpha}^{j_1}+\tilde\varepsilon_1^{j_kj_1} P_{s_1}f_{+,-\alpha}^{j_1}).
	\end{split}
	\end{equation}
	We further compute $\braket{f_{-,\alpha}^{j_1}}{P_{s_1} f_{+,\beta}^{j_1}}$ with the result being
	\begin{equation} \label{eq:inner_alpha_beta}
	\begin{split}
	\braket{f_{-,\alpha}^{j_1}}{ P_{s_1}f_{+,\beta}^{j_1}} &= \phi(\alpha) \phi(\beta) \frac{\big(\tfrac12-\beta\big)^{N-j_1}}{\big(\tfrac12+\alpha\big)^{N-j_1}} \int_{s_1}^\infty e^{-(\alpha + \beta)x} dx\\
	&= \phi(\alpha) \phi(\beta) \frac{\big(\tfrac12-\beta\big)^{N-j_1}}{\big(\tfrac12+\alpha\big)^{N-j_1}} \frac{e^{-(\alpha + \beta) s_1}} {\alpha + \beta}=\frac{f_{-,\alpha}^{j_1}(s_1)f_{+,\beta}^{j_1}(s_1)}{\alpha+\beta}.
	\end{split}
	\end{equation}
Since $f_{-,\alpha}^{j_k},f_{+,\beta}^{j_1},g_3^{j_k},g_4^{j_k}$ are analytic and the prefactor $\alpha+\beta$ in~\eqref{eq:termB_dec} cancels with the one in~\eqref{eq:inner_alpha_beta}, we have
 \begin{equation}\label{eq:limit_termB2}
 \begin{split}
 \lim_{\beta\to -\alpha}(\alpha+\beta) g_3^{j_k} \braket{f_{-,\alpha}^{j_1}}{P_{s_1} f_{+,\beta}^{j_1}} & = \mathsf{g}_3^{j_k}, \\
 \lim_{\beta\to -\alpha}(\alpha+\beta) g_4^{j_k} \braket{f_{-,\alpha}^{j_1}}{P_{s_1} f_{+,\beta}^{j_1}} & = \mathsf{g}_4^{j_k}.
 \end{split}
\end{equation}	
From~\eqref{eq:inner_alpha_beta}, we have
\begin{equation}
f_{-,-\alpha}^{j_k}(x)\braket{f_{-,\alpha}^{j_1}}{P_{s_1} f_{+,\beta}^{j_1}} =
\frac{\phi(\beta)\big(\tfrac12-\beta\big)^{N-j_1}}{\big(\tfrac12+\alpha\big)^{N-j_1}\big(\tfrac12-\alpha\big)^{N-j_k}}
\frac{e^{\alpha x}e^{-(\alpha + \beta) s_1}}{\alpha + \beta}.
\end{equation}
For $x\leq s_1$ we have
\begin{equation}
\begin{split}
-(\tilde \e_2^{j_k j_1}P_{s_1} f_{+,\beta}^{j_1})(x) &= -\frac{\phi(\beta)\big(\tfrac12-\beta\big)^{N-j_1}}{\big(\tfrac12+\alpha\big)^{N-j_1}\big(\tfrac12-\alpha\big)^{N-j_k}} \int_{s_1}^\infty dy e^{-(y-x)\alpha} e^{-\beta y} \\
&=-\frac{\phi(\beta)\big(\tfrac12-\beta\big)^{N-j_1}}{\big(\tfrac12+\alpha\big)^{N-j_1}\big(\tfrac12-\alpha\big)^{N-j_k}}
\frac{e^{\alpha x}e^{-(\alpha + \beta) s_1}}{\alpha + \beta} = -f_{-,-\alpha}^{j_k}(x)\braket{f_{-,\alpha}^{j_1}}{P_{s_1} f_{+,\beta}^{j_1}}.
\end{split}
\end{equation}
Thus for $x\leq s_1$ it holds that
\begin{equation}
 f_{-,-\alpha}^{j_k}(x)\braket{f_{-,\alpha}^{j_1}}{P_{s_1} f_{+,\beta}^{j_1}}-(\tilde \e_2^{j_k j_1} P_{s_1} f_{+,\beta}^{j_1})(x) = 0.
\end{equation}
For $x>s_1$, we have two distinct contributions. The first is from the $\tilde \e_2^{j_k j_1} P_{s_1} f_{+,\beta}^{j_1}$ term:
\begin{equation}
\begin{split}
-(\tilde \e_2^{j_k j_1} P_{s_1} f_{+,\beta}^{j_1})(x) = &-\frac{\phi(\beta)\big(\tfrac12-\beta\big)^{N-j_1}}{\big(\tfrac12+\alpha\big)^{N-j_1}\big(\tfrac12-\alpha\big)^{N-j_k}} \int_{x}^\infty dy e^{-(y-x)\alpha} e^{-\beta y} \\
&+\frac{\phi(\beta)\big(\tfrac12-\beta\big)^{N-j_1}}{\big(\tfrac12-\alpha\big)^{N-j_1}\big(\tfrac12+\alpha\big)^{N-j_k}}
\int_{s_1}^x dy e^{-(x-y)\alpha} e^{-\beta y}\\
= &-\frac{\phi(\beta)\big(\tfrac12-\beta\big)^{N-j_1}}{\big(\tfrac12+\alpha\big)^{N-j_1}\big(\tfrac12-\alpha\big)^{N-j_k}} \frac{e^{-\beta x}}{\alpha+\beta}\\
&+\frac{\phi(\beta)\big(\tfrac12-\beta\big)^{N-j_1}e^{-\alpha x}}{\big(\tfrac12-\alpha\big)^{N-j_1}\big(\tfrac12+\alpha\big)^{N-j_k}} \frac{e^{(\alpha-\beta)x}-e^{(\alpha-\beta)s_1}}{\alpha-\beta}.
\end{split}
\end{equation}
The second is from $f_{-,-\alpha}^{j_k}(x)\braket{f_{-,\alpha}^{j_1}}{P_{s_1} f_{+,\beta}^{j_1}}$. Thus for $x>s_1$ the two terms together give
\begin{equation}
\begin{split}
f_{-,-\alpha}^{j_k}(x)\braket{f_{-,\alpha}^{j_1}}{P_{s_1} f_{+,\beta}^{j_1}}-(\tilde \e_2^{j_k j_1} f_{+,\beta}^{j_1})(x) = &\frac{\big(\tfrac12-\beta\big)^{N-j_1} \phi(\beta) e^{-\alpha x}}{\big(\tfrac12-\alpha\big)^{N-j_1}\big(\tfrac12+\alpha\big)^{N-j_k}}\frac{e^{(\alpha-\beta)x}-e^{(\alpha-\beta)s_1}}{\alpha-\beta}\\
&-\frac{\big(\tfrac12-\beta\big)^{N-j_1} \phi(\beta) e^{\alpha x}}{\big(\tfrac12+\alpha\big)^{N-j_1}\big(\tfrac12-\alpha\big)^{N-j_k}}\frac{e^{-(\alpha+\beta)x}-e^{-(\alpha+\beta)s_1}}{\alpha+\beta}.
\end{split}
\end{equation}
The $\beta\to-\alpha$ limit of the above is
\begin{equation}
\begin{split}
&\lim_{\beta\to-\alpha}f_{-,-\alpha}^{j_k}(x)\braket{f_{-,\alpha}^{j_1}}{P_{s_1} f_{+,\beta}^{j_1}}-(\tilde \e_2^{j_k j_1}P_{s_1} f_{+,\beta}^{j_1})(x)\\
& \quad =\frac{\phi(-\alpha)}{\big(\frac 12-\alpha\big)^{N-j_1}}\Big[\big(\tfrac 12+\alpha\big)^{j_k-j_1}e^{-\alpha x}\frac{e^{2\alpha x}-e^{2\alpha s_1}}{2\alpha}+\big(\tfrac 12-\alpha\big)^{j_k-j_1}e^{\alpha x}(x-s_1)\Big]\\
& \quad = \frac{\phi(-\alpha)}{\big(\frac 12-\alpha\big)^{N-j_1}}\Big[\big(\tfrac 12+\alpha\big)^{j_k-j_1}e^{\alpha s_1}\frac{\sinh \alpha(x-s_1)}{\alpha}+\big(\tfrac 12-\alpha\big)^{j_k-j_1}e^{\alpha x}(x-s_1)\Big],
\end{split}
\end{equation}
which is our $\mathsf{j}$ function.
\end{proof}

We now put everything together.

\begin{lem} \label{lem:term_B}
The term $\braket{P_sY_2}{ (\Id - \breve{G}_s)^{-1} \widehat{G}_s \widetilde X_2 } = \braket{P_sY_{2, \rm conj}}{ (\Id - \breve{G}_{s, \rm conj})^{-1} Q_{2, \rm conj} }$ is analytic for $\alpha,\beta\in(-1/2,1/2)$. Its $\beta \to-\alpha$ limit is given by:
\begin{equation}
\lim_{\beta\to-\alpha} \braket{P_sY_2}{ (\Id - \breve{G}_s)^{-1} \widehat{G}_s \widetilde X_2 } = \braket{P_s Y_2}{ (\Id - \breve{\mathsf G}_s)^{-1} \mathsf{Q}_{2}}
\end{equation}
with $\mathsf{Q}_2$ defined in~\eqref{eq:Q_2_mathsf} and with $\breve{\mathsf G}_s = \lim_{\beta \to -\alpha} \breve{G}_s$.
\end{lem}

\begin{proof}
Analyticity of the whole inner product follows from the analyticity of all of the different entries in it, together with the bounds on the conjugated kernel and functions obtained above in~\eqref{eq3.65} (see also~\eqref{eq3.31}), \eqref{eq:Y_2_bounds}, and~\eqref{eq:Q_2_bounds}.

We see that the various products are bounded by functions which decay exponentially as $x, y \to \infty$. E.g.~the bound on the integrand inside the scalar product in the even $2k$ summands ($1 \leq k \leq m$) is, up to constants, $e^{-\epsilon x/2}$ as $x \to \infty$. Informally speaking, this case corresponds to the integral $\int_{s_k}^\infty (Y_{2, \rm conj})_{2k} (x) (Q_{2, \rm conj})_{2k} (x) dx$ coming from the identity term in the Neumann expansion of $(\Id-\breve{G}_{s, \rm conj})^{-1}$. The same is true for the odd summands of the overall inner product. These bounds allow us to pass the $\beta \to -\alpha$ limit inside each integral yielding the result.
\end{proof}

\subsubsection{Analyticity of term C}

Let us begin by computing the vector $U_2 = W (\Id-P_s) \widetilde{X}_2$.
\begin{lem}
 The $2m \times 1$ vector $U_2 = W (\Id-P_s) \widetilde{X}_2$ is analytic for $\alpha, \beta \in (-1/2, 1/2)$ having the following explicit form:
 \begin{equation}
 (U_2)_{\ell}(x) =
 \begin{dcases}
 0, & \text{if } \ell=2 \text{ or $\ell$ is odd},\\
  \Id_{[x \geq s_1]} f^{j_1}_{+, \beta} (s_1) \oint\limits_{\Gamma_{1/2}} \frac{dz}{2 \pi \I} \frac{e^{-(x-s_1)z}}{(\tfrac12-z)^{j_k-j_1} (z-\beta)} - \Id_{[x < s_1]} f^{j_k}_{+, \beta} (x), & \text{if } \ell=2k>2.
 \end{dcases}
 \end{equation}
\end{lem}

\begin{proof}
 Let us begin by fixing a small positive $\epsilon$ and picking $\beta \in [-1/2+\epsilon, 1/2-\epsilon]$. From the explicit (and sparse) form of $W$, it follows immediately that the $\ell$-th component $(U_2)_\ell$ is zero unless $\ell = 2k > 2$ for some $k$. For that case we can perform the product explicitly:
 \begin{equation}
 \begin{split}
 (U_2)_{2k}(x) &= \int_{-\infty}^{s_1} V^{j_k j_1} (x, y) f^{j_1}_{+, \beta} (y) dy \\
 &= -\phi(\beta) (\tfrac12-\beta)^{N-j_1} \int_{-\infty}^{s_1} dy e^{-y \beta} \int_{\I \R + \beta + \eta} \frac{dz}{2 \pi \I} \frac{e^{-(x-y)z}}{(\tfrac12-z)^{j_k-j_1}} \\
 &= - f^{j_1}_{+, \beta} (s_1) \int_{\I \R + \beta + \eta} \frac{dz}{2 \pi \I} \frac{e^{-(x-s_1)z}}{(\tfrac12-z)^{j_k-j_1} (z-\beta)} \\
 &= \Id_{[x \geq s_1]} f^{j_1}_{+, \beta} (s_1) \oint\limits_{\Gamma_{1/2}} \frac{dz}{2 \pi \I} \frac{e^{-(x-s_1)z}}{(\tfrac12-z)^{j_k-j_1} (z-\beta)} - \Id_{[x < s_1]} f^{j_k}_{+, \beta} (x),
 \end{split}
 \end{equation}
for some small $0 < \eta \ll 1$. For the second equality we used that
\begin{equation}
 V^{j_k j_1} (x, y) = -\Id_{[j_k > j_1]} \int_{\I \R} \frac{dz}{2 \pi \I} \frac{e^{-(x-y)z}}{(\tfrac12-z)^{j_k-j_1}}
\end{equation}
and $\Id_{[j_k > j_1]} = 1$ as $k > 1$. We can shift the contour to the right so it sits between $\beta$ and $1/2$ (becoming $\I \R + \beta + \eta$ for small positive $\eta$). This ensures $\Re(z) > \beta$ and we can then explicitly perform the $dy$ integral obtaining the third equality. We can then close the vertical contour at $\infty$ if $x \geq s_1$ (and at $-\infty$ otherwise) to pick the two terms in the fourth equality. The first summand changes sign due to reversing the orientation of the clockwise contour into counter-clockwise $\Gamma_{1/2}$, while the second summand is just the residue of the integrand at $z = \beta$ upon rearranging some factors. The end result is clearly analytic in $\beta$ since we can fix the $\Gamma_{1/2}$ close enough to $1/2$ ($|z-1/2|=\epsilon/2$ would do). It is also independent of (and hence analytic in) $\alpha$.
\end{proof}

To continue, we wish to recall the following fact:
\begin{equation}
 \braket{P_sY_2}{(\Id-\breve{G}_s)^{-1} P_s W (\Id-P_s) \widetilde{X}_2} = \braket{P_sY_{2, \rm conj}}{(\Id-\breve{G}_{s, \rm conj})^{-1} P_s U_{2, \rm conj}},
\end{equation}
where
\begin{equation}
 Y_{2,\rm conj}=Y_2 M, \quad U_{2,\rm conj}=M^{-1} U_2
\end{equation}
and $M$ is defined in~\eqref{eq:M_def}.

\begin{lem} \label{lem:lim_U_2}
For $\alpha, \beta \in [-1/2 + \epsilon, 1/2 - \epsilon]$ we have the following bounds:
\begin{equation}\label{eq:U_2_bounds}
 |(U_{2, \rm conj})_{2k} (y)| \leq C e^{[\mu_k - (\frac{1}{2} - \frac{\epsilon}{2})] y}, \quad 1 < k \leq m, \quad (U_{2, \rm conj})_\ell(y)=0 \quad \text{ otherwise}.
\end{equation}
Let us write $\mathsf{U}_2 = \lim\limits_{b \to -\alpha} U_2$. We then have, for any $\eta\in (-\alpha,1/2)$,
\begin{equation}
 \begin{split}
 (\mathsf{U}_2)_{2k}(y) &= - f^{j_1}_{+, -\alpha} (s_1) \int_{\I \R +\eta} \frac{dz}{2 \pi \I} \frac{e^{-(y-s_1)z}}{(\tfrac12-z)^{j_k-j_1} (z+\alpha)} \\
 &= \Id_{[y \geq s_1]} f^{j_1}_{+, -\alpha} (s_1) \oint\limits_{\Gamma_{1/2}} \frac{dz}{2 \pi \I} \frac{e^{-(y-s_1)z}}{(\tfrac12-z)^{j_k-j_1} (z+\alpha)} - \Id_{[y < s_1]} f^{j_k}_{+, -\alpha} (y),
 \end{split}
\end{equation}
 for $1 < k \leq m$ and $(\mathsf{U}_2)_{\ell}(y) = 0$ otherwise.
\end{lem}

\begin{proof}
 The bound comes from taking the integration contour in the even entries $2k > 2$ of $U_2$ to be $|z-1/2|=\epsilon/2$ (note we are only interested in the unbounded regime $y \geq s_1$). By dominated convergence we can pass the limit $\beta \to -\alpha$ inside the integral to obtain the stated result.
\end{proof}

We now put everything together.

\begin{lem} \label{lem:term_C}
The term $\braket{P_sY_2}{(\Id-\breve{G}_s)^{-1} P_s W (\Id-P_s) \widetilde{X}_2} = \braket{P_sY_2}{(\Id-\breve{G}_s)^{-1} P_s U_2}$ is analytic in $\alpha,\beta\in(-1/2,1/2)$. Its $\beta\to-\alpha$ limit is given by:
\begin{equation}
\lim_{\beta\to-\alpha} \braket{P_sY_2}{(\Id-\breve{G}_s)^{-1} P_s U_2} = \braket{P_s Y_2}{(\Id-\breve{\mathsf G}_s)^{-1} P_s \mathsf{U}_2}.
\end{equation}
\end{lem}

\begin{proof}
 Analyticity of the whole inner product follows from the analyticity of all of the different entries in it and the bounds ensuring the scalar product to be well-defined. Fixing $\epsilon >0$ and putting together the bounds of~\eqref{eq3.65}, \eqref{eq:Y_2_bounds}, and~\eqref{eq:U_2_bounds}, we see that the various products are bounded by functions exponentially decaying at infinity. For example, the bound on the integrand inside the scalar product in the $2k$-th summand ($k>1$) is, up to constants, $e^{-\epsilon x/2}$ as $x \to \infty$. Thus we can pass the $\beta \to -\alpha$ limit inside each integral yielding the result.
\end{proof}

\subsection{Proof of Theorem~\ref{thm:main_finite}}

\begin{proof}[Proof of Theorem~\ref{thm:main_finite}]
 We start from the integrable model with parameters $\alpha \in (-1/2, 1/2)$ and $\beta \in (0, 1/2)$. Theorem~\ref{thm:exp_corr} then gives a formula for the multipoint distribution of the integrable (Pfaffian) LPP times, but with $\beta>0$. By the shift argument, Lemma~\ref{lem:shift}, we can remove the random variable $\tilde{\omega}_{1,1}$ at the origin. By the analytic continuation contained in: Proposition~\ref{prop:cvgOfFredPf} (for the Fredholm Pfaffian), Lemma~\ref{lem:term_A} (for term A), Lemma~\ref{lem:term_B} (for term B), and Lemma~\ref{lem:term_C} (for term C) we have that: (a) the model is indeed well-defined for any $\alpha, \beta \in(-1/2, 1/2)$, (b) all the terms are furthermore analytic for $\alpha$ and $\beta$ in the described range, and (c) we can take the $\beta \to -\alpha$ limit to obtain the multipoint distribution of the stationary model.

 This implies the result save for the complicated notation we have used in Section~\ref{sec:finite_proof}. We now explain the translation from the notation of Section~\ref{sec:finite_proof} (left below) to the cleaner less cumbersome notation of Section~\ref{sec:results_finite} (right below):
 \begin{equation}
 \begin{split}
 (\mathsf{Q_2}, \mathsf{U}_2, Y_2) &\longrightarrow (\mathsf{Q}, \mathsf{U}, \mathsf{Y}), \\
 (\varepsilon, V) &\longrightarrow (\mathsf{E}, \mathsf{V}), \\
 (g_1^{\dots}, g_2^{\dots}, f_{\dots}^{\dots}) &\longrightarrow (\mathsf{g}_1^{\dots}, \mathsf{g}_2^{\dots}, \mathsf{f}_{\dots}^{\dots}).
 \end{split}
 \end{equation}
 This finishes the proof of Theorem~\ref{thm:main_finite}.
\end{proof}

\section{Asymptotic analysis: proof of Theorem~\ref{thm:main_asymptotics}} \label{sec:proof_asymptotics}

In this section we prove Theorem~\ref{thm:main_asymptotics}, our main asymptotic result.

Let us fix an integer $m \geq 1$, $\delta \in \R$, $m$ real numbers $S_1, \dots, S_m$, and $m$ \emph{ordered} non-negative real numbers $u_1 > u_2 > \dots > u_m \geq 0$. As anticipated in Section~\ref{sec:results_asymptotics} we consider the scaling
\begin{equation}
 \begin{split}
 \alpha &= \delta 2^{-4/3} N^{-1/3}, \\
 s_k & =4 N - 2 u_k 2^{5/3}N^{2/3} + S_k\, 2^{4/3}N^{1/3}, \\
 N - j_k &= u_k 2^{5/3} N^{2/3}
 \end{split}
\end{equation}
and we find it convenient at times to use $s$ ($S$) for one of the $s_k$'s ($S_k$'s) generically, $(j, j')$ for a generic pair $(j_k, j_\ell)$, and $(u, v)$ for a generic pair $(u_k, u_\ell)$. Therefore we have
\begin{equation}
 \begin{split}
 s &= 4 N - 2 u 2^{5/3}N^{2/3} + S\, 2^{4/3}N^{1/3},\\
 (N - j, N-j') &= (u, v) 2^{5/3} N^{2/3}.
 \end{split}
\end{equation}

Accordingly, in the functions and/or kernels, we need to scale $x,y$ as
\begin{equation}
\begin{split}
 x &=4 N -2 u 2^{5/3}N^{2/3} + X 2^{4/3} N^{1/3},\\
 y &=4 N -2 v 2^{5/3}N^{2/3} + Y 2^{4/3} N^{1/3},
\end{split}
\end{equation}
while in the integrals we will consider the change of variables
\begin{equation}
z=\zeta/(2^{4/3}N^{1/3}), \quad w=\omega/(2^{4/3} N^{1/3}).
\end{equation}
Observe from the Fredholm expansion\footnote{Here $I_k = (s_k, \infty)$ and $[K^{(n)} (j_{i_a}, x_a; j_{i_b}, x_b)]_{1 \leq a, b \leq n}$ is the skew-symmetric $2n \times 2n$ matrix with $2 \times 2$ block at $(a, b)$ given by the matrix kernel $K (j_{i_a}, x_a; j_{i_b}, x_b)$, $1 \leq a, b \leq m$.}
 \begin{equation}
 \pf (J - P_s \breve{\mathsf K} P_s) = \sum_{n=0}^{\infty} \frac{(-1)^n}{n!} \sum_{i_1, \dots, i_n = 1}^m \int_{I_{i_1} \times \cdots \times I_{i_n}} \pf [\breve{\mathsf K}^{(n)} (j_{i_a}, x_a; j_{i_b}, x_b)]_{1 \leq a, b \leq n} \prod_{a=1}^n d x_a
 \end{equation}
	that the Pfaffian has to be multiplied by the volume element $(2^{4/3}N^{1/3})^n$. This implies that elements of each block of the Pfaffian kernel have to be rescaled and conjugated as follows:
\begin{equation}
\begin{split}
\mathcal{K}_{11}^{u v ,\, \rm resc}(X,Y) &= (2^{4/3}N^{1/3})^2 2^{2N-(j+j')} \mathcal{K}^{j j'}_{11} (x,y),\\
\mathcal{K}_{12}^{u v ,\, \rm resc}(X,Y) &= 2^{4/3} N^{1/3} 2^{j'-j} \mathcal{K}^{j j'}_{12} (x,y),\\
\mathcal{K}_{22}^{u v ,\, \rm resc}(X,Y) &= 2^{-2N+j+j'} \mathcal{K}^{j j'}_{22} (x,y),
\end{split}
\end{equation}
where $\mathcal{K} \in\{\breve{\mathsf{K}}, \widetilde{\mathsf{K}}\}$.

Similarly we set ${\mathsf E}^{u v,\, \rm resc}_k(X,Y)=2^{-2N + j + j'}{\mathsf E}^{j j'}_{k}(x,y)$ with $k=0,1$ or empty. Now we rescale the functions
\begin{equation}
 \mathsf{f}_{+}^{-\delta,\, -u,\, {\rm resc}}(X) = 2^{N-j} \mathsf f^j_{+,-\alpha}(x), \quad
 \mathsf{e}^{\delta,\, u,\, \rm resc}(X) = 2^{-4/3} N^{-1/3} \mathsf{e}^{\alpha,\, j}(x),
\end{equation}
as well as
\begin{alignat}{2}
\mathsf{g}^{\delta,\, u,\, \rm resc}_1(X) &= 2^{4/3}N^{1/3}2^{j-N} \mathsf g^{j}_1(x), \quad & \mathsf{g}^{\delta,\, u,\, \rm resc}_2 (X) &=2^{j-N} \mathsf{g}^{j}_2(x) , \nonumber \\
\mathsf{g}^{\delta,\, u,\, \rm resc}_3(X) &= 2^{4/3}N^{1/3}2^{j-N} \mathsf{g}^{j}_3(x), \quad & \mathsf{g}^{\delta,\, u,\, \rm resc}_4(X) &= 2^{j-N} \mathsf{g}^{j}_4(x)
\end{alignat}
and
\begin{equation}
 \begin{split}
 \mathsf{j}^{\delta,\, u ,\, \rm resc} (S, X) &= 2^{-4/3} N^{-1/3} 2^{j-N} \mathsf j^{\alpha,\, j }(s,x), \\
 \mathsf{h}^{\delta,\, u ,\, \rm resc}_1 (X) &= 2^{-4/3} N^{-1/3} 2^{j-N} \mathsf h^{\alpha,\, j }_1(x), \\
 \mathsf{h}^{\delta,\, u ,\, \rm resc}_2 (X) &= 2^{N - j}  \mathsf h^{\alpha,\, j }_2(x)
 \end{split}
\end{equation}
and finally
\begin{equation}
 \begin{split}
 \mathsf{U}_{2k}^{\delta,\, u_k,\, \rm resc} (Y) &= 2^{4/3}N^{1/3}2^{k-N}  \mathsf{U}_{2k} (y), \quad 1 < k \leq m, \\
 \mathsf{U}_{\ell}^{\delta,\, u_k,\, \rm resc} (Y) &= 0, \quad \text{otherwise}.
 \end{split}
\end{equation}

The functions and the kernels above are similar to the functions of~\cite[Section~2.3]{BFO20} (for the one-point half-space stationary case) and to those of~\cite{BFP09} (for the full-space multipoint stationary case). The analysis is mostly very similar as well. For these reasons we are not going to repeat all the details of the asymptotic analysis, but only point out the relevant differences.

The limits of the functions entering in the statement of Theorem~\ref{thm:main_asymptotics} are the following.

\begin{lem}\label{lem:AsymptoticsFunctions}
	For any given $L>0$, the following limits hold uniformly for $X\in [-L,L]$:
	\begin{equation}
\begin{aligned}
&\lim_{N\to\infty} \mathsf{f}^{-\delta,\,-u,\, \rm resc}_+(X)=\mathpzc{f}^{-\delta,\, -u}(X),\\
&\lim_{N\to\infty} \mathsf{e}^{\delta,\, u,\, \rm resc}(X)=\mathpzc{e}^{\delta,\, u} (X),\\
&\lim_{N\to\infty}\mathsf{j}^{\delta,\, u,\, \rm resc}(S, X) =\mathpzc{j}^{\delta,\, u}(S,X),
\end{aligned}
	\end{equation}
 as well as
 \begin{equation}
 \lim_{N \to \infty} \mathsf{g}_a^{\delta,\, u,\, \rm resc}(X) = \mathpzc{g}_a^{\delta,\, u} (X),\quad 1 \leq a \leq 4.
 \end{equation}
 Furthermore, for any $X\geq -L$, we have the following bounds which hold uniformly in $N$:
\begin{equation}
 \begin{split}
 |\mathsf{f}^{-\delta,\, -u, \rm resc}_+ (X)| &\leq C e^{\delta X},\\
 |\mathsf{j}^{\delta,\, u, \rm resc}(S,X)| &\leq C |X|e^{|\delta X|},
 \end{split}
\end{equation}
for some constant $C$. Finally, any $\kappa>0$ we have
\begin{equation}
\begin{split}
|\mathsf{g}_1^{\delta,\, u,\, \rm resc}(X)| &\leq C e^{-\kappa X}, \\
|\mathsf{g}_2^{\delta,\, u,\, \rm resc}(X)| &\leq C (e^{-\delta X}+ e^{-\kappa X}),\\
|\mathsf{g}_3^{\delta,\, u,\, \rm resc}(X)| &\leq C e^{-\kappa X}, \\
|\mathsf{g}_4^{\delta,\, u,\, \rm resc}(X)| &\leq C (|X|e^{|\delta X|}+e^{-\kappa X}).
\end{split}
\end{equation}
\end{lem}

\begin{proof}
 The proof is the same as~\cite[Lemma 31]{BFO20} following computations similar to those of~\cite[Lemma 4.6]{BFP09}.
\end{proof}

The limits of the kernels are the following.

\begin{lem}\label{lem:AsymptoticsKernels}
	For any given $L>0$ the following limits hold uniformly for $X,Y\in [-L,L]$:
	\begin{equation}
	\lim_{N \to \infty} \breve{\mathsf{K}}^{uv,\, \rm resc}_{a b}(X,Y)= \breve{\mathcal{A}}^{uv}_{ab} (X, Y),\quad a,b\in\{1,2\}.
	\end{equation}
 Furthermore, for any $X,Y\geq -L$ and $\kappa > 0$, we have the following bounds which hold uniformly in $N$:
\begin{equation}
 \begin{split}
 |\breve{\mathsf{K}}^{uv,\, \rm resc}_{11}(X,Y)| &\leq C e^{-\kappa (X+Y)}, \\
 |\breve{\mathsf{K}}^{uv,\, \rm resc}_{12}(X,Y)| &\leq C (e^{-\kappa (X+Y)}+e^{-\kappa X} e^{\delta Y}) + |\mathcal{V}^{uv,\, \rm resc} (X, Y)|, \\
 |\breve{\mathsf{K}}^{uv,\, \rm resc}_{21}(X,Y)| &\leq C (e^{-\kappa (X+Y)}+e^{\delta X} e^{-\kappa Y}) + |\mathcal{V}^{vu,\, \rm resc} (X, Y)|, \\
 |\breve{\mathsf{K}}^{uv,\, \rm resc}_{22}(X,Y)| &\leq |{\mathcal E}^{uv,\, \rm resc}(X,Y)|+C (e^{-\kappa X}e^{\delta Y}+e^{\delta X} e^{-\kappa Y})
 \end{split}
\end{equation}
and
\begin{equation}
 \begin{split}
 |\mathcal{V}^{uv,\, \rm resc} (X, Y)| &\leq \Id_{[u < v]} C e^{-|X-Y|}, \\
 |{\mathcal E}^{uv,\, \rm resc}_0(X,Y)| &\leq C e^{\delta |X-Y|}, \\
 |{\mathcal E}^{uv,\, \rm resc}_1(X,Y)| &\leq C e^{-(|\delta|+\kappa) |X-Y|},
 \end{split}
\end{equation}
 for some constant $C$.
\end{lem}

\begin{proof}
 The proof is similar to~\cite[Lemma~32]{BFO20}. Furthermore the asymptotics of the double integrals and the uniform bounds follow the same arguments as in~\cite[Lemma~4.4]{BFP09} and~\cite[Lemma~4.5]{BFP09}. For the bounds, we do the computation in two steps: first we calculate explicitly the values of the poles at $\pm \alpha$ if inside the integration contours; for the rest we have Airy-like super-exponential decay in both variables yielding the terms $e^{-\kappa X}$ and $e^{-\kappa Y}$. For ${\mathcal{E}}^{uv,\, \rm resc}_1(X,Y)$, we can take the $\zeta$ contour to pass on the right of $|\delta|$ at distance $\kappa$; it can be deformed to become vertical while still keeping the integral convergent since we also have the quadratic term in $\zeta$. For $\mathcal{V}^{uv,\, \rm resc} (X, Y)$ the bound is immediate from its explicit form.
 \end{proof}

To obtain the limits of $\mathsf h_1^{\delta,\, u, \,\rm resc}$ and $\mathsf h_2^{\delta,\, u,\, \rm resc}$, we need the limits of $\widetilde{\mathsf{K}}^{u v,\, \rm resc}_{12}$ and $\widetilde{\mathsf{K}}^{u v,\, \rm resc}_{22}$ when applied to $v=u_1$.

\begin{lem}\label{lem:AsymptoticsKernelsTilde}
	For any given $L>0$, the following limits hold uniformly for $X, Y\in [-L,L]$:
	\begin{equation}
	\lim_{N\to\infty} \widetilde{\mathsf{K}}^{u v,\, \rm resc}_{12}(X,Y)=\widetilde{\mathcal{A}}^{u v }_{12} (X, Y),\quad
	\lim_{N\to\infty} \widetilde{\mathsf{K}}^{u v,\, \rm resc}_{22}(X,Y)=\widetilde{\mathcal{A}}^{u v }_{22} (X, Y).
	\end{equation}
 Furthermore, for any $X,Y\geq -L$ and $\kappa > 0$, we have the following bounds which hold uniformly in $N$:
 \begin{equation}
 \begin{split}
 |\widetilde{\mathsf{K}}^{u v,\, \rm resc}_{12}(X,Y)| &\leq C e^{-\kappa (X+Y)}, \\
 |\widetilde{\mathsf{K}}^{u v,\, \rm resc}_{22}(X,Y)| &\leq C (e^{-\kappa X}+e^{\delta X}) e^{-\kappa Y},
 \end{split}
 \end{equation}
 for some constant $C$.\end{lem}

\begin{proof}
	The proof of Lemma~\ref{lem:AsymptoticsKernels} applies \emph{mutatis mutandis}.
\end{proof}

\begin{cor}\label{cor:AsymptoticsFredholmPfaffian}
	For any positive integer $m$ and any $S_1, \dots, S_m \in \R$, we have
	\begin{equation}
	\lim_{N\to\infty} \pf(J - P_s \breve{\mathsf{K}}^{\rm resc} P_s) = \pf(J - P_S \breve{\mathcal{A}} P_S).
	\end{equation}
\end{cor}

\begin{proof}
 Let us expand the Fredholm expansion as its defining series. By considering $\kappa > |\delta|$, we can use the bounds from Lemma~\ref{lem:AsymptoticsKernels} and dominated convergence to exchange the summation/integration with the limit $N\to\infty$. This yields the desired result.
 \end{proof}

 \begin{cor}\label{cor:AsymptoticsHs}
 For any given $L>0$, we have the following limits which hold uniformly in $Y\in [-L,L]$:
 \begin{equation}
 \lim_{N\to\infty} \mathsf h_1^{\delta,\, u ,\, \rm resc}(Y) = \mathpzc{h}^{\delta,\, u }_1(Y),\quad
 \lim_{N\to\infty} \mathsf h_2^{\delta,\, u ,\, \rm resc}(Y) = \mathpzc{h}^{\delta,\, u }_2(Y).
 \end{equation}
 Furthermore, for any $Y\geq -L$ and $\kappa > 0$, we have:
 \begin{equation}
 |\mathsf{h}_1^{\delta,\, u,\, \rm resc}(Y)|\leq C |Y| e^{|\delta Y|},\quad |\mathsf{h}_2^{\delta,\, u,\, \rm resc}(Y)|\leq C e^{-\kappa Y},
 \end{equation}
 for some constant $C$, uniformly in $N$.
 \end{cor}

 \begin{proof}
 It is the same as~\cite[Corollary 35]{BFO20}.
 \end{proof}

 \begin{cor}\label{cor:AsymptoticsUs}
 For any given $L>0$, we have the following limits which hold uniformly in $Y\in [-L,L]$:
 \begin{equation}
 \lim_{N\to\infty} \mathsf{U}_{2k}^{\delta,\, u_k,\, \rm resc} (Y) = \mathcal{U}_{2k} (Y), \quad 1 < k \leq m
 \end{equation}
 with all other components identically zero. This means we have, as $2m \times 1$ vectors, $\lim_{N\to\infty} \mathsf{U}^{\delta,\, u_k,\, \rm resc} = \mathcal{U}$. Furthermore, for any $Y\geq -L$ and $\kappa > 0$, we have:
 \begin{equation}
 |\mathsf{U}_{2k}^{\delta,\, u_k,\, \rm resc} (Y)| \leq C e^{|\delta Y|}, \quad 1 < k \leq m,
 \end{equation}
 for some constant $C$, uniformly in $N$.
 \end{cor}

 \begin{proof}
 The proof is similar to that of Corollary~\ref{cor:AsymptoticsHs} and almost identical to~\cite[Lemma 4.9]{BFP09}.
 \end{proof}

Given all of the above, we now prove Theorem~\ref{thm:main_asymptotics}, our main asymptotic result.

\begin{proof}[Proof of Theorem~\ref{thm:main_asymptotics}]
	The convergence of the distribution follows from direct application of the previous lemmas and corollaries: Lemma~\ref{lem:AsymptoticsFunctions} for the functions $\mathsf g_1^{\delta,\, u,\,\rm resc}(X)$ and $\mathsf g_2^{\delta,\, u,\,\rm resc}(X)$, components of the vector $\mathsf Y$; the same lemma gives also bounds and limits for the rest of the functions entering in the right hand side of~\eqref{eq:main_finite}; Corollary~\ref{cor:AsymptoticsFredholmPfaffian} and Lemma~\ref{lem:AsymptoticsKernels} for the Fredholm Pfaffian; Lemma~\ref{lem:AsymptoticsKernelsTilde} and Corollary~\ref{cor:AsymptoticsHs} for the functions $\mathsf h_1^{\delta,\,u,\,\rm resc}(Y)$ and $\mathsf h_2^{\delta,\, u,\,\rm resc}(Y)$ entering in the vector $\mathsf Q$; and finally Corollary~\ref{cor:AsymptoticsUs} for the vector $\mathsf{U}$. As a consequence of the aforementioned bounds we can take the $N\to\infty$ limit inside the integrals by dominated convergence.
	
 Furthermore, the derivatives in the variables $s_k$ ($1 \leq k \leq m$) pose no issues as they produce only polynomial factors and we have exponential bounds for every term (after possible conjugation). This again allows us to use dominated convergence for the corresponding derivatives.

 As was the case in~\cite{BFO20}, the inverse operator also poses no problems, since for any $s \in \{s_1, \dots, s_k \}$ the derivative yields
 \begin{equation}
 \partial_s(\Id-J^{-1}\breve{\mathsf{K}})^{-1}=(\Id-J^{-1}\, \breve{\mathsf{K}})^{-1} J^{-1}\partial_s\breve{\mathsf{K}}\, (\Id-J^{-1}\breve{\mathsf{K}})^{-1}
 \end{equation}
 and, once multiplied by the Fredholm Pfaffian, the resolvent can be rewritten as a linear combination of two Fredholm Pfaffians. This last observation leads to the claimed result.
\end{proof}

\section{Limit to the Airy$_{\rm stat}$ process: proof of Theorem~\ref{thm:LimitAiryStat}} \label{sec:proof_Airy_stat}

In order to recover the kernel of the Airy$_{\rm stat}$ process, we need to consider the $\delta\to -\infty$ limit after the following replacements:
\begin{equation}
u_i=-\tau_i-\delta, \quad S_i=s_i+\delta(2u_i+\delta), \quad 1 \leq i \leq m,
\end{equation}
for some fixed ordered times $\tau_1 < \dots < \tau_m$ and numbers $s_1, \dots, s_m \in \R$.

To prove Theorem~\ref{thm:LimitAiryStat} we will have to undo some of the complexity of Section~\ref{sec:results_asymptotics} (which in turn comes from the complexity of Section~\ref{sec:results_finite}).

We first note that the complicated decomposition of the kernel $\breve{\cal A}^{uv}(x,y)$ that appears in~\ref{sec:results_asymptotics} was necessary so that things are convergent for generic values of $\delta$. However here we are interested in $\delta<0$ and $u,v>0$; in this case the situation simplifies a lot. it can be written as a single double integral.
\begin{lem}\label{lem:AiryStatA22term}
For $\delta<0$ and $u,v>0$ we have
\begin{equation}
\breve{\mathcal{A}}^{u v}_{22} (X, Y)=- \int_{\I\R+\eta_1} \frac{d \zeta}{2\pi\I} \int_{\I\R+\eta_2} \frac{d \omega}{2\pi\I} \frac{ e^{\frac{\zeta^3}{3} + \zeta^2 u - \zeta X} }{ e^{\frac{\omega^3}{3} - \omega^2 v - \omega Y} } \frac{1}{\zeta - \omega} \left(\frac{1}{ \zeta + \delta}+\frac{1}{\omega-\delta}\right),
\end{equation}
where $\max\{\delta,\, -u\}<\eta_1<\eta_2<\min\{v,-\delta\}$.
\end{lem}
\begin{proof}
First notice that the two contours in the double integral term of $\breve{\mathcal{A}}^{u v}_{22}$ in \eqref{eq:227} can be taken to be the equal provided $\delta<0$. Furthermore, the contours can be deformed to become vertical provided $-u<\Re(\zeta)$ and $\Re(\omega)<v$. By doing this, we have vertical contours with $\Re(\omega)<\Re(\zeta)$. Exchanging the contours so that in the end they satisfy $\max\{\delta,\, -u\}<\Re(\zeta)<\Re(\omega)<\min\{v,-\delta\}$, we pick up a residue. This latter equals $-\mathcal{E}^{uv}(X,Y)$.
\end{proof}

We next need two identities used to undo the steps of Lemma~\ref{lemma3.14} after the limit. That result was crucial for the general case; for pre-limit $\beta>0$ (corresponding post-limit to $\delta < 0$) said result is not necessary.

\begin{lem}\label{Lemma:5.2}
It holds that
\begin{equation}\label{eq5.3}
\widetilde{\mathcal{A}}^{u v}_{22} (X, Y)+{\cal E}^{uv}_1(X,Y)=\breve{\mathcal{A}}^{u v}_{22} (X, Y) - \mathpzc{g}_2^{\delta,\, u}(X) \mathpzc{f}^{-\delta,-v}(Y)+\Id_{[X>Y]} e^{\delta^2(u+v)}(e^{\delta(X-Y)}+e^{-\delta(X-Y)}).
\end{equation}
Moreover, for $v=u_1$, it holds that
\begin{equation}\label{eq5.4}
\widetilde{\mathcal{A}}^{u v}_{12}(X,Y)=\breve{\mathcal{A}}^{u v}_{12}(X,Y)+\mathpzc{g}_1^{\delta,\, u}(X) \mathpzc{f}^{-\delta,\, -v}(Y)
\end{equation}
and that
\begin{equation}\label{eq5.5}
\mathpzc{j}^{\delta,\, u}(S_1,X)=\int_{S_1}^\infty dY \Id_{[X>Y]}e^{\delta^2(u+v)}(e^{\delta(X-Y)}+e^{-\delta(X-Y)}) \mathpzc{f}^{-\delta,\, v}(Y).
\end{equation}
\end{lem}

\begin{proof}
We start with identity for the $22$ entry. Let us recall
\begin{equation}
\widetilde{\mathcal{A}}^{uv}_{22} (X, Y)= \int\limits_{\zcd\, {}_{-\delta}} \frac{d \zeta}{2\pi\I} \int\limits_{\wcu\, {}_{\delta,\zeta}} \frac{d \omega}{2\pi\I} \frac{ e^{\frac{\zeta^3}{3} + \zeta^2 u - \zeta X} } { e^{\frac{\omega^3}{3} - \omega^2 v - \omega Y} } \frac{1}{\zeta - \omega}\left(\frac{1}{\zeta + \delta}++\frac{1}{\omega - \delta}\right).
\end{equation}
By moving the integration contour $\omega$ to pass to the right of $\delta$, we get the double integral term in $\breve{\mathcal{A}}^{uv}_{22}$ and a correction term given by subtracting the pole at $\omega=\delta$. Thus we have
\begin{equation}
\widetilde{\mathcal{A}}^{uv}_{22} (X, Y) = \breve{\mathcal{A}}^{uv}_{22} (X, Y)-\mathcal{E}^{uv}(X,Y)-\int\limits_{{}_\delta \,\zcd}\frac{d \zeta}{2\pi\I} \frac{e^{\frac{\zeta^3}{3} + \zeta^2 u - \zeta X}}{\zeta-\delta} \mathpzc{f}^{-\delta,-v}(Y).
\end{equation}
By moving the integration path to the left of $\delta$ we get
\begin{equation}
-\int\limits_{{}_\delta \,\zcd}\frac{d \zeta}{2\pi\I} \frac{e^{\frac{\zeta^3}{3} + \zeta^2 u - \zeta X}}{\zeta-\delta} = -\mathpzc{g}_2^{\delta,\, u}(X) + e^{\frac{\delta^3}{3} + \delta^2 u - \delta X}.
\end{equation}
Combining this to the definition of $\mathcal{E}^{uv}(X,Y)$ leads to \eqref{eq5.3}.

For the $12$ entry identity we start with
\begin{equation}
\widetilde{\mathcal{A}}^{uv}_{12} (X, Y) = -\int\limits_{{}_{0}{\zcd }} \frac{d \zeta}{2\pi\I} \int\limits_{\wcu\, {}_{\delta,\zeta}} \frac{d \omega}{2\pi\I} \frac{ e^{\frac{\zeta^3}{3} - \zeta^2 u - \zeta X} }{ e^{\frac{\omega^3}{3} - \omega^2 v -\omega Y} } \frac{\zeta-\delta}{\omega-\delta} \frac{\zeta+\omega}{2 \zeta (\zeta-\omega)}.
\end{equation}
Moving the integration contour for $\omega$ to the right of $\delta$, we get the double integral of $\breve{\mathcal{A}}^{uv}_{12}(X,Y)$ and the correction term is $\mathpzc{g}_1^{\delta,\, u}(X) \mathpzc{f}^{-\delta,-v}(Y)$, thus giving \eqref{eq5.4}.

Finally, the identity~\eqref{eq5.5} is an elementary computation. The integral is empty for $X<S_1$ and it is from $[S_1,X]$ otherwise.
\end{proof}

As a consequence of Lemma~\ref{Lemma:5.2}, we have the following identities.

\begin{cor}
Let us define
\begin{equation}
\begin{aligned}
\tilde{\mathpzc{h}}^{\delta,\, u}_1(Y) &= -\int_{S_1}^\infty dV \breve{\mathcal{A}}^{u u_1}_{22}(Y, V) \mathpzc{f}^{-\delta,\, u_1}(V) + \mathpzc{g}_4^{\delta,\, u}(Y),\\
\tilde{\mathpzc{h}}^{\delta,\, u}_2(Y) &= \int_{S_1}^\infty dV \breve{\mathcal{A}}^{u u_1}_{12}(Y,V) \mathpzc{f}^{-\delta,\, u_1}(V) + \mathpzc{g}_3^{\delta,\, u}(Y).
\end{aligned}
\end{equation}
Then
\begin{equation}
\begin{aligned}
\mathpzc{h}^{\delta,\, u}_1(Y) &= \tilde{\mathpzc{h}}^{\delta,\, u}_1(Y) + \mathpzc{g}_2^{\delta,\, u}(Y) \braket{\mathpzc{f}^{-\delta,\, -u_1}}{P_{S_1} \mathpzc{f}^{-\delta,\, u_1}},\\
\mathpzc{h}^{\delta,\, u}_2(Y) &= \tilde{\mathpzc{h}}^{\delta,\, u}_2(Y) + \mathpzc{g}_1^{\delta,\, u}(Y) \braket{\mathpzc{f}^{-\delta,\, -u_1}}{P_{S_1} \mathpzc{f}^{-\delta,\, u_1}}.
\end{aligned}
\end{equation}
\end{cor}

We further remark that the term $\braket{\mathpzc{f}^{-\delta,\, -u_1}}{P_{S_1} \mathpzc{f}^{-\delta,\, u_1}}<\infty$ for any $\delta<0$.

The reason for all these preliminary steps is that taking the limit $\delta\to -\infty$ in $\widetilde{\mathcal{A}}^{uv}_{12}(X,Y)$ and $\widetilde{\mathcal{A}}^{uv}_{22}(X,Y)$ is less straightforward than taking the same limit in $\breve{\mathcal{A}}^{uv}_{12}(X,Y)$ and $\breve{\mathcal{A}}^{uv}_{22}(X,Y)$ because the integration contour for $\omega$ passes to the left of $\delta$ in the former case, while in the latter case it passes to the right.

We now make the above precise and show that replacing the $\mathpzc{h}$'s with the $\tilde{\mathpzc{h}}$'s does not change the result provided that $\delta<0$.

\begin{thm}
For any $\delta<0$, Theorem~\ref{thm:main_asymptotics} is also correct if we replace $\cal Q$, defined in \eqref{eq2.24}, with $\widetilde {\cal Q}$ given by
\begin{equation}
 \widetilde{\cal Q} = (\tilde{\mathpzc{h}}^{\delta,\, u_1}_1, \tilde{\mathpzc{h}}^{\delta,\, u_1}_2, \dots, \tilde{\mathpzc{h}}^{\delta,\, u_m}_1, \tilde{\mathpzc{h}}^{\delta,\, u_m}_2 )^t.
\end{equation}
\end{thm}

\begin{proof}
We need to prove that
\begin{equation}
\bra{-\mathpzc{g}_1^{\delta,\, u_1},\mathpzc{g}_2^{\delta,\, u_1},\ldots,-\mathpzc{g}_1^{\delta,\, u_m},\mathpzc{g}_2^{\delta,\, u_m}}P_S (\Id-J^{-1} \breve{\mathcal{A}}_s)^{-1} P_S \ket{\mathpzc{g}_2^{\delta,\, u_1},\mathpzc{g}_1^{\delta,\, u_1},\ldots,\mathpzc{g}_2^{\delta,\, u_m},\mathpzc{g}_1^{\delta,\, u_m}}=0.
\end{equation}
This expression is the asymptotic limit of \eqref{eq3.84} and the argument is exactly the same as in Lemma~\ref{lemma3.14} since only the (anti-)symmetry properties of the kernels are used there.
\end{proof}

Now we are ready to take the $\delta\to-\infty$ limit of the different terms entering in the statement of Theorem~\ref{thm:main_asymptotics}.

As usual the limit of the functions and kernels is well-defined under appropriate conjugation. Let us set
\begin{equation}\label{eqScalingAirystat}
u_i=-\tau_i-\delta,\quad S_i=s_i+\delta(2u_i+\delta),\quad X=x+\delta(2u+\delta),\quad Y=y+\delta(2v+\delta),
\end{equation}
for $1 \leq i \leq m$, where $u,v$ can be any values in $\{u_1,\ldots,u_m\}$. We denote $u=-\tau-\delta$ and $v=-\sigma-\delta$, thus $\tau,\sigma\in\{\tau_1,\ldots,\tau_m\}$.
\begin{prop}\label{prop:LimitKernels}
With the notations as in \eqref{eqScalingAirystat} we have the following $\delta\to-\infty$ limits:
\begin{equation}
\begin{aligned}
&\frac{e^{\frac23 u^3+u X}}{e^{-\frac23 v^3-v Y}} \breve{\mathcal{A}}^{uv}_{11}(X,Y) \\
&\qquad =- \int\limits_{{}_{\delta+\tau}\zcd } \frac{dz}{2\pi\I}\int\limits_{\wcu\, {}_{-\delta-\sigma,z-2\delta-\tau-\sigma}} \frac{dw}{2\pi\I}\frac{ e^{\frac{z^3}{3} - z(x+\tau^2)}}{ e^{\frac{w^3}{3}- w (y+\sigma^2)}} \frac{(z-2\delta-\tau)(w+2\delta+\sigma)(w+z-\tau-\sigma)}{4(z-\delta-\tau)(w+\delta+\sigma)(z-w-2\delta-\tau-\sigma)}\\
&\qquad\stackrel{\delta\to-\infty}{\longrightarrow} 0,
\end{aligned}
\end{equation}
\begin{equation}
\begin{aligned}
&\frac{e^{\frac23 u^3+u X}}{e^{\frac23 v^3+v Y}} \breve{\mathcal{A}}^{uv}_{12}(X,Y) =-\Id_{[\sigma<\tau]}\Id_{[y-2\delta \sigma\leq x-2\delta\tau]} \frac{e^{-\frac23\tau^3-\tau x}}{e^{-\frac23\sigma^3-\sigma y}} \frac{e^{-(x-y)^2/(4(\tau-\sigma))}}{\sqrt{4\pi(\tau-\sigma)}}\\
&\qquad-\int\limits_{{}_{\delta+\tau}{\zcd}} \frac{dz}{2\pi\I} \int\limits_{{}_{2\delta+\tau}\wcu\, {}_{z-\tau+\sigma}} \frac{dw}{2\pi\I} \frac{ e^{\frac{z^3}{3}-z(x+\tau^2)}}{ e^{\frac{w^3}{3} - w (y+\sigma^2)}} \frac{(z-2\delta-\tau)(w+z-2\delta-\tau-\sigma)}{2(z-\delta-\tau)(w-2\delta-\sigma)(z-w-\tau+\sigma)},\\
&\qquad\stackrel{\delta\to-\infty}{\longrightarrow} -\Id_{[\sigma<\tau]} \frac{e^{-\frac23\tau^3-\tau x}}{e^{-\frac23\sigma^3-\sigma y}} \frac{e^{-(x-y)^2/(4(\tau-\sigma))}}{\sqrt{4\pi(\tau-\sigma)}} -\int\limits_{{\zcd}} \frac{dz}{2\pi\I} \int\limits_{\wcu{}_{z-\tau+\sigma}} \frac{dw}{2\pi\I} \frac{ e^{\frac{z^3}{3}-z(x+\tau^2)}}{ e^{\frac{w^3}{3} - w (y+\sigma^2)}} \frac{1}{(z-w-\tau+\sigma)}
\end{aligned}
\end{equation}
and
\begin{equation}
\begin{aligned}
&\frac{e^{-\frac23 u^3-u X}}{e^{\frac23 v^3+v Y}} \breve{\mathcal{A}}^{uv}_{22}(X,Y) \\
&\qquad -\int\limits_{\I\R+\theta_1} \frac{dz}{2\pi\I} \int\limits_{\I\R+\theta_2} \frac{dw}{2\pi\I} \frac{ e^{\frac{z^3}{3}-z(x+\tau^2)}}{ e^{\frac{w^3}{3} - w (y+\sigma^2)}} \frac{w+z-\sigma+\tau}{(z+2\delta+\tau)(w-2\delta-\sigma)(z-w+2\delta+\tau+\sigma)}\\
&\qquad\stackrel{\delta\to-\infty}{\longrightarrow} 0
\end{aligned}
\end{equation}
with $\theta_1,\theta_2$ satisfying the conditions: $\theta_1>\max\{-\tau,0\}$, $\theta_2<\min\{0,\sigma\}$, and $\theta_1-\theta_2<-2\delta-\tau-\sigma$.
\end{prop}
\begin{proof}
The expression of the conjugated $\breve{\mathcal{A}}^{uv}_{11}(X,Y)$ is obtained by the change of variables $\zeta=z+u$ and $\omega=w-v$.
For $\breve{\mathcal{A}}^{uv}_{12}(X,Y)$, the first term comes from $\mathcal{V}^{uv}(X,Y)$ after the change of variables $\zeta=z-\delta$ and a Gaussian integral, while the second term comes from the double integral after the change of variables $\zeta=z+u$ and $\omega=w+v$.
The result for $\breve{\mathcal{A}}^{uv}_{22}(X,Y)$ is simply obtained starting from the expression of Lemma~\ref{lem:AiryStatA22term} after the change of variables $\zeta=z-u$ and $\omega=w+v$.
\end{proof}

Let us now turn to the limits of the various functions entering the main expression of Theorem~\ref{thm:main_asymptotics}.

\begin{prop}\label{prop:LimitsFunctions}
With notation as in \eqref{eqScalingAirystat} and as $\delta\to-\infty$, we have the following $\delta$-independent functions:
\begin{equation}
\begin{aligned}
\mathpzc{e}^{\delta,\, u_1}(S_1)  & = -e^{-\frac23 \tau_1^3-\tau_1 s_1} \int\limits_{\zcd {}_{-\tau_1}} \frac{dz}{2\pi\I}\frac{e^{\frac{z^3}{3}-z(s_1+\tau_1^2)}}{(z+\tau_1)^2}={\cal R}, \\
e^{\frac23 u_1^3+u_1 Y}\mathpzc{f}^{-\delta,\, u_1}(Y) & = e^{-\frac23 \tau_1^3-\tau_1 y},\\
e^{-\frac23 u^3-uX} \mathpzc{g}_2^{\delta,\, u}(X) & = \int\limits_{\zcd\, {}_{-\tau}} \frac{dz}{2\pi\I} e^{\frac{z^3}{3} - z(x+\tau^2)} \frac{1}{z+\tau},\\
e^{\frac23 u^3+u X}\mathpzc{g}_3^{\delta,\, u}(X) & =\int\limits_{{}_{\tau}\zcd} \frac{dz}{2\pi\I} e^{\frac{z^3}{3}-z(x+\tau^2)} \frac{1}{z-\tau}, \\
e^{\frac23 u_k^3+u_k Y} {\cal U}_{2k}(Y) & = -e^{-\frac23 \tau_k^3-\tau_k y} \int\limits_{\I\R+\eta}\frac{dz}{2\pi\I} \frac{e^{(\tau_k-\tau_1) z^2-z(y-s_1)}}{z},
\end{aligned}
\end{equation}
where $\eta>0$. Moreover, we have the following $\delta$-dependent functions converging to $0$:
\begin{equation}
\begin{aligned}
e^{\frac23 u^3+uX} \mathpzc{g}_1^{\delta,\, u}(X)&= \int\limits_{{}_{\delta+\tau}\zcd} \frac{dz}{2\pi\I} e^{\frac{z^3}{3}- z(x+\tau^2)} \frac{z-\tau}{2(z-\tau-\delta)}\stackrel{\delta\to-\infty}{\longrightarrow} 0,\\
e^{-\frac23 u^3-u X} \mathpzc{g}_4^{\delta,\, u}(X)&=\int\limits_{\zcd\, {}_{-\tau,-2\delta+\tau}} \frac{dz}{2\pi\I} e^{\frac{z^3}{3} -z(x+\tau^2)} \frac{2(z+\delta+\tau)}{(z+\tau)(z+2\delta+\tau)^2}\stackrel{\delta\to-\infty}{\longrightarrow}0.
\end{aligned}
\end{equation}
\end{prop}
\begin{proof}
The result for $\mathpzc{e}^{\delta,\, u_1}(S_1)$ is obtained by the change of variables $\zeta=z-u_1$. The one for $\mathpzc{g}_1^{\delta,\, u}(X)$ by the change of variables $\zeta=z+u$, while for $\mathpzc{g}_2^{\delta,\, u}(X)$ the change of variables is $\zeta=z-u$. The expression for $\mathpzc{f}^{-\delta,\, u_1}(Y)$ is a direct computation. The formula for $\mathpzc{g}_4^{\delta,\, u}(X)$ follows by taking $\zeta=z-u$. Finally, for $\mathpzc{g}_3^{\delta,\, u}(X)$ we change variables as $\zeta=z+u$, and for ${\cal U}_{2k}(Y)$ as $\zeta=z-\delta$.
\end{proof}

\begin{proof}[Proof of Theorem~\ref{thm:LimitAiryStat}]
The expression of the joint distribution in terms of a Fredholm Pfaffian and scalar product is well-defined for any value of $\delta$. This follows by just analyzing the behavior in $x,y$ of the kernels and functions. The only elements which are $\delta$-dependent include contour integrals with terms like $e^{z^3/3-zx}$ or $e^{-w^3/3+w y}$. These give (super-) exponential decay terms in $x,y$. This decay is not affected by the $\delta$-dependent term. Therefore the limiting result as $\delta\to-\infty$ is the one obtained by taking the limit of the different terms inside the integrals by the use of dominated convergence.

Since the diagonal terms of the Pfaffian kernel go to $0$, the Fredholm Pfaffian goes to a Fredholm determinant of a scalar kernel. The term $e^{-\frac23 u^3-uY}\tilde{\mathpzc{h}}^{\delta,\, u}_1(Y)$ as well as $e^{\frac23 u^3+uX}\mathpzc{g}_1^{\delta,\, u}(X)$ go to zero. The only term which has not been yet computed is the limit of $e^{\frac23 u^3+uY}\tilde{\mathpzc{h}}^{\delta,\, u}_2(Y)$, in particular the term $e^{\frac23 u^3+u Y}(\breve{\mathcal{A}}_{12}^{u_k u_1}P_{S_1} \mathpzc{f}^{-\delta,\, u_1})(Y)$. This elementary computation gives the last term in $\Phi^k(y)$ of \eqref{eqDefinFctAiryStat}.

We obtain the claimed result by putting all of the above together.
\end{proof}

Finally we want to write the limiting expressions in terms of Airy functions and exponentials, to compare with the result derived in~\cite{BFP09}.

\begin{lem}\label{lem:KernelAiryrepresentation}
With notation as in \eqref{eqScalingAirystat} we have
\begin{equation}
\begin{aligned}
\lim_{\delta\to-\infty}\frac{e^{\frac23 u^3+u X}}{e^{\frac23 v^3+v Y}} \breve{\mathcal{A}}^{uv}_{12}(X,Y)
=&-\Id_{[\sigma<\tau]} \frac{e^{-\frac23\tau^3-\tau x}}{e^{-\frac23\sigma^3-\sigma y}} \frac{e^{-(x-y)^2/(4(\tau-\sigma))}}{\sqrt{4\pi(\tau-\sigma)}}\\
&+\int_0^\infty d\lambda \Ai(x+\lambda+\tau^2)\Ai(y+\lambda+\sigma^2)e^{\lambda(\tau-\sigma)},
\end{aligned}
\end{equation}
which is the kernel $\widehat K_{\rm Ai}$ of~\cite[Thm.~1.2]{BFP09} after setting our $(\sigma,\tau)$ to their $(\tau_i,\tau_j)$.
\end{lem}

\begin{proof}
We replace
\begin{equation}
\frac{1}{z-w-\tau+\sigma}=\int_0^\infty d\lambda e^{-\lambda(z-w-\tau+\sigma)}, \quad \Re(z-w-\tau+\sigma)>0
\end{equation}
into the limiting expression of Proposition~\ref{prop:LimitKernels} and the use of the two Airy identities (see e.g.~\cite[Equation~(A.1)]{BFP09})
\begin{equation}\label{eqAiryIdentities}
-\int\limits_{\zcd {}}\frac{dz}{2\pi\I} e^{\frac{z^3}{3}-zx} = \Ai(x),\qquad \int\limits_{\wcu {}}\frac{dw}{2\pi\I} e^{-\frac{w^3}{3}+wx} = \Ai(x)
\end{equation}
to get the result.
\end{proof}

\begin{lem}\label{lem:FunctionR}
With notation as in~\eqref{eqScalingAirystat} we have
\begin{equation}
\mathpzc{e}^{\delta,\, u_1}(S_1)= s_1 + e^{-\frac23 \tau_1^3}\int_{s_1}^\infty d\lambda \int_0^\infty d\mu e^{-\tau_1(\lambda+\mu)}\Ai(\lambda+\mu+\tau_1^2),
\end{equation}
which is the function $\cal R$ in~\cite[Definition~1.1]{BFP09}.
\end{lem}

\begin{proof}
One starts with the integral expression in Proposition~\ref{prop:LimitsFunctions}. Moving the integration contour to the right of $-\tau_1$ leads to a residue contribution equal to $s_1$. Then we use
\begin{equation}
\frac{1}{(z+\tau_1)^2}=\int_0^\infty d\lambda \int_\lambda^\infty d\mu e^{-\mu(z+\tau_1)}, \quad \Re(z)>-\tau_1
\end{equation}
together with an Airy identity from~\eqref{eqAiryIdentities} and finally we shift the integration variables $\lambda$ and $\mu$ as in the final expression.
\end{proof}

\begin{lem}\label{lem:FunctionPsi}
With notation as in \eqref{eqScalingAirystat} we have
\begin{equation}
e^{-\frac23 u^3-uX} \mathpzc{g}_2^{\delta,\, u}(X) = e^{\frac23\tau^3+\tau x}-\int_0^\infty d\lambda e^{-\lambda \tau} \Ai(x+\lambda+\tau^2),
\end{equation}
which is equal to the function $\Psi_j$ of~\cite[Definition~1.1]{BFP09} after setting our $\tau$ to their $\tau_j$.
\end{lem}

\begin{proof}
It follows directly by first taking the integration path to the right of $-\tau$ and then using $(z+\tau)^{-1}=\int_0^\infty d\lambda e^{-\lambda(z+\tau)}$ together with one Airy identity from~\eqref{eqAiryIdentities}.
\end{proof}

\begin{lem}\label{lem:FunctionPhi}
With notation as in \eqref{eqScalingAirystat} we have
\begin{equation}
\begin{aligned}
&e^{\frac23 u^3+u X}\left[\mathpzc{g}_3^{\delta,\, u_k}(X)-{\cal U}_{2k}(X)+ (\breve{\mathcal{A}}_{12}^{u_k u_1}P_{S_1} \mathpzc{f}^{-\delta,\, u_1})(X)\right] \\
& \qquad = -\int_0^\infty d\lambda e^{\lambda \tau_k} \Ai(x+\lambda+\tau_k^2)+\Id_{[\tau_1<\tau_k]}e^{-\frac23\tau_k^3-\tau_k x} \int_{-\infty}^{s_1-x} d\lambda \frac{e^{-\lambda^2/(4(\tau_k-\tau_1))}}{\sqrt{4\pi (\tau_k-\tau_1)}}\\
&\qquad \ +e^{-\frac23\tau_1^3} \int_0^\infty d\lambda \int_{s_1}^\infty d\mu e^{-\lambda(\tau_1-\tau_k)} e^{-\tau_1 \mu} \Ai(\mu+\lambda+\tau_1^2) \Ai(x+\lambda+\tau_k^2),
\end{aligned}
\end{equation}
which is equal to the function $\Phi_k$ of~\cite[Definition~1.1]{BFP09}.
\end{lem}

\begin{proof}
The computation for $\mathpzc{g}_3^{\delta,\, u_k}(X)$ is as in the previous lemmas and leads to the first term on the right-hand side. For $-{\cal U}_{2k}(X)$ we replace $z^{-1}=\int_0^\infty d\lambda e^{-\lambda z}$ since $\Re(z)>0$, and then perform a Gaussian integral. This leads to the second term. Finally, for the computation of $(\breve{\mathcal{A}}_{12}^{u_k u_1}P_{S_1} \mathpzc{f}^{-\delta,\, u_1})(X)$ we just use the representation in Lemma~\ref{lem:KernelAiryrepresentation} and the explicit formula for $\mathpzc{f}^{-\delta,\, u_1}$.
\end{proof}

\appendix

\section{On Pfaffians and point processes} \label{sec:pfaff}

In this section we recall some basics of Pfaffians and Fredholm Pfaffians. For more on the latter see~\cite[Appendix]{OQR16}; for the former see~\cite{Ste90}.

\paragraph{Pfaffians.} The Pfaffian of an anti-symmetric $2n \times 2n$ matrix $(a_{i j})$ is defined as:
\begin{equation}
 \pf [a_{i j}]_{1\leq i < j\leq 2n} = \frac{1}{2^{n} n!}\sum_{\sigma \in S_{2n}}\sgn(\sigma)a_{\sigma(1),\sigma(2)} a_{\sigma(3),\sigma(4)} \cdots a_{\sigma(2n-1),\sigma(2n)},
\end{equation}
where $S_{2n}$ is the permutation group of $\{1, \dots, 2n\}$. Observe that the Pfaffian is determined entirely by the upper triangular part of the matrix. Furthermore one has the following relation
\begin{equation}\label{eq:pf_det}
\left(\pf[a_{i j}]_{1\leq i < j\leq 2n}\right)^2=\det [a_{i j}]_{1\leq i,j\leq 2n}.
\end{equation}

Suppose we start with a $2 \times 2$ anti-symmetric matrix kernel $K(x, y)$, i.e.\,$K$ is a $2 \times 2$ matrix function of $(x, y)$ which satisfies $K(x, y) = -K^t(y, x)$ ($t$ is the transposition). Given such a kernel and points $x_1, \dots, x_n$, we can define a $2n \times 2n$ anti-symmetric matrix $K^{(n)}$ block-wise as follows: the $2 \times 2$ block at position $(i,j)$ ($1 \leq i, j \leq n$) is the matrix $K(x_i, x_j)$. $K^{(n)}$ thus defined is even-dimensional and anti-symmetric because $K(x, y) = -K^t(y, x)$ so its Pfaffian is well-defined.

\paragraph{Pfaffian processes.} A point process\footnote{See e.g.\,\cite{Jo05, BOO00, Bor10} for more on point processes.} on a configuration space $\Lambda$ is called \emph{Pfaffian with $2 \times 2$ matrix correlation kernel $K$} if there exists a $2 \times 2$ matrix $K$ satisfying $K(x, y) = -K^t(y, x)$ such that
the $n$-point correlation functions $\rho_n(x_1, x_2, \dots, x_n) = \Pb (S : x_1 \in S, \dots, x_n \in S)$ of the process, for all $n \geq 1$, are Pfaffians of the associated $2n \times 2n$ matrix $K^{(n)}$:
\begin{equation}
 \rho_n(x_1, x_2, \dots, x_n) = \pf[K^{(n)} (x_i, x_j)]_{1 \leq i, j\leq n}.
\end{equation}
For instance, one has $\rho_1(x) = K_{12}(x, x)$.

\paragraph{Fredholm Pfaffians.} Given a $2 \times 2$ anti-symmetric matrix kernel $K$ defined on a configuration space $\Lambda$ equipped with a measure $dx$, the \emph{Fredholm Pfaffian} of $K$ restricted to the subspace $U \subset \Lambda$ is defined as
 \begin{equation}\label{eq:fredholm_pf_def}
\pf(J + \lambda K)_{L^2(U)} = \sum_{n=0}^{\infty}\frac{\lambda^n}{n!} \int_{U^n} \pf[ K^{(n)} (x_i, x_j)]_{1\leq i, j \leq n} \prod_{i=1}^n d x_i.
\end{equation}
Here $J$ is the anti-symmetric matrix kernel $J(x,y)=\delta_{x,y} \left( \begin{smallmatrix} 0 & 1 \\ -1 & 0 \end{smallmatrix} \right)$.

Fredholm Pfaffians are defined up to conjugation, in the following sense. Suppose $\widetilde{K}$ is the anti-symmetric matrix kernel
\begin{equation}
\widetilde{K}(x,y)=
\left(\begin{smallmatrix} e^{f(x)} & 0 \\ 0 & e^{-f(x)} \end{smallmatrix}\right)
K(x,y)
\left(\begin{smallmatrix} e^{f(y)} & 0 \\ 0 & e^{-f(y)} \end{smallmatrix}\right)
\end{equation}
for a $dx$-measurable function $f$. Then $\pf [K^{(n)} (x_i, x_j)]_{1\leq i<j\leq 2n} = \pf [\widetilde{K}^{(n)} (x_i, x_j)]_{1\leq i<j\leq 2n}$ and so $\pf(J+\lambda K)_{L^2(U)} = \pf(J+\lambda \widetilde{K})_{L^2(U)}$. Importantly, we can use this to define $\pf(J+\lambda K)_{L^2(Y)}$ even if $K$ is not trace-class provided we find an appropriate $f$ which makes $\widetilde{K}$ trace-class.

We have the following relation between Fredholm Pfaffians with $2 \times 2$ matrix kernels $K$ and block Fredholm determinants with related kernel $J^{-1}K$:
\begin{equation}\label{eq:fred_pf_det}
 \pf(J + \lambda K)^2_{L^2(U)} = \det(\Id + \lambda J^{-1} K)_{L^2(U)},
\end{equation}
where we remark the Fredholm determinant on the right hand side is defined as in~\eqref{eq:fredholm_pf_def} with pf replaced by det and $K^{(n)}$ by $(J^{-1}K)^{(n)}$.

\paragraph{Extended kernels and Pfaffians.} In this note we are interested in \emph{(time-) extended Pfaffian point processes} as they provide the starting formulae for our work. We fix some integer $m \geq 1$ and look at the process at $m$ different time-space positions. Such Pfaffian processes can be viewed two (equivalent) ways: as processes with \emph{extended} $2 \times 2$ matrix kernels or as point processes with a $2m \times 2m$ matrix kernel. Rather than giving the definition here, we exemplify what this means in the next section in Remark~\ref{rem:ext_pfaffian}.

\section{Correlations for geometric weights} \label{sec:geom_wts}
We now state the main result on multi-point correlations for last passage percolation with independent geometric random variables. Specializing appropriately and taking the resulting parameters to $1$ will recover the exponential weights studied in this paper, and notably we will have proven Theorem~\ref{thm:exp_corr} this way.

\paragraph{Generic geometric weights.} Let $a, x_1, \dots, x_N$ be real numbers satisfying
\begin{equation}
0 \leq a < \min_i \tfrac{1}{x_i}, \quad 0 < x_1, \dots, x_N < 1
\end{equation}
and consider the following independent geometric weights $(W_{i,j})_{1 \leq j \leq i \leq N}$ on the corresponding lattice sites forming the half-space:
\begin{equation}
 W_{i,j} = \begin{cases}
 \mathrm{Geom}(a x_i), & i = j, \\
 \mathrm{Geom}(x_i x_j), & i > j.
 \end{cases}
\end{equation}
Here a random variable $X$ is said \textit{geometric} $\mathrm{Geom}(q)$ if $\Pb(X = k) = (1-q)q^k, \forall k\in\{0,1,2,\ldots\}$. See Figure~\ref{fig:geom_lpp} for an example.

\begin{figure}[t!]
 \centering
 \includegraphics[height=5.5cm]{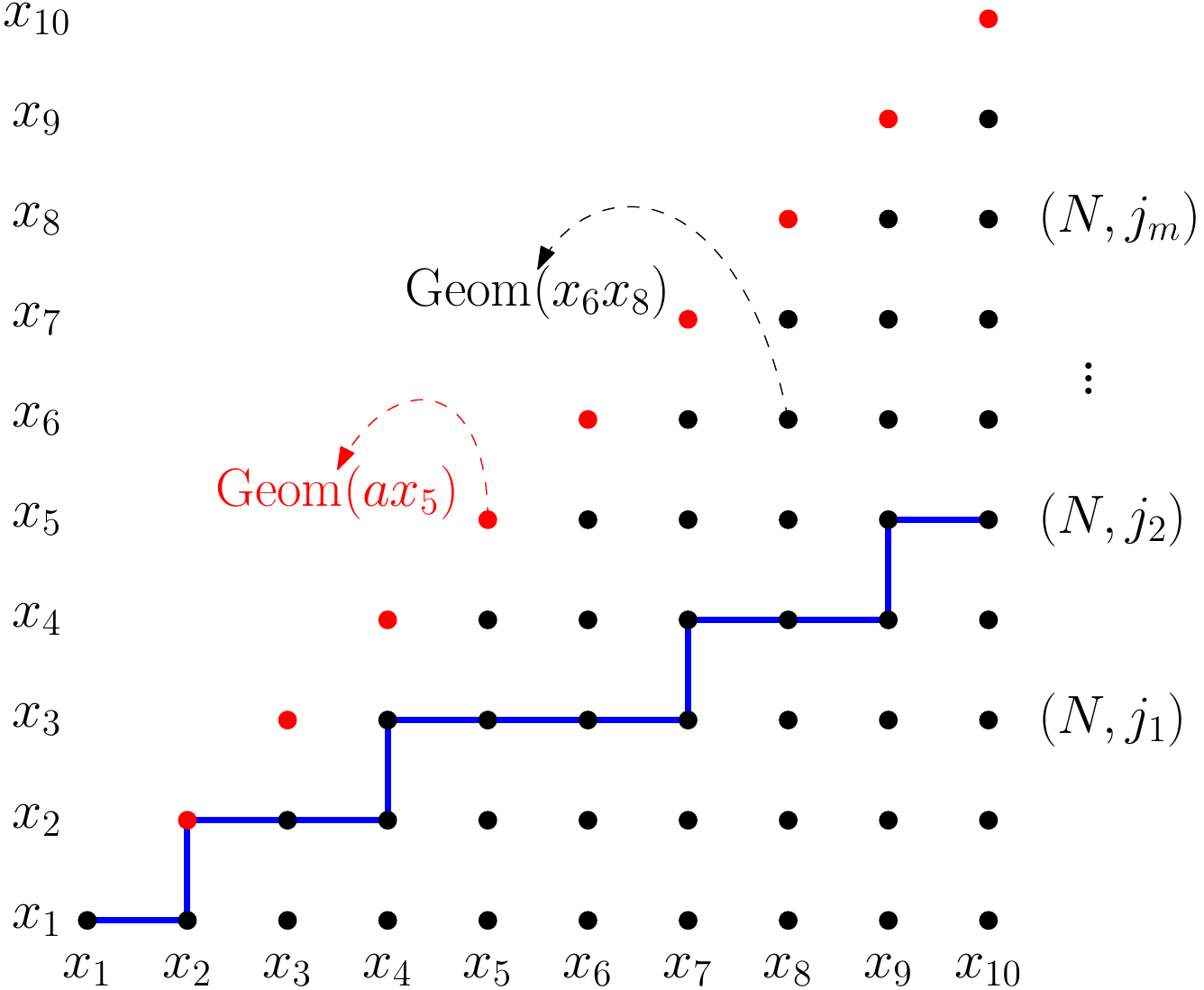}
 \caption{A possible LPP path (polymer) from $(1, 1)$ to $(N, j_2)=(10, 5)$. The non-diagonal dots are independent geometric random variables with parameters assigned by multiplying the row and column $x$ parameters; the diagonal has an extra parameter $a$.}
 \label{fig:geom_lpp}
\end{figure}

Let $L_{N, j_\ell}$ be the LPP from $(1,1)$ to $(N,j_\ell)$. The joint distribution function $\Pb(\bigcap_{\ell=1}^{m} \{ L_{N, j_\ell} \leq s_\ell \})$ is a Fredholm Pfaffian, a result we state next. We follow the exposition of Betea--Bouttier--Nejjar--Vuleti\'c~\cite{BBNV18}. The precise statement as stated below was previously proven by Baik--Barraquand--Corwin--Suidan~\cite{BBCS17}. The non-extended kernel first appeared, perhaps not completely rigourously, in~\cite{SI03} (the case $(m=1, j_1 = N)$). The Pfaffian structure for the case $(m=1, j_1=N)$ was first derived by Rains~\cite{Ra00}, and subsequently extended in~\cite{SI03, RB04, BBCS17, Gho17, BBNV18}. See also the related algebraic work of Baik--Rains~\cite{BR99, BR01b} for an alternative but equivalent approach via Toeplitz+Hankel determinants and matrix integrals, as well as the more combinatorial approach of Forrester--Rains~\cite{FR07}.

\begin{thm} \label{thm:geom_corr}
 The joint distribution function of the $L_{N, j_{\ell}}$'s, $1 \leq \ell \leq m$, is a Fredholm Pfaffian
 \begin{equation}
 \Pb \left( \bigcap_{\ell=1}^{m} \{ L_{N, j_\ell} \leq s_\ell \} \right) = \pf (J-K)_{\ell^2(X)},
 \end{equation}
 where
 \begin{equation}
 X = \bigcup_{\ell = 1}^m \{ \ell \} \times I_\ell, \qquad I_\ell = \{s_\ell + 1, s_\ell + 2, s_\ell + 3, \dots \}
 \end{equation}
 (the union is disjoint) and with $2 \times 2$ matrix correlation kernel $K$ given by:
 \begin{equation}
 \begin{split}
 K_{11}(j, k; j',k') &= \ics \oint \frac{dz}{z^{k}} \oint \frac{dw}{w^{k'}} F(j, z) F(j', w) \frac{(z-w) (z-a) (w-a)} {(z^2-1) (w^2-1) (zw-1)}, \\
 K_{12}(j, k; j', k') &= \ics \oint \frac{dz}{z^{k}} \oint \frac{dw}{w^{-k'+1}} \frac{F(j, z)} {F(j', w)} \frac{(zw-1) (z-a)}{(z-w) (z^2-1) (w-a)} \\
 &= -K_{21}(j', k'; j, k), \\
 K_{22}(j, k; j', k') &= \ics \oint \frac{dz}{z^{-k+1}} \oint \frac{dw}{w^{-k'+1}} \frac{1}{F(j, z) F(j', w)} \frac{ (z-w)}{ (zw-1) (z-a) (w-a)},
 \end{split}
 \end{equation}
 where
 \begin{equation}
 F(j, z)= \frac{\prod_{\ell=1}^N (1-x_\ell/z)}{\prod_{\ell=1}^{j} (1-x_\ell z)}
 \end{equation}
and where the contours are positively oriented circles centered around the origin satisfying the following conditions:
 \begin{itemize}
 \item for $K_{11}$, $1 < |z|, |w| < \min_i \frac{1}{x_i}$;
 \item for $K_{12}$, $\max \{ \max_i x_i, a\} < |w|$, $1 < |z| < \min_i \frac{1}{x_i}$, as well as $|w| < |z|$ if $j \leq j'$ and $|z| < |w|$ otherwise;
 \item for $K_{22}$, $\max \{ \max_i x_i, a\} < |w|, |z|$ and $1 < |zw|$.
 \end{itemize}
\end{thm}

\paragraph{Remarks on extended Pfaffians.} We make a few remarks regarding the result we just stated, most notably to introduce extended kernels and to give the expansion of their Fredholm Pfaffians. This discussion is often skipped in the literature; see~\cite{Ra00} and~\cite{Bor10} for more on pfaffian and determinantal extended point processes.

\begin{rem} \label{rem:ext_pfaffian}
 Observe the following:
 \begin{itemize}
 \item with $\chi_k$ the characteristic function of $\{k\} \times I_k$ (note the latter are disjoint) we have
 \begin{equation} \label{eq:iso_geom}
 \begin{split}
  \ell^2(X) = \ell^2 \left(\bigcup_{k=1}^m \{k\} \times I_k \right) \cong & \bigoplus_{k=1}^m \ell^2(I_k), \quad \chi_1 f_1 + \cdots + \chi_m f_m \mapsfrom (f_1, \dots, f_m);
 \end{split}
 \end{equation}
 \item we can alternatively write the probability of interest as
 \begin{equation} \label{eq:geom_2_ker}
 \Pb \left( \bigcap_{\ell=1}^{m} \{ L_{N, j_\ell} \leq s_\ell \} \right) = \pf (J-P_s K P_s)_{\ell^2(\{1, 2, \dots, m \} \times \Z)},
 \end{equation}
 where $P_s(\ell, k) = \Id_{[k \in I_\ell]}$. Introducing the parameter $\lambda$ ($\lambda=-1$ in the case of interest), the expansion of the Fredholm Pfaffian $\pf (J + \lambda P_s K P_s)_{\ell^2(\{1, 2, \dots, m \} \times \Z)}$ becomes:
 \begin{equation} \label{eq:multi_expansion}
 \pf (J + \lambda P_s K P_s)_{\ell^2(\{1, 2, \dots, m \} \times \Z)} = \sum_{n = 0}^{\infty} \frac{\lambda^n}{n!} \sum_{i_1, \dots, i_n = 1}^{m} \sum_{k_1 \in I_{i_1}} \cdots \sum_{k_n \in I_{i_n}} \pf [K^{(n)} (j_{i_c}, k_c; j_{i_d}, k_d)]_{1 \leq c, d \leq n},
 \end{equation}
 where $[K^{(n)} (j_{i_c}, k_c; j_{i_d}, k_d)]_{1 \leq c, d \leq n}$ is the skew-symmetric $2n \times 2n$ matrix with $2 \times 2$ block at $(c, d)$ ($1 \leq c, d \leq n$) given by the matrix kernel $K (j_{i_c}, k_c; j_{i_d}, k_d)$.
 \item under the isomorphism from~\eqref{eq:iso_geom}, we can also view the multi-point probability of Theorem~\ref{thm:geom_corr} as the following Fredholm Pfaffian
 \begin{equation} \label{eq:geom_2m_ker}
 \begin{split}
 \Pb \left( \bigcap_{\ell=1}^{m} \{ L_{N, j_\ell} \leq s_\ell \} \right) &= \pf (J^{(m)}-K^{(m)})_{\ell^2(I_1) \oplus \cdots \oplus \ell^2 (I_m)} \\
 &= \pf (J^{(m)}-P^{(m)}_s K^{(m)} P^{(m)}_s)_{\ell^2(\Z) \oplus \cdots \oplus \ell^2 (\Z)},
 \end{split}
 \end{equation}
 where $J^{(m)}$ is the $2m \times 2m$ anti-symmetric matrix having just $J = \left( \begin{smallmatrix} 0 & 1 \\ -1 & 0 \end{smallmatrix} \right)$ on the diagonal; $K^{(m)}(k, k')$ is the $2m \times 2m$ matrix kernel whose $(c, d)$ $2 \times 2$ block/component ($1 \leq c, d \leq m$) is the $2 \times 2$ matrix kernel $K(j_c, k; j_d, k')$ from Theorem~\ref{thm:geom_corr}; and $P^{(m)}_s$ is the $2m \times 2m$ diagonal matrix $\mathrm{diag} (\chi_1, \chi_1, \dots, \chi_m, \chi_m)$ where $\chi_k$ is the characteristic function of $I_k$. The expansion remains that of equation~\eqref{eq:multi_expansion}; doing so however enables us to do useful computations with the matrix kernel $K^{(m)}$. To go between the two equalities in~\eqref{eq:geom_2m_ker} one uses
 \begin{equation}
 (P_s K P_s) (k, k') = \sum_{i_1, i_2 = 1}^m \begin{pmatrix} \chi_{i_1} (k) & 0 \\ 0 & \chi_{i_1} (k) \end{pmatrix} (K^{(m)})^{i_1 i_2} (k, k') \begin{pmatrix} \chi_{i_2} (k') & 0 \\ 0 & \chi_{i_2} (k') \end{pmatrix},
 \end{equation}
 where $(K^{(m)})^{i_1 i_2} (k, k')$ is just the $2 \times 2$ matrix kernel (block) $K (j_{i_1}, k; j_{i_2}, k')$.
 \end{itemize}
\end{rem}

\paragraph{Geometric weights of interest.} Let us now process the kernel of Theorem~\ref{thm:geom_corr} into a form suitable for our needs. We start with parameters $a, b$ with $0 \leq q, b < 1$ and $a\geq 0$ so that $ab<1$ and $a \sqrt{q} < 1$. We are interested in the following choice of $x$ parameters:
\begin{equation}
 x_1 = b, \quad x_i = \sqrt{q}, \quad i \geq 2.
\end{equation}
Let us denote
\begin{equation}
 \phi^{\rm geo}(z) = \left[\frac{1-\sqrt{q}/z} {1-\sqrt{q} z} \right]^{N-1} \quad \textrm{and}\quad B^{\rm geo}(z) = \frac{z-b}{1-bz}.
\end{equation}
With this in mind we have the following result.

\begin{thm} \label{thm:geom_corr_2}
Consider integers $k, k' \geq 0$ and parameters $a, b, \sqrt{q}$ all different. The kernel of Theorem~\ref{thm:geom_corr}, which we now label $K^{\rm geo}$, becomes (with the above choice of parameters)
\begin{align}
 & K^{\rm geo}_{11}(j, k; j',k') = \ics \Bigg[ \oint\limits_{\Gamma_{b, \sqrt{q}}} dw \oint\limits_{\Gamma_{\frac{1}{\sqrt{q}}}} dz + \oint\limits_{\Gamma_{\sqrt{q}}} dw \oint\limits_{\Gamma_{\frac{1}{b}}} dz \Bigg] \nonumber \\
 & \quad \times \frac{w^{k'}}{z^{k+1}} \frac{\phi^{\rm geo}(z) B^{\rm geo} (z)} {\phi^{\rm geo}(w) B^{\rm geo} (w)} (1-\sqrt{q} z)^{N-j} (1-\sqrt{q}/w)^{N-j'} \frac{(zw-1) (z-a) (1-aw)} {(z-w) (z^2-1) (w^2-1)}, \nonumber \\
 & K^{\rm geo}_{12}(j, k; j', k') = V^{\rm geo}(j, k; j', k') - \ics \Bigg[ \oint\limits_{\Gamma_{a, b, \sqrt{q}}} dw \oint\limits_{\Gamma_{\frac{1}{\sqrt{q}}}} dz + \oint\limits_{\Gamma_{a, \sqrt{q}}} dw \oint\limits_{\Gamma_{\frac{1}{b}}} dz \Bigg] \\
 & \quad \times \frac{w^{k'}}{z^{k+1}} \frac{\phi^{\rm geo}(z) B^{\rm geo} (z)} {\phi^{\rm geo}(w) B^{\rm geo} (w)} \frac{(1-\sqrt{q} z)^{N-j}} {(1-\sqrt{q} w)^{N-j'} } \frac{(zw-1) (z-a)}{(z-w) (z^2-1) (w-a)}, \nonumber \\
 & K^{\rm geo}_{22}(j, k; j', k') = E^{\rm geo} (j, k; j', k') - \ics \Bigg[ \oint\limits_{\Gamma_{\sqrt{q}}} dw \!\!\!\!\!\! \oint\limits_{\Gamma_{\frac{1}{a}, \frac{1}{b}, \frac{1}{\sqrt{q} } } } \!\!\!\!\! dz + \oint\limits_{\Gamma_{a}} dw \!\!\!\! \oint\limits_{\Gamma_{\frac{1}{b}, \frac{1}{\sqrt{q}}}} \!\!\!\! dz + \oint\limits_{\Gamma_{b}} dw \!\!\!\!
 \oint\limits_{\Gamma_{\frac{1}{a}, \frac{1}{\sqrt{q}}}} \!\!\!\! dz \Bigg] \nonumber \\
 & \quad \times \frac{w^{k'}}{z^{k+1}} \frac{\phi^{\rm geo}(z) B^{\rm geo} (z)} {\phi^{\rm geo}(w) B^{\rm geo} (w)} \frac{1}{(1-\sqrt{q}/z)^{N-j} (1-\sqrt{q} w)^{N-j'} } \frac{(z w-1)} { (z-w) (1-za) (w-a)}, \nonumber
\end{align}
where as always $K^{\rm geo}_{21}(j', k'; j, k) = - K^{\rm geo}_{12}(j, k; j', k')$. $E^{\rm geo}$ and $V^{\rm geo}$ are given by
\begin{equation} \label{eq:V_E_geo}
 \begin{split}
 E^{\rm geo} (j, k; j', k') &= \oint\limits_{\Gamma_{0, a, \sqrt{q}} } \frac{dz}{2 \pi \I} \frac{z^{k'-k-1}} {(1-\sqrt{q}/z)^{N-j} (1-\sqrt{q} z)^{N-j'}} \frac{(1-z^2)}{(1-az)(z-a)}, \\
 V^{\rm geo} (j, k; j', k') &= -\Id_{[j > j']} \oint\limits_{\Gamma_{0}} \frac{dz}{2 \pi \I} \frac{z^{k'-k-1}} {(1-\sqrt{q} z)^{j-j'}} = -\Id_{[j > j']} \Id_{[k \geq k']} \frac{q^{\frac{k-k'}{2}} (j-j')_{k-k'} } {(k-k')!}
 \end{split}
\end{equation}
with $(x)_n = \prod_{0 \leq \ell < n} (x+\ell)$ being the Pochhammer symbol.
\end{thm}

\begin{proof}
 The proof is identical to that of~\cite[Lemmas C.3, C.4, C.5]{BFO20} with one exception. For the $12$ entry there is the extra $V^{\rm geo}$ kernel. It appears for $j > j'$ since in that case only, as can be seen in Theorem~\ref{thm:geom_corr}, the $w$ contour is on the outside of the $z$ contour. Exchanging the two, we pick up a residue at $w=z$, and this is exactly $V^{\rm geo}$ in its integral form. Obtaining the second form of $V^{\rm geo}$ is just a residue computation. The proof for the rest of the $12$ entry then proceeds similarly to~\cite[Lemma C.4]{BFO20}.
\end{proof}

We can rewrite $E^{\rm geo}$ in a form more suitable for asymptotic analysis. We record the result below.

\begin{lem}
 We have
 \begin{equation}
 E^{\rm geo} (j, k; j', k') =
 \begin{dcases}
  -\oint\limits_{\Gamma_{1/\sqrt{q}, 1/a}} \frac{dz}{2 \pi \I} \frac{z^{k'-k-1}} {(1-\sqrt{q}/z)^{N-j} (1-\sqrt{q} z)^{N-j'}} \frac{(1-z^2)}{(1-az)(z-a)}, & \text{if } k \geq k',\\
 \oint\limits_{\Gamma_{\sqrt{q}, a}} \frac{dz}{2 \pi \I} \frac{z^{k'-k-1}} {(1-\sqrt{q}/z)^{N-j} (1-\sqrt{q} z)^{N-j'}} \frac{(1-z^2)}{(1-az)(z-a)}, & \text{if } k < k'
 \end{dcases}
 \end{equation}
and this explicitly shows $E^{\rm geo}$ is anti-symmetric:
\begin{equation}
 E^{\rm geo} (j, k; j', k') = -E^{\rm geo} (j', k'; j, k).
\end{equation}
\end{lem}

\begin{proof}
 If $k < k'$ we see $0$ is not a pole of the integrand so we can exclude it from the contour of Theorem~\ref{thm:geom_corr_2}. If $k \geq k'$, we see that $\infty$ is not a pole of the integrand. We can then deform the contours via infinity to enclose $1/a$ and $1/\sqrt{q}$, picking up an overall minus sign in the process. Let us remark this argument works regardless of whether $a > 1$ or $a \leq 1$ (as long as $ab, a \sqrt{q} < 1$) at the cost of possibly using disjoint contours around $a^{\pm 1}$ and $\sqrt{q}^{\pm 1}$. Finally, we observe that
 \begin{equation}
 E^{\rm geo}(j, k; j, k) = 0
 \end{equation}
 in two steps using the formula from Theorem~\ref{thm:geom_corr_2}. If $j=N$, the residue contributions from $0$ and $a$ cancel. If $j < N$, $0$ is not a pole anymore but then the residue contributions from $a$ and $\sqrt{q}$ cancel as well. This proves the result.
\end{proof}

\section{From geometric to exponential weights: proof of Theorem~\ref{thm:exp_corr}} \label{sec:geom_exp_limit}

The integrable LPP model with exponential weights of Section~\ref{sec:int} and its correlation kernel is a limit $q = e^{-\epsilon} \to 1-$ as $\epsilon \to 0+$ of the geometric model described above. In this section we make this limit explicit. The proofs of~\cite[Appendix C.3]{BFO20} apply mutatis mutandis modulo the change in some conjugation factors. Thus we will only state the statements and explain the differences from~\cite{BFO20} without repeating the details.

Throughout this section we fix $m$ a positive integer, two real numbers $\alpha \in (-1/2, 1/2), \beta \in (0, 1/2)$ and $m$ ordered integers $1 \leq j_1 < \dots < j_m \leq N$.

We are looking at the limit:
\begin{equation}
(a,b)=(1-\epsilon \alpha,1-\epsilon \beta),\quad q=1-\epsilon ,\quad (k, k')=\epsilon^{-1}(x, x').
\end{equation}
We wish to show that the $2m \times 2m$ matrix kernel $K^{\rm geo}$
converges to the kernel $K$ of Theorem~\ref{thm:exp_corr}
Further we will show the corresponding convergence of Fredholm Pfaffians:
\begin{equation}
 \pf (J-P_k K^{\rm geo} P_k) \to \pf (J-P_s K P_s),
\end{equation}
where $P_k$ and $P_s$ are the $2m \times 2m$ projectors
\begin{equation}
 \begin{split}
 P_k &= {\rm diag} (\Id_{[\ell > k_1]}, \Id_{[\ell > k_1]}, \dots, \Id_{[\ell > k_m]}, \Id_{[\ell > k_m]}), \\
 P_s &= {\rm diag} (\Id_{[x > s_1]}, \Id_{[x > s_1]}, \dots, \Id_{[x > s_m]}, \Id_{[x > s_m]})
 \end{split}
\end{equation}
and $k_c = s_c / \epsilon$ as $\epsilon \to 0+$. This will then finish the proof.

Let us consider the accordingly rescaled and conjugated kernel
\begin{alignat}{2}
K^{\rm geo,\, \epsilon}_{11}(j, x; j', x')& =\epsilon ^{-2-2N + (j+j')}K^{\rm geo}_{11}(j, k; j', k'), &\quad K^{\rm geo,\, \epsilon}_{12}(j, x; j', x') &= \epsilon^{j-j'-1}K^{\rm geo}_{12}(j, k; j', k'), \nonumber \\
K^{\rm geo,\, \epsilon}_{21}(j, x; j', x')& =\epsilon^{j'-j-1}K^{\rm geo}_{21}(j, k; j', k'), &\quad K^{\rm geo,\, \epsilon}_{22}(j, x; j', x') &= \epsilon^{2N - (j+j')}K^{\rm geo}_{22}(j, k; j', k'),
\end{alignat}
where $j, j' \in \{ j_1, \dots, j_m \}$ and the kernels on the left are still the discrete ones but now conjugated and with limiting parameters now depending on $\epsilon$.

The following result is straightforward.

\begin{prop}\label{prop:ConvGeom}
Uniformly for $x, x'$ in a compact subset of $\R_+$ and for any $j, j' \in \{j_1, \dots, j_m\}$, we have that:
\begin{equation}
\lim_{\epsilon \to 0} K^{\rm geo,\, \epsilon}_{cd}(j, x; j', x')= K_{cd}(j, x; j', x'), \quad 1 \leq c,d \leq 2
\end{equation}
with the extended kernel $K$ being the one of the exponential model from Theorem~\ref{thm:exp_corr}.
\end{prop}

\begin{proof}

 The proof is the same as that of Proposition~45 (Appendix C.3) from~\cite{BFO20}, modulo different notation and the fact we have an extended kernel (which does not change any of the limiting arguments). We make the change of integration variables as $z=1+Z \epsilon $, $w=1+W \epsilon $. We further observe that all variables are in an $\epsilon $-neighborhood of $1$ and hence $Z$- and $W$-contours can be fixed independently of $\epsilon$ as long as they are in the correct position with respect to the poles.

 The following pointwise limits are then immediate:
\begin{equation}
\begin{aligned}
\lim_{\epsilon \to 0} \frac{(1-\sqrt{q}/z)^{N-1}}{(1-\sqrt{q}z)^{N-1}} z^{-k}& =\left[ \frac{\tfrac12+z}{\tfrac12-z} \right]^{N-1} e^{-x Z}=\Phi(x,Z),\\
\lim_{\epsilon \to 0} (1-\sqrt{q} z)^{N-j} \epsilon^{j-N} &= (\tfrac12-Z)^{j-N},\\
\lim_{\epsilon \to 0} (1-\sqrt{q} /w)^{N-j} \epsilon^{j-N} &= (\tfrac12+W)^{j-N}
\end{aligned}
\end{equation}
and
\begin{equation}
\begin{aligned}
\lim_{\epsilon \to 0} \frac{1-b/z}{1-b z}\frac{1-b w}{1-b/w}\frac{(z-a)(1-w a)(zw-1)}{(z^2-1)(1-w^2)(z-w)}&=\frac{(Z+\alpha)(Z+\beta)(W-\alpha)(W-\beta)(Z+W)}{4Z W (Z-W)(W+\beta)(Z-\beta)},\\
\lim_{\epsilon \to 0} \frac{1-b/z}{1-b z}\frac{1-b w}{1-b/w}\frac{(z-a)(zw-1)}{(z^2-1)(w-a)(z-w)} \epsilon &=\frac{(Z+\alpha)(Z+\beta)(W-\beta)(Z+W)}{2Z (Z-W)(W+\alpha)(W+\beta)(Z-\beta)},\\
\lim_{\epsilon \to 0} \frac{1-b/z}{1-b z}\frac{1-b w}{1-b/w}\frac{zw-1}{(1-a z)(w-a)(z-w)} \epsilon^{2} &=\frac{-(Z+\beta)(W-\beta)(Z+W)}{(Z-W)(W+\alpha)(W+\beta)(Z-\alpha)(Z-\beta)},\\
\lim_{\epsilon \to 0} z^{k'-k} \frac{z^{-1}-z}{(1-a z)(z-a)} \epsilon &= e^{-Z(x-x')}\frac{2Z}{Z^2-\alpha^2}, \\
\lim_{\epsilon \to 0} \frac{z^{k'-k-1}} {(1-\sqrt{q} z)^{j-j'}} \epsilon^{j-j'} &= \frac{e^{-Z(x-x')}} {(\tfrac12 - Z)^{j-j'}}.
\end{aligned}
\end{equation}
The last limit is used for the $V$ part of $K^{\rm geo,\, \epsilon}$; note that $V^{\rm geo,\, \epsilon}$ also has an \emph{explicit} form in terms of Pochhammer symbols, factorials, and powers of $q$ as in~\eqref{eq:V_E_geo}; it is then routine to alternatively use Stirling's approximation to compute its $q \to 1-$ limit and conclude it equals the exponential $V$ kernel of~\eqref{eq:V}.

To finish, let us choose the integration paths for $Z$ and $W$ such that they are bounded away from $0$ and from any of the poles of the integrands for any small enough $\epsilon > 0$. It follows that the integrands, appropriately multiplied by some powers of $\epsilon$, are uniformly bounded. We can then take the $\epsilon \to 0+$ limit inside using dominated convergence. This finishes the proof.
\end{proof}

We next provide exponential decay for $K^{\rm geo,\, \epsilon}$, allowing us to conclude by dominated convergence that the discrete Fredholm Pfaffian converges to the continuous one. Since not all terms of $K^{\rm geo,\, \epsilon}$ are exponentially decaying, we need a conjugation. Let
\begin{equation} \label{eq:geo_bounds}
 \nu=(\beta-\max\{0,-\alpha\})/4, \quad \tilde \mu=\max\{0,-\alpha\}+\frac32 \nu,\quad \mu =\max\{0,-\alpha\}+2\nu
\end{equation}
which implies $0<\nu<\tilde\mu<\mu<\beta$ for all $\beta>0$ and $\alpha+\beta>0$.

For $\alpha \geq 0$, we have $\nu=\beta/4$, $\tilde\mu=3\beta/8$ and $\mu=\beta/2$. They satisfy
\begin{equation}\label{eqA1}
\mu+\nu=3\beta/4<\beta,\quad \tilde\mu-\nu = \beta/8>0\geq -\alpha.
\end{equation}
For $\alpha<0$, we have $\nu=(\alpha+\beta)/4$, $\tilde\mu=(3\beta-5\alpha)/8$ and $\mu=(\beta-\alpha)/4$. They satisfy, using $-\alpha<\beta$,
\begin{equation}\label{eqA2}
\mu+\nu=3\beta/4-\alpha/4<\beta,\quad \tilde\mu-\nu=\nu/2-\alpha>-\alpha.
\end{equation}
Consider $m$ positive real numbers $\mu_c$, $1 \leq c \leq m$, satisfying
\begin{equation} \label{eq:geo_bounds_2}
\tilde \mu \leq \mu_m < \cdots < \mu_1 \leq \mu.
\end{equation}
Note that we will also have $|\mu_c-\mu_d|\leq \nu/2$.

Let $K^{\rm geo,\, \epsilon}_{\rm conj}$ be the $2m \times 2m$ matrix kernel given by
\begin{equation}
 K^{\rm geo,\, \epsilon}_{\rm conj} = M(x) K^{\rm geo,\, \epsilon} M(x'),
\end{equation}
where $M(x)={\rm diag}(e^{\mu_1 x}, e^{-\mu_1 x}, \dots, e^{\mu_m x}, e^{-\mu_m x})$. What we mean by the above is that, precisely, the $2 \times 2$ block at position $(c, d)$ ($1 \leq c, d \leq m$) of $K^{\rm geo,\, \epsilon}_{\rm conj}$ is
\begin{equation}
 K^{\rm geo,\, \epsilon}_{\rm conj} (j_c, x; j_d, x') =
 \left(
 \begin{array}{cc}
 K^{\rm geo,\, \epsilon}_{11} (j_c, x; j_d, x') e^{\mu_c x} e^{\mu_d x'} & K^{\rm geo,\, \epsilon}_{12} (j_c, x; j_d, x') e^{\mu_c x} e^{-\mu_d x'} \\
 K^{\rm geo,\, \epsilon}_{21} (j_c, x; j_d, x') e^{-\mu_c x} e^{\mu_d x'} & K^{\rm geo,\, \epsilon}_{22} (j_c, x; j_d, x') e^{-\mu_c x} e^{-\mu_d x'}
 \end{array}
 \right).
\end{equation}
This conjugation does not change the value of the Fredholm Pfaffian, that is, $\pf (J - P_k K^{\rm geo,\, \epsilon} P_k) = \pf (J - P_k K^{\rm geo,\, \epsilon}_{\rm conj} P_k)$.

\begin{lem}\label{lem:Bounds}
Fix $\beta\in (0,1/2)$, $\alpha\in (-1/2,1/2)$ so that $\alpha+\beta>0$. Further let $\nu, \mu$ along with $\mu_1, \dots, \mu_m$ be as in~\eqref{eq:geo_bounds} and~\eqref{eq:geo_bounds_2}. For $K^{\rm geo,\, \epsilon}_{\rm conj}$ the following bounds hold uniformly for $x, x'$ in compact subsets of $\R_+$ and for any $1 \leq c, d \leq m$:
\begin{equation}
 \begin{aligned}
 |K^{\rm geo,\, \epsilon}_{\rm conj,\, 11} (j_c, x; j_d, x')|  \leq C e^{-\nu (x+x')}, &\quad  |K^{\rm geo,\, \epsilon}_{\rm conj,\, 12} (j_c, x; j_d, x')|  \leq C e^{-\nu (x+x')}, \\
 |K^{\rm geo,\, \epsilon}_{\rm conj,\, 21} (j_c, x; j_d, x')|  \leq C e^{-\nu (x+x')}, &\quad |K^{\rm geo,\, \epsilon}_{\rm conj,\, 22} (j_c, x; j_d, x')|  \leq C e^{-\nu (x+x')}.
 \end{aligned}
\end{equation}
\end{lem}

\begin{proof}
The $k,k'$ dependence of our kernels only appears in the term $w^{k'}/z^k$; the rest of the integrands remain nicely bounded and converge to their exponential analogues, see Proposition~\ref{prop:ConvGeom}.

For $K^{\rm geo,\, \epsilon}_{11}(j_c, x; j_d, x')$, let us choose contours satisfying $|z|\geq 1+(\mu_c+\nu)\epsilon $ and $|w|\leq 1-(\mu_d+\nu)\epsilon $. This is compatible with the contour requirement: (a) $w$ includes the poles at $b=1-\epsilon \beta$ and/or $\sqrt{q}\simeq 1-\epsilon/2$, since $-(\mu_d+\nu)\geq -(\mu+\nu)>-\beta$ by \eqref{eqA1}-\eqref{eqA2}, and (b) $z$ includes the poles at $1/\sqrt{q}\simeq 1+\epsilon/2$ and/or $1/b\simeq 1+\epsilon\beta$, since $\mu_c+\nu\leq \mu+\nu<\beta$ by \eqref{eqA1}-\eqref{eqA2}.

These yield: $|w^{k'}/z^k|\leq \frac{(1-(\mu_d + \nu)\epsilon )^{x'/\epsilon }}{(1+(\mu_c+\nu)\epsilon )^{x/\epsilon }}\simeq e^{-\mu_c x - \mu_d x'} e^{-\nu(x+x')}$ and the exponentially decaying bounds still persist after conjugation.

For the double contour integral in $K^{\rm geo,\, \epsilon}_{12}(j_c, x;j_d, x')$, we choose contours so that $|z| \geq 1+(\mu_c+\nu) \epsilon $ and $|w|\leq 1+(\mu_d-\nu)\epsilon$. The contour for $z$ need to include the poles at $1+\beta\epsilon$ and $1+\epsilon/2$, which is satisfied since $\mu_c+\nu<\beta$ by \eqref{eqA1}-\eqref{eqA2}. The contour for $w$ need to include the poles at $1-\epsilon/2$, $1-\epsilon\beta$ and $1-\epsilon\alpha$, which is satisfied since $\mu_d-\nu\geq \tilde\mu-\nu>\max\{0,-\alpha\}$ by \eqref{eqA1}-\eqref{eqA2}. We just need to have that the contours do not intersects, and this is satisfied by the condition $\mu_c+\nu>\mu_d-\nu$, equivalent to $\mu_d-\mu_c<2\nu$, which is clearly satisfied.

We then have: $|w^{k'}/z^k| \lesssim e^{-\mu_c x + \mu_d x'} e^{-\nu(x+x')}$ and conjugation still yields the first part of the exponentially decaying bound above. The kernel $V^{\rm geo,\, \epsilon}$ is explicit and gives the bound $e^{-\frac12(x-x')}e^{\mu_c x-\mu_d x'}\leq e^{-\frac{x-x'}{2}} e^{\mu_c x - \mu_d x'}$ for all $x\geq x'$. The term coming from $V$ is not decaying for $x-x'$ constant and thus the conjugation is here essential. Using the contour $|z|=1+(\mu_d-\nu)\epsilon$, for $V^{\rm geo,\, \epsilon}$ we get the bound $e^{-(x'-x)(\mu_d-\nu)}$, for $x>x'$ (otherwise it is simply $0$). After conjugation we have $e^{-(x'-x)(\mu_d-\nu)} e^{-\mu_c x-\mu_d x'}\leq e^{-\nu x'-\mu_c x}\leq e^{-\nu(x+x')}$.

The bounds for the $21$ entries $K^{\rm geo,\, \epsilon}_{\rm conj,\, 21} (j_c, x; j_d, x')$ follow from the anti-symmetry relation $K^{\rm geo,\, \epsilon}_{\rm conj,\, 21} (j_c, x; j_d, x') = - K^{\rm geo,\, \epsilon}_{\rm conj,\, 12} (j_d, x'; j_c, x)$.

Finally we turn to the $22$ entries. For the double integrand in $K^{\rm geo,\, \epsilon}_{22}(j_c, x; j_d, x')$ there are several terms. In the first and third one, we choose contours satisfying $|z| \geq 1-(\mu_c-\nu) \epsilon$ and $|w| \leq 1-(\mu_d+ \nu)\epsilon$, while in the second one $ |z| \geq 1+(\mu_c+\nu) \epsilon$ and $|w| \leq 1 + (\mu_d-\nu)\epsilon$. Again, the relations \eqref{eqA1}-\eqref{eqA2} imply that the conditions on the paths are satisfied. Combining the cases, we have $|w^{k'}/z^k| \lesssim e^{\mu_c x + \mu_d x'} e^{-\nu(x + x')}$. The exponentially decaying bound $e^{-\nu(x + x')}$ persists after conjugation. Finally there are the terms coming from the $E$ kernel. For $k>k'$ (and so $x > x'$) we take a contour with $|z| \geq 1-(\mu_c-\nu)\epsilon$, which gives $|z^{k'-k}| \lesssim e^{(\mu_c-\nu) (x - x')}$. After conjugation, this term is still decaying faster than $e^{-\nu x} e^{-\mu_d y}\leq e^{-\nu (x+y)}$. For $k<k'$ the bound follows from the anti-symmetry of the $E$ term.
\end{proof}

Finally the geometric Fredholm Pfaffian converges to the corresponding exponential one.

\begin{proof}[Proof of Theorem~\ref{thm:exp_corr}]
 The proof of Theorem~\ref{thm:exp_corr} is routine given Proposition~\ref{prop:ConvGeom} and Lemma~\ref{lem:Bounds}. We write the series expansion of the Fredholm Pfaffian and then use the Hadamard bound for determinants/Pfaffians to justify interchanging limit and summation.
\end{proof}

\section{On the Airy$_{\rm stat}$ and half-space Airy$_{\hs}$ processes}
\label{AppWellDef}

In Definition~\ref{def:Airy_half_stat_def} we defined a stochastic process by giving a formula for its finite-dimensional distributions. For the formula to actually define a stochastic process, it must have the following properties: when one of the $S_k\to -\infty$, the distribution goes to $0$; when all $S_k\to\infty$, the distribution goes to $1$; and that the distributions form a consistent family of distributions. All these properties follow if we show that the vector
\begin{equation}
\left(\frac{L_{N,j_k} - 4N + 4 u_k (2N)^{2/3}}{2^{4/3}N^{1/3}}, \quad k=1,\ldots,m\right)
\end{equation}
is tight, which in turns follows by showing that each component of the vector is tight. The consistency then follows by the fact that it is true for the finite-size formula by construction.

We already know that the distribution function converges to a limit, but we still need to verify that the limiting formula is a distribution function. Let us fix $k=1$, the result holding true for each $k$. We have
\begin{equation}
\begin{aligned}
F(S_1)&:=\lim_{N\to\infty} \Pb \left(\frac{L_{N,j_1} - 4N + 4 u_1 (2N)^{2/3}}{2^{4/3}N^{1/3}} \leq S_1 \right) \\
&= \partial_{S_1} \left\{ \pf (J - \breve{\mathcal{A}}_{S_1}) \left[ \mathpzc{e}^{\delta,\, u_1} (S_1) - \braket{P_{S_1} \mathcal{Y}} {(\Id-J^{-1} \breve{\mathcal{A}}_{S_1} )^{-1} P_{S_1} \mathcal{Q}} \right] \right\}.
\end{aligned}
\end{equation}

Using the formula
\begin{equation}
\pf(J-K)\braket{g}{(\Id-J^{-1}K)^{-1} h} = \pf(J-K)-\pf\big[J-K-J \ketbra{h}{g}-\ketbra{g}{h}J\big],
\end{equation}
we get
\begin{equation}
\begin{aligned}
F(S_1)=& \pf (J - \breve{\mathcal{A}}_{S_1}) \partial_{S_1}\mathpzc{e}^{\delta,\, u_1} (S_1) +(\mathpzc{e}^{\delta,\, u_1} (S_1)-1) \partial_{S_1} \pf (J - \breve{\mathcal{A}}_{s_1})\\
& +\partial_{S_1} \pf \big[J - \breve{\mathcal{A}}_{S_1}-P_{S_1}J\ketbra{\mathcal{Q}}{\mathcal{Y}}P_{S_1}- P_{S_1}\ketbra{\mathcal{Y}}{\mathcal{Q}} J P_{S_1}\big].
\end{aligned}
\end{equation}
For the term including $\mathpzc{e}^{\delta,\, u_1} (S_1)$ and its derivative, computing the residue at $\zeta=\delta$ explicitly from \eqref{eq2.26} we get that
\begin{equation}\label{eqE1}
\mathpzc{e}^{\delta,\, u_1} (S_1)=S_1-\delta(2u_1+\delta)-\int\limits_{{}_\delta\zcd} \frac{d \zeta}{2\pi\I} \frac{e^{\frac{\zeta^3}{3} + \zeta^2 u_1 - \zeta S_1} }{e^{\frac{\delta^3}{3} + \delta^2 u_1 - \delta S_1}} \frac{1}{(\zeta - \delta)^2}
\end{equation}
as well as
\begin{equation}\label{eqE2}
\partial_{S_1}\mathpzc{e}^{\delta,\, u_1} (S_1)=1+\int\limits_{{}_\delta\zcd} \frac{d \zeta}{2\pi\I} \frac{e^{\frac{\zeta^3}{3} + \zeta^2 u_1 - \zeta S_1} }{e^{\frac{\delta^3}{3} + \delta^2 u_1 - \delta S_1}} \frac{1}{\zeta - \delta}.
\end{equation}
The integrals in \eqref{eqE1} and \eqref{eqE2} have superexponential decay in $S_1$ as $S_1\to\infty$.

In Section~\ref{sec:proof_asymptotics} we have already obtained asymptotics and bounds on the kernels and vectors. Those results were uniform in $N$ and thus hold also for the limiting kernel and vectors. Using this we see that
\begin{equation}
\begin{aligned}
&\pf (J - \breve{\mathcal{A}}_{S_1}) \to 1, \quad \partial_{S_1}\pf (J - \breve{\mathcal{A}}_{S_1})\to 0, \\
&\partial_{S_1} \pf \big[J - \breve{\mathcal{A}}_{S_1}-P_{S_1}J\ketbra{\mathcal{Q}}{\mathcal{Y}}P_{S_1}- P_{S_1}\ketbra{\mathcal{Y}}{\mathcal{Q}} J P_{S_1}\big] \to 0,
\end{aligned}
\end{equation}
as $S_1\to\infty$. 

The above results imply that $\lim_{S_1\to\infty} F(S_1)=1$.

Let us remark that everything is actually truly uniform in $N$: that is, it holds for the scaled random variables and not only for the limiting formula.

Next we need to verify that $\lim_{S_1\to -\infty} F(S_1)=0$. For this purpose, define
\begin{equation}
X_N=\frac{L_{N,j_1}-L_{N,N} + 4 u_1 (2N)^{2/3}}{2^{4/3}N^{1/3}},\quad Y_N=\frac{L_{N,N}-4N}{2^{4/3}N^{1/3}}.
\end{equation}
Then we have
\begin{equation}
F(S_1)\leq \lim_{N\to\infty} \Pb(X_N\leq \tfrac12 S_1)+\Pb(Y_N\leq \tfrac12 S_1).
\end{equation}
By Lemma~\ref{lem:stationarity} we have that
\begin{equation}
-X_N = \sum_{k=1}^{N-j_1} \xi_k
\end{equation}
with $\xi_1,\xi_2,\ldots$ i.i.d.~${\rm Exp}(\tfrac12+\alpha)$ random variables. By the standard exponential Chebyshev inequality we obtain that for large $N$, 
\begin{equation*}
	\Pb(X_N\leq \tfrac12 S_1)\leq C e^{-(S_1+2\delta)^2/(4 u_1)}\to 0, \quad \text{as } S_1\to -\infty.
\end{equation*}

For the term $\Pb(Y_N\leq \tfrac12 S_1)$, notice that we can couple the stationary LPP with the LPP where the parameters at the two boundaries are both ${\rm Exp}(1)$, that is, we take the integrable model in \eqref{eq:int_wts} with $\alpha=\beta=\tfrac12$. With this choice of $\alpha$ and $\beta$ we have, in distribution, $L^{\rm pf}_{N,N}-\tilde\omega_{1,1}\stackrel{(d)}\leq L_{N,N}$. Thus
\begin{equation}\label{eqE3}
\lim_{N\to\infty}\Pb(Y_N\leq\tfrac12 S_1)\leq \lim_{N\to\infty} \Pb\left(\frac{L^{\rm pf}_{N,N}-\tilde\omega_{1,1}}{2^{4/3} N^{1/3}}\leq \tfrac12 S_1\right).
\end{equation}
The random variable $\tilde \omega_{1,1}$ is asymptotically irrelevant and from Theorem~1.3 of~\cite{BBCS17}, case (1), we get that the right-hand side of~\eqref{eqE3} converges to the GSE Tracy--Widom distribution~\cite{TW96}. Therefore
\begin{equation}
\lim_{N\to\infty}\Pb(Y_N\leq\tfrac12 S_1)\leq F_{\rm GSE}(\tfrac12 S_1)\to 0
\end{equation}
as $S_1\to -\infty$ (see e.g.~\cite{BR00} after Definition 1).

Finally, let us remark that a similar argument can be made for the Airy$_{\rm stat}$ process of Definition~\ref{defAiryStat}. It was shown in~\cite{BFP09} that for the full-space case, the joint distribution of the rescaled LPP converges to the right-hand side of~\eqref{eqAiryStat}, but it has not been verified (in op.~cit.) that no mass escapes at infinity. However, we know from the bounds obtained in~\cite{BFP12} that indeed it is a probability distribution function, so that the Airy$_{\rm stat}$ process is likewise well-defined.


\begin{thebibliography}{10}

	\bibitem{Agg16}
	A.~Aggarwal, \emph{{Current Fluctuations of the Stationary ASEP and Six-Vertex
	  Model}}, Duke Math. J. \textbf{167} (2018), 269--384.
	
	\bibitem{AB19}
	A.~Aggarwal and A.~Borodin, \emph{{Phase transitions in the ASEP and stochastic
	  six-vertex model}}, Ann. Probab. \textbf{47} (2019), 613--689.
	
	\bibitem{ACQ10}
	G.~Amir, I.~Corwin, and J.~Quastel, \emph{{Probability distribution of the free
	  energy of the continuum directed random polymer in 1+1 dimensions}}, Comm.
	  Pure Appl. Math. \textbf{64} (2011), 466--537.
	
	\bibitem{BBCS17b}
	J.~Baik, G.~Barraquand, I.~Corwin, and T.~Suidan, \emph{Facilitated exclusion
	  process}, Computation and Combinatorics in Dynamics, Stochastics and Control
	  (E.~Celledoni, G.~Di Nunno, K.~Ebrahimi-Fard, and H.Z. Munthe-Kaas, eds.),
	  Springer, 2018.
	
	\bibitem{BBCS17}
	J.~Baik, G.~Barraquand, I.~Corwin, and T.~Suidan, \emph{Pfaffian {S}chur
	  processes and last passage percolation in a half-quadrant}, Ann. Probab.
	  \textbf{46} (2018), 3015--3089.
	
	\bibitem{BDJ99}
	J.~Baik, P.A. Deift, and K.~Johansson, \emph{On the distribution of the length
	  of the longest increasing subsequence of random permutations}, J. Amer. Math.
	  Soc. \textbf{12} (1999), 1119--1178.
	
	\bibitem{BFP09}
	J.~Baik, P.L. Ferrari, and S.~P{\'e}ch{\'e}, \emph{{Limit process of stationary
	  TASEP near the characteristic line}}, Comm. Pure Appl. Math. \textbf{63}
	  (2010), 1017--1070.
	
	\bibitem{BFP12}
	J.~Baik, P.L. Ferrari, and S.~P{\'e}ch{\'e}, \emph{{Convergence of the
	  two-point function of the stationary TASEP}}, {Singular Phenomena and Scaling
	  in Mathematical Models}, Springer, 2014, pp.~91--110.
	
	\bibitem{BL16}
	J.~Baik and Z.~Liu, \emph{{TASEP on a ring in sub-relaxation time scale}}, J.
	  Stat. Phys. \textbf{165} (2016), 1051--1085.
	
	\bibitem{BL17}
	J.~Baik and Z.~Liu, \emph{{Multi-point distribution of periodic TASEP}}, J.
	  Amer. Math. Soc. \textbf{32} (2019), 609--674.
	
	\bibitem{BR00}
	J.~Baik and E.M. Rains, \emph{Limiting distributions for a polynuclear growth
	  model with external sources}, J. Stat. Phys. \textbf{100} (2000), 523--542.
	
	\bibitem{BR01b}
	J.~Baik and E.M. Rains, \emph{Algebraic aspects of increasing subsequences},
	  Duke Math. J. \textbf{109} (2001), 1--65.
	
	\bibitem{BR99b}
	J.~Baik and E.M. Rains, \emph{The asymptotics of monotone subsequences of
	  involutions}, Duke Math. J. \textbf{109} (2001), 205--281.
	
	\bibitem{BR99}
	J.~Baik and E.M. Rains, \emph{Symmetrized random permutations}, Random Matrix
	  Models and Their Applications, vol.~40, Cambridge University Press, 2001,
	  pp.~1--19.
	
	\bibitem{BCS06}
	M.~Bal{\'a}zs, E.~Cator, and T.~Sepp{\"a}l{\"a}inen, \emph{Cube root
	  fluctuations for the corner growth model associated to the exclusion
	  process}, Electron. J. Probab. \textbf{11} (2006), 1094--1132.
	
	\bibitem{Bar14}
	G.~Barraquand, \emph{{A phase transition for $q$-TASEP with a few slower
	  particles}}, Stoch. Proc. Appl. \textbf{125} (2015), 2674--2699.
	
	\bibitem{BBC18}
	G.~Barraquand, A.~Borodin, and I.~Corwin, \emph{{Half-space Macdonald
	  processes}}, Forum of Mathematics, Pi \textbf{8} (2020).
	
	\bibitem{BBCW17}
	G.~Barraquand, A.~Borodin, I.~Corwin, and M.~Wheeler, \emph{{Stochastic
	  six-vertex model in a half-quadrant and half-line open ASEP}}, Duke Math. J.
	  \textbf{167} (2018), 2457--2529.
	
	\bibitem{BKLD20}
	G.~Barraquand, A.~Krajenbrink, and P.~Le~Doussal, \emph{{Half-Space Stationary
	  Kardar--Parisi--Zhang Equation}}, J. Stat. Phys. \textbf{181} (2020), no.~4,
	  1149--1203.
	
	\bibitem{BKS85}
	{H. van} Beijeren, R.~Kutner, and H.~Spohn, \emph{Excess noise for driven
	  diffusive systems}, Phys. Rev. Lett. \textbf{54} (1985), 2026--2029.
	
	\bibitem{BBNV18}
	D.~Betea, J.~Bouttier, P.~Nejjar, and M.~Vuleti\'c, \emph{The free boundary
	  {S}chur process and applications {I}}, Ann. Henri Poincar\'e \textbf{19}
	  (2018), 3663--3742.
	
	\bibitem{BFO20}
	D.~Betea, P.L. Ferrari, and A.~Occelli, \emph{Stationary half-space last
	  passage percolation}, Comm. Math. Phys. \textbf{377} (2020), 421--467.
	
	\bibitem{BZ17}
	E.~Bisi and N.~Zygouras, \emph{{Point-to-line polymers and orthogonal Whittaker
	  functions}}, Trans. Amer. Math. Soc. \textbf{371} (2019), 8339--8379.
	
	\bibitem{BZ19}
	E.~Bisi and N.~Zygouras, \emph{{Transition between characters of classical
	  groups, decomposition of Gelfand-Tsetlin patterns and last passage
	  percolation}}, arXiv:1905.09756 (2019).
	
	\bibitem{Bor10}
	A.~Borodin, \emph{Determinantal point processes}, The Oxford Handbook of Random
	  Matrix Theory, Chapter 11 (G.~Akemann, J.~Baik, and P.~Di Francesco, eds.),
	  Oxford University Press, USA, 2010.
	
	\bibitem{BCF12}
	A.~Borodin, I.~Corwin, and P.L. Ferrari, \emph{Free energy fluctuations for
	  directed polymers in random media in $1+1$ dimension}, Comm. Pure Appl. Math.
	  \textbf{67} (2014), 1129--1214.
	
	\bibitem{BCFV14}
	A.~Borodin, I.~Corwin, P.L. Ferrari, and B.~Vet{\H o}, \emph{{Height
	  fluctuations for the stationary KPZ equation}}, Mathematical Physics,
	  Analysis and Geometry \textbf{December 2015} (2015).
	
	\bibitem{BF07}
	A.~Borodin and P.L. Ferrari, \emph{{Large time asymptotics of growth models on
	  space-like paths I: PushASEP}}, Electron. J. Probab. \textbf{13} (2008),
	  1380--1418.
	
	\bibitem{BFP06}
	A.~Borodin, P.L. Ferrari, and M.~Pr{\"a}hofer, \emph{{Fluctuations in the
	  discrete TASEP with periodic initial configurations and the Airy$_1$
	  process}}, Int. Math. Res. Papers \textbf{2007} (2007), rpm002.
	
	\bibitem{BFPS06}
	A.~Borodin, P.L. Ferrari, M.~Pr{\"a}hofer, and T.~Sasamoto, \emph{{Fluctuation
	  properties of the TASEP with periodic initial configuration}}, J. Stat. Phys.
	  \textbf{129} (2007), 1055--1080.
	
	\bibitem{BG12}
	A.~Borodin and V.~Gorin, \emph{Lectures on integrable probability},
	  arXiv:1212.3351 (2012).
	
	\bibitem{BOO00}
	A.~Borodin, A.~Okounkov, and G.~Olshanski, \emph{Asymptotics of {P}lancherel
	  measures for symmetric groups}, J. Amer. Math. Soc. \textbf{13} (2000),
	  481--515.
	
	\bibitem{RB04}
	A.~Borodin and E.M. Rains, \emph{{Eynard-Mehta theorem, Schur process, and
	  their Pfaffian analogs}}, J. Stat. Phys. \textbf{121} (2006), 291--317.
	
	\bibitem{BF20}
	A.~Bufetov and P.L. Ferrari, \emph{{Shock fluctuations in TASEP under a variety
	  of time scalings}}, arXiv:2003.12414 (2019).
	
	\bibitem{Bur56}
	P.J. Burke, \emph{The output of a queuing system}, Operations Res. \textbf{4}
	  (1956), 699--704.
	
	\bibitem{CFS16}
	S.~Chhita, P.L. Ferrari, and H.~Spohn, \emph{{Limit distributions for KPZ
	  growth models with spatially homogeneous random initial conditions}}, Ann.
	  Appl. Probab. \textbf{28} (2018), 1573--1603.
	
	\bibitem{Cor11}
	I.~Corwin, \emph{{The Kardar-Parisi-Zhang equation and universality class}},
	  Random Matrices: Theory Appl. \textbf{01} (2012), 1130001.
	
	\bibitem{CLW16}
	I.~Corwin, Z.~Liu, and D.~Wang, \emph{{Fluctuations of TASEP and LPP with
	  general initial data}}, Ann. Appl. Probab. \textbf{26} (2016), 2030--2082.
	
	\bibitem{DMO05}
	M.~Darief, J.~Mairesse, and N.~O'Connell, \emph{Queues, stores, and tableaux},
	  J. Appl. Probab. \textbf{4} (2005), 1145--1167.
	
	\bibitem{Dim20}
	E.~Dimitrov, \emph{{Two-point convergence of the stochastic six-vertex model to
	  the Airy process}}, arXiv:2006.15934 (2020).
	
	\bibitem{FF94b}
	P.A. Ferrari and L.~Fontes, \emph{{Shock fluctuations in the asymmetric simple
	  exclusion process}}, Probab. Theory Relat. Fields \textbf{99} (1994),
	  305--319.
	
	\bibitem{Fer10b}
	P.L. Ferrari, \emph{{From interacting particle systems to random matrices}}, J.
	  Stat. Mech. (2010), P10016.
	
	\bibitem{FGN17}
	P.L. Ferrari, P.~Ghosal, and P.~Nejjar, \emph{{Limit law of a second class
	  particle in TASEP with non-random initial condition}}, Ann. Inst. Henri
	  Poincaré Probab. Statist. \textbf{55} (2019), 1203--1225.
	
	\bibitem{FN13}
	P.L. Ferrari and P.~Nejjar, \emph{{Anomalous shock fluctuations in TASEP and
	  last passage percolation models}}, Probab. Theory Relat. Fields \textbf{161}
	  (2015), 61--109.
	
	\bibitem{FN16}
	P.L. Ferrari and P.~Nejjar, \emph{{Fluctuations of the competition interface in
	  presence of shocks}}, ALEA, Lat. Am. J. Probab. Math. Stat. \textbf{14}
	  (2017), 299--325.
	
	\bibitem{FN19}
	P.L. Ferrari and P.~Nejjar, \emph{{Statistics of TASEP with three merging
	  characteristics}}, J. Stat. Phys. \textbf{180} (2019), 398--413.
	
	\bibitem{FO17}
	P.L. Ferrari and A.~Occelli, \emph{{Universality of the GOE Tracy-Widom
	  distribution for TASEP with arbitrary particle density}}, Eletron. J. Probab.
	  \textbf{23} (2018), no.~51, 1--24.
	
	\bibitem{FO18}
	P.L. Ferrari and A.~Occelli, \emph{Time-time covariance for last passage
	  percolation with generic initial profile}, Math. Phys. Anal. Geom.
	  \textbf{22} (2019), 1.
	
	\bibitem{FS05a}
	P.L. Ferrari and H.~Spohn, \emph{Scaling limit for the space-time covariance of
	  the stationary totally asymmetric simple exclusion process}, Comm. Math.
	  Phys. \textbf{265} (2006), 1--44.
	
	\bibitem{FS10}
	P.L. Ferrari and H.~Spohn, \emph{{Random Growth Models}}, The Oxford handbook
	  of random matrix theory (G.~Akemann, J.~Baik, and P.~{Di Francesco}, eds.),
	  Oxford Univ. Press, Oxford, 2011, pp.~782--801.
	
	\bibitem{FS16}
	P.L. Ferrari and H.~Spohn, \emph{{On time correlations for KPZ growth in one
	  dimension}}, SIGMA \textbf{12} (2016), 074.
	
	\bibitem{FSW15}
	P.L. Ferrari, H.~Spohn, and T.~Weiss, \emph{{Brownian motions with one-sided
	  collisions: the stationary case}}, Electron. J. Probab. \textbf{20} (2015),
	  1--41.
	
	\bibitem{FV20}
	P.L. Ferrari and B.~Vet\H{o}, \emph{{Upper tail decay of KPZ models with
	  Brownian initial conditions}}, arXiv:2007.13496 (2020).
	
	\bibitem{FV13}
	P.L. Ferrari and B.~Vet{\H o}, \emph{{Tracy-Widom asymptotics for q-TASEP}},
	  Ann. Inst. H. Poincar\'e, Probab. Statist. \textbf{51} (2015), 1465--1485.
	
	\bibitem{FR07}
	P.J. Forrester and E.M. Rains, \emph{Symmetrized models of last passage
	  percolation and non-intersecting lattice paths}, J. Stat. Phys. \textbf{129}
	  (2007), 833--855.
	
	\bibitem{FNS77}
	D.~Forster, D.R. Nelson, and M.J. Stephen, \emph{Large-distance and long-time
	  properties of a randomly stirred fluid}, Phys. Rev. A \textbf{16} (1977),
	  732--749.
	
	\bibitem{Gho17}
	P.~Ghosal, \emph{Correlation functions of the {P}faffian {S}chur process using
	  {M}acdonald difference operators}, SIGMA \textbf{15} (2019), 092.
	
	\bibitem{Gro04}
	S.~Grosskinsky, \emph{Phase transitions in nonequilibrium stochastic particle
	  systemswith local conservation laws}, Ph.D. thesis, Technische
	  Universit{\"a}t M{\"u}nchen, https://mediatum.ub.tum.de/602023, 2004.
	
	\bibitem{IMS19}
	T.~Imamura, M.~Mucciconi, and T.~Sasamoto, \emph{{Stationary Stochastic Higher
	  Spin Six Vertex Model and $q$-Whittaker measure}}, Probab. Theory Relat.
	  Fields \textbf{177} (2020), 923--1042.
	
	\bibitem{SI04}
	T.~Imamura and T.~Sasamoto, \emph{Fluctuations of the one-dimensional
	  polynuclear growth model with external sources}, Nucl. Phys. B \textbf{699}
	  (2004), 503--544.
	
	\bibitem{IS12}
	T.~Imamura and T.~Sasamoto, \emph{{Stationary correlations for the 1D KPZ
	  equation}}, J. Stat. Phys. \textbf{150} (2013), 908--939.
	
	\bibitem{SI17b}
	T.~Imamura and T.~Sasamoto, \emph{{Free energy distribution of the stationary
	  O'Connell–Yor directed random polymer model}}, J. Phys. A: Math. Theor.
	  \textbf{50} (2017), 285203.
	
	\bibitem{SI17}
	T.~Imamura and T.~Sasamoto, \emph{{Fluctuations for stationary $q$-TASEP}},
	  Probab. Theory Relat. Fields \textbf{174} (2019), 647--730.
	
	\bibitem{Jo00b}
	K.~Johansson, \emph{Shape fluctuations and random matrices}, Comm. Math. Phys.
	  \textbf{209} (2000), 437--476.
	
	\bibitem{Jo03b}
	K.~Johansson, \emph{Discrete polynuclear growth and determinantal processes},
	  Comm. Math. Phys. \textbf{242} (2003), 277--329.
	
	\bibitem{Jo05}
	K.~Johansson, \emph{Random matrices and determinantal processes}, Mathematical
	  Statistical Physics, Session LXXXIII: Lecture Notes of the Les Houches Summer
	  School 2005 (A.~Bovier, F.~Dunlop, A.~van Enter, F.~den Hollander, and
	  J.~Dalibard, eds.), Elsevier Science, 2006, pp.~1--56.
	
	\bibitem{Jo15}
	K.~Johansson, \emph{{Two time distribution in Brownian directed percolation}},
	  Comm. Math. Phys. \textbf{351} (2017), 441--492.
	
	\bibitem{Jo18}
	K.~Johansson, \emph{{The two-time distribution in geometric last-passage
	  percolation}}, arXiv:1802.00729 (2018).
	
	\bibitem{JR19}
	K.~Johansson and M.~Rahman, \emph{{Multi-time distribution in discrete
	  polynuclear growth}}, arXiv:1906.01053 (2019).
	
	\bibitem{KPZ86}
	M.~Kardar, G.~Parisi, and Y.Z. Zhang, \emph{Dynamic scaling of growing
	  interfaces}, Phys. Rev. Lett. \textbf{56} (1986), 889--892.
	
	\bibitem{KPS19}
	A.~Knizel, L.~Petrov, and A.~Saenz, \emph{{Generalizations of TASEP in discrete
	  and continuous inhomogeneous space}}, Commun. Math. Phys. \textbf{372}
	  (2019), 797--864.
	
	\bibitem{KLD19}
	A.~Krajenbrink and P.~Le Doussal, \emph{{Replica Bethe Ansatz solution to the
	  Kardar-Parisi-Zhang equation on the half-line}}, SciPost Phys. \textbf{8}
	  (2020), 035.
	
	\bibitem{KS92}
	J.~Krug and H.~Spohn, \emph{Kinetic roughening of growning surfaces}, Solids
	  far from equilibrium: growth, morphology and defects, Cambridge University
	  Press, 1992, pp.~479--582.
	
	\bibitem{Lig75}
	T.M. Liggett, \emph{{Ergodic theorems for the asymmetric simple exclusion
	  process}}, Trans. Amer. Math. Soc. \textbf{213} (1975), 237--261.
	
	\bibitem{Lig77}
	T.M. Liggett, \emph{{Ergodic theorems for the asymmetric simple exclusion
	  process II}}, Ann. Probab. \textbf{4} (1977), 795--801.
	
	\bibitem{Li99}
	T.M. Liggett, \emph{Stochastic interacting systems: contact, voter and
	  exclusion processes}, Springer Verlag, Berlin, 1999.
	
	\bibitem{Liu19}
	Z.~Liu, \emph{{Multi-time distribution of TASEP}}, arXiv:1907.09876 (2019).
	
	\bibitem{Mea98}
	P.~Meakin, \emph{Fractals, scaling and growth far from equilibrium}, Cambridge
	  {U}niversity {P}ress, Cambridge, 1998.
	
	\bibitem{ND17}
	J.~De Nardis and P.~Le Doussal, \emph{{Tail of the two-time height distribution
	  for KPZ growth in one dimension}}, J. Stat. Mech. \textbf{053212} (2017).
	
	\bibitem{ND18}
	J.~De Nardis and P.~Le Doussal, \emph{{Two-time height distribution for 1D KPZ
	  growth: the recent exact result and its tail via replica}}, J. Stat. Mech.
	  \textbf{093203} (2018).
	
	\bibitem{NDT17}
	J.~De Nardis, P.~Le Doussal, and K.A. Takeuchi, \emph{{Memory and Universality
	  in Interface Growth}}, Phys. Rev. Lett. \textbf{118} (2017), 125701.
	
	\bibitem{NKDT20}
	J.~De Nardis, A.~Krajenbrink, P.~Le Doussal, and T.~Thiery, \emph{Delta-{B}ose
	  gas on a half-line and the {K}ardar--{P}arisi--{Z}hang equation: boundary
	  bound states and unbinding transitions}, J. Stat. Mech. \textbf{043207}
	  (2020).
	
	\bibitem{Nej17}
	P.~Nejjar, \emph{{Transition to shocks in TASEP and decoupling of last passage
	  times}}, ALEA, Lat. Am. J. Probab. Math. Stat. \textbf{15} (2018),
	  1311--1334.
	
	\bibitem{OQR16}
	J.~Ortmann, J.~Quastel, and D.~Remenik, \emph{{Exact formulas for random growth
	  with half-flat initial data}}, Ann. Appl. Probab. \textbf{26} (2016),
	  507--548.
	
	\bibitem{PS00}
	M.~Pr{\"a}hofer and H.~Spohn, \emph{Universal distributions for growth
	  processes in 1+1 dimensions and random matrices}, Phys. Rev. Lett.
	  \textbf{84} (2000), 4882--4885.
	
	\bibitem{PS02}
	M.~Pr{\"a}hofer and H.~Spohn, \emph{Scale invariance of the {PNG} droplet and
	  the {A}iry process}, J. Stat. Phys. \textbf{108} (2002), 1071--1106.
	
	\bibitem{PS02b}
	M.~Pr{\"a}hofer and H.~Spohn, \emph{Exact scaling function for one-dimensional
	  stationary {KPZ} growth}, J. Stat. Phys. \textbf{115} (2004), 255--279.
	
	\bibitem{Qua11}
	J.~Quastel, \emph{{Introduction to KPZ}}, Current Developments in Mathematics
	  (2011), 125--194.
	
	\bibitem{QR16}
	J.~Quastel and D.~Remenik, \emph{{How flat is flat in a random interface
	  growth?}}, Trans. Amer. Math. Soc. \textbf{371} (2019), 6047--6085.
	
	\bibitem{QS15}
	J.~Quastel and H.~Spohn, \emph{The one-dimensional kpz equation and its
	  universality class}, J. Stat. Phys. \textbf{160} (2015), 965--984.
	
	\bibitem{Ra00}
	E.M. Rains, \emph{Correlation functions for symmetrized increasing
	  subsequences}, arXiv:math.CO/0006097 (2000).
	
	\bibitem{Sas05}
	T.~Sasamoto, \emph{Spatial correlations of the {1D KPZ} surface on a flat
	  substrate}, J. Phys. A \textbf{38} (2005), L549--L556.
	
	\bibitem{SI03}
	T.~Sasamoto and T.~Imamura, \emph{Fluctuations of a one-dimensional polynuclear
	  growth model in a half space}, J. Stat. Phys. \textbf{115} (2004), 749--803.
	
	\bibitem{SS10b}
	T.~Sasamoto and H.~Spohn, \emph{{Exact height distributions for the KPZ
	  equation with narrow wedge initial condition}}, Nucl. Phys. B \textbf{834}
	  (2010), 523--542.
	
	\bibitem{SS10}
	T.~Sasamoto and H.~Spohn, \emph{{One-dimensional Kardar-Parisi-Zhang equation:
	  an exact solution and its universality}}, Phys. Rev. Lett. \textbf{104}
	  (2010), 230602.
	
	\bibitem{Ste90}
	J.R. Stembridge, \emph{Nonintersecting paths, {P}faffians, and plane
	  partitions}, Adv. Math. \textbf{83} (1990), 96--131.
	
	\bibitem{Tak13}
	K.A. Takeuchi, \emph{{Crossover from Growing to Stationary Interfaces in the
	  Kardar-Parisi-Zhang Class}}, Phys. Rev. Lett. \textbf{110} (2013), 210604.
	
	\bibitem{Tak16}
	K.A. Takeuchi, \emph{{An appetizer to modern developments on the
	  Kardar--Parisi--Zhang universality class}}, Physica A \textbf{504} (2016),
	  77--105.
	
	\bibitem{TW94}
	C.A. Tracy and H.~Widom, \emph{{Level-spacing distributions and the Airy
	  kernel}}, Comm. Math. Phys. \textbf{159} (1994), 151--174.
	
	\bibitem{TW96}
	C.A. Tracy and H.~Widom, \emph{On orthogonal and symplectic matrix ensembles},
	  Comm. Math. Phys. \textbf{177} (1996), 727--754.
	
	\bibitem{Zyg18}
	N.~Zygouras, \emph{{Some algebraic structures in the KPZ universality}},
	  arXiv:1812.07204 (2018).
	
\end{thebibliography}

\end{document}